\documentclass{amsart}
\usepackage{url}
\usepackage{latexsym}
\usepackage{lscape} 
\usepackage{geometry}
\usepackage{pdflscape}
\usepackage{float}
\usepackage{amssymb}
\usepackage{amscd}
\usepackage{graphicx}
\usepackage{tikz, tikz-cd}
\usepackage[all]{xy}
\usepackage{subfiles} 
\usepackage[extrasp=0em,mono]{inconsolata}  
\newcommand{\bighat}{\raisebox{.8ex}{\resizebox{1.3ex}{!}{$\boldsymbol{\land}$\,}}}

\newtheorem{thm}[equation]{Theorem}

\newtheorem{ex}[equation]{Example}
\newtheorem{lem}[equation]{Lemma}
\newtheorem{cor}[equation]{Corollary}

\newtheorem{prop}[equation]{Proposition}

\theoremstyle{remark}
\newtheorem{rem}[equation]{Remark}

\theoremstyle{definition}
\newtheorem{Def}[equation]{Definition}

\numberwithin{equation}{section}

\newcommand{\C}{\text{\bf C}}           

\newcommand{\End}{\operatorname{End}}

\newcommand{\Ad}{\operatorname{Ad}}

\newcommand{\pr}{\operatorname{pr}}
\newcommand{\Ev}{\operatorname{Ev}}

\newcommand{\vspan}{\operatorname{span}}

\newcommand{\fa}{{\mathfrak a}}             

\newcommand{\fg}{{\mathfrak g}}

\newcommand{\fk}{{\mathfrak k}}

\newcommand{\fp}{{\mathfrak p}}

\newcommand{\ft}{{\mathfrak t}}

\newcommand{\half}{\frac{1}{2}}

\newcommand{\Cal}{\mathcal}
\newcommand{\cal}{\mathcal}

\newcommand{\be}{\begin{equation}}
\newcommand{\beu}{\begin{equation*}}

\newcommand{\diag}{\operatorname{diag}}
\newcommand{\bbar}{\,\big|\,}
\newcommand{\sbar}{\,|\,}

\usepackage{graphicx,color,cancel}

\newcommand\eps{{\varepsilon}}
\newcommand\la{{\lambda}}

\newcommand\twedge{\textstyle{\bigwedge}}

\newcommand\Sp{\operatorname{Sp}}

\newcommand\SL{\operatorname{SL}}
\newcommand\GL{\operatorname{GL}}
\newcommand\SO{\operatorname{SO}}
\newcommand\OO{\operatorname{O}}
\newcommand\UU{\operatorname{U}}
\newcommand\SU{\operatorname{SU}}

\newcommand\Mat{\operatorname{Mat}}
\newcommand\Sym{\operatorname{Sym}}
\newcommand\Alt{\operatorname{Alt}}
\newcommand\Her{\operatorname{Her}}

\newcommand\st{\tilde s}

\newcommand{\Kt}{\widetilde{K}}

\newcommand\gr{\operatorname{gr}}
\newcommand\Gr{\operatorname{Gr}}
\newcommand\im{\operatorname{im}}
\newcommand\Ker{\operatorname{ker}}

\newcommand\id{\operatorname{id}}
\newcommand\Id{\operatorname{id}}
\newcommand\Spin{\operatorname{Spin}}
\newcommand\Pin{\operatorname{Pin}}

\newcommand{\lan}{\langle}
\newcommand{\ran}{\rangle}
\newcommand{\lara}{\langle\,,\rangle}

\newcommand{\pf}{\begin{proof}}
\newcommand{\epf}{\end{proof}}
\newcommand{\eq}{\begin{equation}}
\newcommand{\eeq}{\end{equation}}
\newcommand{\eqn}{\begin{equation*}}
\newcommand{\eeqn}{\end{equation*}}

\newcommand{\fra}{\mathfrak{a}}

\newcommand{\frg}{\mathfrak{g}}
\newcommand{\frh}{\mathfrak{h}}
\newcommand{\frk}{\mathfrak{k}}

\newcommand{\frp}{\mathfrak{p}}

\newcommand{\frt}{\mathfrak{t}}

\newcommand{\frsl}{\mathfrak{sl}}

\newcommand{\frso}{\mathfrak{so}}
\newcommand{\frsp}{\mathfrak{sp}}

\newcommand{\bbC}{\mathbb{C}}
\newcommand{\bbG}{\mathbb{G}}
\newcommand{\bbH}{\mathbb{H}}
\newcommand{\bbN}{\mathbb{N}}

\newcommand{\bbR}{\mathbb{R}}

\newcommand{\bbZ}{\mathbb{Z}}

\newcommand{\tr}{\operatorname{tr}}

\newcommand{\even}{\operatorname{even}}
\newcommand{\odd}{\operatorname{odd}}
\newcommand{\Proj}{\operatorname{Pr}}
\newcommand{\proj}{\operatorname{Pr}}
\newcommand{\LGr}{\operatorname{LGr}}
\newcommand{\HLGr}{\operatorname{HLGr}}
\newcommand{\OLGr}{\operatorname{OLGr}}

\newcommand{\smat}{\left(\begin{smallmatrix}}
\newcommand{\esmat}{\end{smallmatrix}\right)}
\newcommand{\pmat}{\begin{pmatrix}}
\newcommand{\epmat}{\end{pmatrix}}

\newcommand{\bbK}{\mathbb{K}}
\newcommand{\bbA}{\mathbb{A}}
\newcommand{\bbP}{\mathbb{P}}
\newcommand{\bbM}{\mathbb{M}}
\newcommand{\bbL}{\mathbb{L}}
\newcommand{\bbF}{\mathbb{F}}

\newcommand{\frP}{\mathfrak{P}}
\newcommand{\frL}{\mathfrak{L}}
\newcommand{\frN}{\mathfrak{N}}
\newcommand{\frG}{\mathfrak{G}}
\newcommand{\frA}{\mathfrak{A}}
\newcommand{\frK}{\mathfrak{K}}
\newcommand{\frM}{\mathfrak{M}}
\newcommand{\frH}{\mathfrak{H}}

\begin{document}

\sloppy
\title[Cohomology rings of Grassmannians]
{Clifford algebras, symmetric spaces and cohomology rings of Grassmannians}

\author{Kieran Calvert}
\address[Calvert]{Department of Mathematics and Statistics, Fylde Building, Lancaster University, Lancaster, UK}
\email{k.calvert@lancaster.ac.uk}
\author{Kyo Nishiyama}
\address[Nishiyama]{Department of Mathematics, Aoyama Gakuin University, Fuchinobe 5-10-1, Chuo-ku,
Sagamihara 252-5258, Japan}
\email{kyo@math.aoyama.ac.jp}
\author{Pavle Pand\v zi\'c}
\address[Pand\v zi\'c]{Department of Mathematics, Faculty of Science, University of Zagreb, Bijeni\v cka 30, 10000 Zagreb, Croatia}
\email{pandzic@math.hr}
\thanks{K.~Nishiyama is supported by JSPS KAKENHI Grant Number \#{21K03184}.}
\thanks{P.~Pand\v zi\'c is supported by the QuantiXLie  Center of Excellence, a project 
cofinanced by the Croatian Government and European Union through the European Regional Development Fund - the Competitiveness and Cohesion Operational Programme 
(KK.01.1.1.01.0004).}

\subjclass[2020]{primary: 14M15, 57T15, 53C35, 22E47. secondary: 32M15, 15A66}
\keywords{Grassmannian, Lagrangian Grassmannian, symmetric space, de Rham cohomology, spin module, Clifford algebra, exterior algebra, invariants}
\begin{abstract} 
We study various kinds of Grassmannians or Lagrangian Grassmannians over $\bbR$, $\bbC$ or $\bbH$, all of which can be expressed as $\bbG/\bbP$ where $\bbG$ is a classical group and $\bbP$ is a parabolic subgroup of $\bbG$ with abelian unipotent radical. The same Grassmannians can also be realized as (classical) compact symmetric spaces $G/K$. We give explicit generators and relations for the de Rham cohomology rings of $\bbG/\bbP\cong G/K$. At the same time we describe certain filtered deformations of these rings, related to Clifford algebras and spin modules. While the cohomology rings are of our primary interest, the filtered setting of $K$-invariants in the Clifford algebra actually provides a more conceptual framework for the results we obtain.
\end{abstract}

\dedicatory{Dedicated to Bert Kostant}

\maketitle

\section{Introduction}
\label{sec intro}

This paper is motivated by email correspondence between the third named author and Bert Kostant in 2004 \cite{kostantemail}. In his study of the action of Dirac operators on Harish-Chandra modules attached to a real reductive Lie group $G_\bbR$, the third named author was led to consider the algebra $C(\frp)^K$. Here $K$ is a maximal compact subgroup of $G_\bbR$ corresponding to a Cartan decomposition $\frg=\frk\oplus\frp$ of the complexified Lie algebra of $G_\bbR$ and $C(\frp)$ is the Clifford algebra of $\frp$ with respect to the (extended) Killing form. The third named author remembered hearing Kostant speak about the graded version of $C(\frp)^K$, $(\twedge\frp)^K$, and its relation to cohomology, so he asked Kostant about the latter algebra expecting the description of the graded version of the algebra will give some information about the algebra $C(\frp)^K$ he was interested in. Kostant replied as follows:
\bigskip

{\ttfamily

\noindent From kostant@math.mit.edu  Mon Mar 15 21:07:54 2004
	
\qquad for <pandzic@math.hr>; Mon, 15 Mar 2004 21:07:53 +0100 (MET)

\noindent Subject: Re: a question

\bigskip
 
\noindent Dear Pavle
                 
\qquad \qquad The following is known. In general 
(\textbackslash wedge p)\bighat k is the

\noindent cochain complex whose cohomology is H(G/K). In the symmetric case (G/K is a
symmetric space) the cobounbary operator is trivial so that (\textbackslash wedge p)\bighat k =
H(G/K). If in addition rank K = rank G then all cohomology is even
dimension and if e = Euler characteristic of G/K then of course
           \newline dim (\textbackslash wedge p)\bighat  k = e. Also in
this case dim p is even and C(p) = End S where S is the spin module.
Furthermore under the action of K

\begin{center}
    S=V\_1 + .... + V\_e
\end{center}

\noindent where the V\_i are irreducible K modules and all distinct. Thus C(p)\bighat K is an
abelian (semisimple) algebra of dimension e. This is a special case of my
result with Sternberg, Ramond and Gross (GKRS) in the PNAS. It is a nice
unsolved problem to locate the 1-dimensional idempotents which project onto
the various V\_ i. In case G/K is Hermitian symmetric these idempotents I
believe correspond to the Schubert classes in (\textbackslash wedge p)\bighat k. If so one has
some sort of generalization of Schubert classes when G/K is not Hermitian.
\medskip

\noindent Best regards
 
\medskip

\noindent Bert
}
\medskip

\bigskip

It turns out Kostant was not exactly right in thinking that the idempotents will correspond to the Schubert classes; in fact, they typically all have nonzero top degree term, and one must take their linear combinations in order to get a basis compatible with filtration. His intuition that this question is related to some sort of generalization of Schubert classes when $G/K$ is not Hermitian was however right; as we shall see, these cases correspond to various kinds of real or quaternionic  Grassmannians which possess their own Schubert calculus.

Our main goal in this paper is to describe the de Rham cohomology rings of these Grassmannians using their realization as compact symmetric spaces.  The main tool is the representations of the Clifford algebras associated with symmetric spaces.  
When the Grassmannians are complex, the results we obtain here are well known.  However, for the real and quaternionic Grassmannians, the results are not widely known.  
For example, the ring structure of the cohomology of the real Grassmannians was conjectured by Casian and Kodama \cite{Casian.Kodama.2013}.  Recently, many related papers have appeared (see \cite{Rabelo.2016,Lambert.Rabelo.2022,Rabelo.SMartin.2019,Matszangosz.2021} for integral cohomology, for example), 
but these papers treat the Grassmannians individually, and not in a uniform way. See also \cite{chen,CHL,esch,HL} for identifications of compact symmetric spaces and Grassmannians. 

We study these cohomology rings systematically.  The key ideas are 
the usage of Clifford algebras as mentioned above and the description of the Grassmannians as flag manifolds 
corresponding to 
maximal parabolic subgroups with abelian unipotent radicals.  
In fact, we start with the Grassmannian $ \bbG/\bbP $, where $ \bbP $ is a maximal parabolic subgroup with abelian unipotent radical in a reductive Lie group $ \bbG $, then produce symmetric spaces of compact and noncompact type using three involutions 
$ \theta, \sigma $ and $ \tau = \theta \sigma $ of $ \bbG $, which are mutually commuting.  

Our results describe the de Rham cohomology ring of each of the Grassmannians in our list by explicit generators and relations, and also give an explicit basis consisting of certain monomials. In most cases (including the well known complex Grassmannian cases) we show that our basis can be replaced by a basis consisting of certain Schur polynomials. These have the advantage of a rather well understood multiplication table, related to the Littlewood-Richardson coefficients. However our monomials with their clear structure of generators and relations also lead to an explicit multiplication table, as explained in Section \ref{coho symm}. In this way we get an alternative approach to Schubert calculus on the Grassmannians in question. 

The paper is organized as follows.

In Section \ref{sec gras} we describe our Grassmannians $ \bbG/\bbP $ and give their realizations as compact symmetric spaces. The cases are summarized in the table at the end of the introduction. 
In each case $ \bbP $ is a maximal parabolic subgroup 
of $ \bbG $
with abelian unipotent radical.  
Note that a Grassmannian is the set
of certain subspaces of fixed dimension in a vector space $ V $.  
On this set of subspaces, the automorphism group $ \bbG $ of $ V $, which typically preserves additional structure (quadratic or symplectic forms), acts transitively.  This is justified by the following well known theorem by Witt.

\begin{thm}\cite{Witt}, \cite[Ch.~1-2]{Bou}, \cite{Die}. \label{witt}
Let $V$ be a vector space over $\bbR$, $\bbC$ or $\bbH$ with a nondegenerate form $\lara$ which is either bilinear symmetric, or bilinear skew symmetric, or Hermitian, or skew Hermitian. Let $U$ and $W$ be subspaces of $V$ and let $\varphi:U\to W$ be an isomorphism preserving $\lara$. Then $\varphi$ extends to an automorphism of $V$ preserving $\lara$.
\end{thm}

Let $ \bbP $ be the stabilizer in $\bbG$ of a standard subspace $ U$ in our Grassmannian.  (In some cases, $ U $ is a Lagrangian or isotropic subspace.  It depends on the situation.)  
Then $ \bbP $ is a maximal parabolic subgroup and 
our Grassmannian is equal to $\bbG/\bbP $.  
So Grassmannians are naturally identified with (partial) flag manifolds $ \bbG/\bbP $.  

As already mentioned, compact symmetric spaces are  diffeomorphic to Grassmannians $\bbG/\bbP $ where $ \bbP $ has abelian unipotent  radical.
In fact, these Grassmannians exhaust all such pairs $ (\bbG, \bbP) $ (see \cite{Wolf}, \cite[\S~5.5.1]{howe}, and also \cite{RRS}).  
Then, in appropriate realizations, the Grassmannians are varieties of 
either ordinary subspaces of a vector space, or of Lagrangian subspaces with respect to a certain form. 

In this paper, 
the group $\bbG$ will be a classical group; in particular, it is a reductive matrix group, and we always consider the standard Cartan involution 
\eq
\label{cart inv}
\theta(g)=(\bar g^t)^{-1},\quad g\in\bbG.
\eeq
The corresponding maximal compact subgroup $\bbG^\theta$ will be denoted by $ \bbK $ and also by $ G $.  

We will use Proposition \ref{k trans} below to see that $\bbK$ acts transitively on the Grassmannian, so that $\bbG/\bbP$ is diffeomorphic to $ \bbK/\bbP\cap\bbK $. 
This follows easily from the fact that $\bbP$ contains a minimal parabolic subgroup $\bbP_0=\bbM\bbA\bbN_0$, and that $\bbG$ has an Iwasawa decomposition
$\bbG=\bbK\bbA\bbN_0$, where $\bbK$ is as above.

It is known by \cite[\S 4]{TK1968} (see also \cite[Lemma 7.3.1]{kobayashi} and \cite{RRS}) 
that the following are equivalent:
\begin{enumerate}
\item 
$\bbP$ has abelian unipotent 
radical;
\item 
$\bbP$ has a Levi subgroup $\bbL$ which is a symmetric subgroup of $\bbG$;
\item 
$ K := \bbP\cap\bbK=\bbL\cap\bbK $ is  a symmetric subgroup of $ G = \bbK$. 
\end{enumerate}
We give a short and comprehensive proof of this result in Theorem \ref{abel sym} for convenience of the readers.

The involution $ \sigma $  
of $\bbG$ mentioned above is related to the above Levi subgroup $\bbL$ of $\bbP$: $\bbL$ is $\bbG^\sigma$, the subgroup of $\bbG$ consisting of points fixed by $\sigma$; we will see that also $ \bbP $ itself is $ \sigma $-stable. 
The involution $\sigma$, which we describe explicitly below, commutes with the Cartan involution $\theta$, and hence $ K = \bbL^{\theta} = \bbL \cap \bbK $ is a maximal compact subgroup of $ \bbL $. 

Let us denote by $ \tau = \theta \sigma $ the third involution of $ \bbG $.  
We denote the fixed point subgroup $ \bbG^{\tau} $ by $ G_{\bbR} $, so that 
$ G_{\bbR}/ K $ is a noncompact Riemannian symmetric space 
with its compact dual 
equal to $ G/K $.
It turns out that in this way we get to cover the full list of compact classical symmetric spaces as listed e.g. in \cite[p.{69}]{howe}.

We summarize the three involutions and symmetric spaces thus obtained in the following diagram.  
See also Table~\ref{tab:Grass_symspaces} at the end of Introduction.
\begin{equation*}
\vcenter{
\xymatrix @R-.3ex @M+.5ex @C-3ex @L+.5ex @H+1ex {
 & \ar@{-}[ld]_{\sigma} \ar@{-}[d]_{\theta} \makebox[3ex][c]{$\bbG$} \ar@{-}[rd]^{\tau = \theta \sigma} &
\\
\ar@{-}[rd]_(.4){\tau} \makebox[3ex][r]{$\bbL = \bbG^{\sigma}$}  & \ar@{-}[d]_{\sigma} \makebox[9ex][c]{$\bbK = \bbG^{\theta} = G$} & \makebox[3ex][l]{$ \bbG^{\tau} = G_{\bbR} $} \ar@{-}[ld]^(.3){\theta}
\\
& \makebox[9ex][c]{$K = \bbL \cap \bbK = G \cap G_{\bbR}$} & 
}}
\qquad\qquad
\begin{aligned}
\bbP = \bbL \bbN :\ &\text{$\sigma $-stable parabolic subgroup}
\\
&\text{with abelian unipotent radical} 
\\
\tau = \theta & \text{ on } \bbL 
\\
\theta = \sigma & \text{ on }  G_{\bbR} 
\\
\sigma = \tau & \text{ on } \bbK = G 
\end{aligned}
\end{equation*}
Since $\bbP$ is block upper triangular with two diagonal blocks, of sizes (say) $p$ and $q$, $\bbL$ can be taken as the block diagonal part of $\bbP$. Now if we denote by $I_{p,q}$ the matrix $\smat I_p&0\cr 0&-I_q\esmat$, then
\[
I_{p,q}\pmat A&B \cr C&D\epmat I_{p,q}=\pmat A&-B \cr -C&D\epmat,
\]
and we see that the involution we want is
\eq
\label{def sigma}
\sigma(g)=I_{p,q}g I_{p,q}.
\eeq
It follows that $ \bbP $ is $ \sigma $-stable, and 
it will now be very easy to identify the groups $\bbL=\bbG^\sigma$. 
(In fact, we will see in Section 2 that $\sigma$ can be described in terms of $\bbP$ only; see Theorem \ref{abel sym}.)

It turns out that 
the complexifications of these groups are exactly the groups listed in \cite[p.~70]{howe} as the groups acting in a skew-multiplicity free way on $\frp$, the complexified tangent space of 
$ G/K \simeq \bbG/ \bbP $  
at the base point $ e K $.  Here ``skew-multiplicity free" means that $ \bbL_\bbC$ acts on 
$ \twedge \frp $ in a multiplicity free way (note that $ \frp $ can also be identified with the complexification of the Lie algebra of $ \bbN $, so that $ \bbL $ acts on it naturally).  

We will also describe the groups $G_\bbR=\bbG^{\tau} $, where $ \tau = {\sigma\theta}$. Note that by \eqref{cart inv} and \eqref{def sigma}, we get
\eq
\label{sigma theta}
\sigma \theta(g)=I_{p,q}(\bar g^t)^{-1}I_{p,q}.
\eeq
The group $G_\bbR$ is a noncompact reductive Lie group with maximal compact subgroup $K=\bbK^\sigma$, 
and $G/K=\bbK/\bbK^{\sigma}$ is the compact dual of the noncompact Riemannian symmetric space $G_\bbR/K$ as explained above. 
In this way, we get to cover the full list of noncompact classical symmetric spaces.  
As in the Hermitian symmetric case, 
the noncompact Riemannian symmetric space $G_\bbR/K$ is embedded into 
$ \bbG/\bbP $ as an open subset (a generalization of the Borel embedding).

The realizations of our Grassmannians as symmetric spaces are known in most (or all) cases, but the results are scattered in the literature. We will indicate some references when we get to the case by case analysis.
Our view point is to produce 
the classical Riemannian symmetric spaces, both compact and noncompact ones, 
in terms of the pairs $ (\bbG, \bbP)$ on our list.

In Section \ref{sec spin} we collect some facts needed for our description of the cohomology rings of compact symmetric spaces $G/K$ as above (and thus also of the corresponding Grassmannians $ \bbG/\bbP $). 
For $\sigma$ as above, its restriction to $G$ is an involution such that $K=G^\sigma$. 
Let $\frg=\frk\oplus\frp$ be the decomposition of the complexified Lie algebra $\frg$ of $G$ into eigenspaces of $\sigma$. We are assuming $G$ is connected, but $K$ need not be connected.

As noted in Kostant's email message, if $\frg$ and $\frk$ have equal rank, then the algebras $C(\frp)^K$ and $(\twedge\frp)^K$ can be expressed as
\[
C(\frp)^K= \Proj(S);\qquad (\twedge\frp)^K= \gr\Proj(S),
\]
where the algebra $\Proj(S)$ is spanned by the projections of the spin module $S$ to its isotypic components for the pin double cover $\Kt$ of $K$.
As explained in Subsections \ref{K dec general} and \ref{cpk eq rk}, one can use the natural map $\alpha:U(\frk)\to C(\frp)$ that gives the spin module its $\frk$-module structure, and the fact that the projections are given by the action of the center of $U(\frk)$,  to express 
the algebra $\Proj(S)$ as the quotient of $\bbC[\frt^*]^{W_K}$ by the ideal generated by the elements of $\bbC[\frt^*]^{W_G}$ vanishing at $\rho$. Here $\frt$ is the complexification of a Cartan subalgebra of the Lie algebra of $K$, $W_K$ is the Weyl group of $K$ (see \eqref{gp W gp}), $W_G=W_\frg$ is the Weyl group of $G$ or equivalently the Weyl group of the root system $\Delta(\frg,\frt)$, and $\rho$ is the half sum of roots in the (fixed) positive root system $\Delta^+(\frg,\frt)$. 
The algebra $\gr\Proj(S)$ attached to the natural filtration of $\Proj(S)$ can be expressed as the quotient of $\bbC[\frt^*]^{W_K}$ by the ideal generated by the elements of $\bbC[\frt^*]^{W_G}$ vanishing at $0$.

In the ``almost equal rank case" $(\frg,\frk)=(\frso(2p+2q+2),\frso(2p+1)\times\frso(2q+1))$, the algebras $C(\frp)^K$ and $(\twedge\frp)^K$ are very close to $\proj(S)$ and $\gr\proj(S)$; one has to tensor with the Clifford respectively exterior algebra of a certain one-dimensional space. This is explained in Subsections \ref{subsec so odd} and \ref{subsec so/soo}.

On the opposite end are the primary and almost primary cases; in these cases the $\Kt$-module $S$ has only one isotypic component and the algebra of projections is trivial. These cases are described in Subsections \ref{subsec primary and aprim} and \ref{subsec u/o}. The algebra $(\twedge\frp)^K$ is now equal to the exterior algebra of a certain subspace of $(\twedge\fp)^K$ (denoted by $\cal P_\wedge(\frp)$), where $\cal P_\wedge(\frp)$ is the subspace orthogonal to the square of the augmentation ideal (Definition \ref{def samspace}). This result is due to Hopf and Samelson \cite{hopf1941,Samelson41} for the group case and Theorem \ref{thm prim and aprim alg} for the remaining (almost) primary cases. Similar isomorphisms 
 go way back to Cartan, Chevalley, Koszul and others; see \cite[p. 568]{CCCvol3}. 
We believe that likewise $C(\frp)^K$ is the Clifford algebra over $\cal P_\wedge(\frp)$, but this is currently known (by the results of Kostant \cite{K97}) only in the group cases, i.e., when $\frg=\frg_1\oplus\frg_1$ and $\frk\cong\frp\cong\frg_1$.

Finally, in Section \ref{coho symm} we 
give a precise description, by generators and relations, of the algebras $(\twedge\frp)^K$ in each of the cases. In the equal rank and almost equal rank cases, we also describe the algebras $C(\frp)^K$, which in these cases amounts to describing the algebras $\Proj(S)$. 
 We also give explicit bases for these algebras. 
 In this way we get to compute the de Rham cohomology of the symmetric spaces on our list, and thus also of the corresponding Grassmannians.

 Namely, as mentioned in Kostant's email message, 
the cohomology of the compact symmetric space $G/K$ is (after complexification) equal to $(\twedge\frp)^K$. 
This fact is quite well known, but it is not easy to find an appropriate reference. It is proved in  
 \cite{taylor} (unpublished) using Hodge theory, partially proved in \cite{leung}, and  proved in \cite{CCCvol3} under the assumption $K$ is connected.
 It is also mentioned in passing in \cite[p.~69]{howe},
 and in \cite[\S~1.6]{BW}. 
Borel and Wallach \cite{BW} attribute the result to \'E. Cartan and de Rham. We start Section \ref{coho symm} by presenting a simple proof of this fact which we learned from Sebastian Goette \cite{goette}.

We are thus led to study the appropriate quotients of the $W_K$-invariants in $\bbC[\frt^*]$ in each of the cases. The results often involve the following algebra.

\begin{Def}
\label{def alg h}
Let $p,q\in \bbZ$ with $1\leq p\leq q$, and let $c=(c_1,\dots,c_{p+q})\in\bbC^{p+q}$. We define 
$\frH(p,q;c)$ to be the algebra generated by $r_1,\dots,r_p$ and $s_1,\dots,s_q$ with relations generated by 
\[
\sum_{i,j\geq 0;\ i+j=k}r_is_j=c_k
\]
for $k=1,\dots,p+q$, where we set $r_0=s_0=1$ and $r_i=0$ if $i>p$, $s_j=0$ if $j>q$. 

We can use the first $q$ of the relations to express $s_1,\dots,s_q$ in terms of the $r_i$, so $\frH(p,q;c)$ is in fact generated by $r_1,\dots,r_p$ only. 
The remaining relations can be used to obtain the relations among the $r_i$, not involving the $s_j$. We do that in the proof of Theorem \ref{gen rel basis}; the relations among the $r_i$ are \eqref{rel explic 1} and \eqref{rel explic other} and they form another set of defining relations. From these relations one can obtain expressions for each monomial in $r_1,\dots,r_p$ of degree $q+1$ as a linear combination of lower degree monomials in $r_1,\dots,r_p$. We will also see that the monomials in the $r_i$ of degree at most $q$ form a basis of the algebra $\frH(p,q;c)$; so $\frH(p,q;c)$ can be identified with the space
\[
\bbC[r_1,\dots,r_p]_{\leq q}
\]
of polynomials in the $r_i$ of degree $\leq q$. 

We show in Remark \ref{rem schur} that the above monomials span the same subspace of $\bbC[r_1,\dots,r_p]$ as the Schur polynomials $s_\lambda$ attached to  partitions $\lambda$ with Young diagrams contained in the $p\times q$ box. Moreover, our basis consisting of monomials and the basis consisting of Schur polynomials are connected by a triangular change of basis. In this way we get a connection with the usual Schubert calculus, where the multiplication of the Schur polynomials is given in terms of Littlewood-Richardson coefficients.
\end{Def}

For $G/K=\UU(p+q)/\UU(p)\times \UU(q)$, we prove in Theorem \ref{gen rel basis} that $C(\frp)^K$ is isomorphic to the algebra $\frH(p,q;c)$, with the $r_i$ being the elementary symmetric functions in the first $p$ coordinate functions $x_1,\dots,x_p$ on $\frt\cong\bbC^{p+q}$, the $s_j$ being the elementary symmetric functions in the last $q$ coordinate functions $x_{p+1},\dots,x_{p+q}$ on $\frt$, and $c=(t_1(\rho),\dots,t_{p+q}(\rho))$, where $t_k$ are the elementary symmetric functions on $x_1,\dots,x_{p+q}$. 
The natural filtration on $C(\frp)^K$ coming from the filtration of $C(\frp)$ corresponds to the filtration on $\frH(p,q;c)$ obtained by setting 
\eq
\label{set deg upq}
\deg r_i=2i,\qquad i=1,\dots,p.
\eeq
The cohomology ring $(\twedge\frp)^K$ is isomorphic to $\frh(p,q,0)$, and its natural grading is again obtained by \eqref{set deg upq}.

For $G/K=\Sp(p+q)/\Sp(p)\times \Sp(q)$ (Theorem \ref{gen rel basis B/BxB}), the algebra $C(\frp)^K$ is again $\frH(p,q;c)$, but now the $r_i$, the $s_j$ and the $t_k$ are elementary symmetric functions on the squares of the appropriate variables. The parameter $c$ is again given by evaluating $t_k$ at $\rho$. The filtration is now obtained by setting $\deg r_i=4i$,  $i=1,\dots,p$. The 
cohomology ring $(\twedge\frp)^K$ is again isomorphic to $\frh(p,q,0)$, and its natural grading is also obtained by setting $\deg r_i=4i$.
For $G/K=\SO(k+m)/S(\OO(k)\times \OO(m))$ (Theorem \ref{gen rel basis SOk+m}), the algebra $C(\frp)^K$ is $\frH(p,q;c)$ if $(k,m)=(2p,2q)$ or $(2p,2q+1)$, with $\{r_i\}$, $\{s_j\}$, $\{t_k\}$ and $c$ defined similarly as above. If $(k,m)=(2p+1,2q+1)$ (the almost equal rank case), there is an extra generator $e$, of degree $2p+2q+1$, squaring to $1$.
The filtration degrees of the $r_i$ are again equal to $4i$. The 
cohomology ring $(\twedge\frp)^K$ is  isomorphic to $\frH(p,q,0)$ or
$\frH(p,q,0)\oplus\frH(p,q;0)e$, and its natural grading is also obtained by setting $\deg r_i=4i$ and $\deg e=2p+2q+1$. In this case we get to prove the conjecture of Casian-Kodama \cite{Casian.Kodama.2013}.

For $G/K=\UU(n)/\OO(n)$ (Theorem \ref{thm prim and aprim alg}), the situation is different: the algebra $(\twedge\frp)^K$ is the exterior algebra on the subspace $\cal P_\wedge(\frp)$ (Definition \ref{def samspace}), and the degrees are given in Table \ref{tab degrees prim}. 

For $G/K=\Sp(n)/U(n)$ (Theorem \ref{thm basis C/A}), we are back to elementary symmetric functions: $r_1,\dots,r_n$ are the elementary symmetric functions on the coordinate functions $x_1,\dots,x_n$ on the Cartan subalgebra $\frt\cong\bbC^n$ of $\frk$, while $t_1,\dots,t_n$ are the elementary symmetric functions on the squares of the $x_i$. The algebras $C(\frp)^K$ and $(\twedge\frp)^K$ are generated by the $r_i$, but the relations are now different:
\[
r_k^2=t_k+2r_{k-1}r_{k+1}-2r_{k-2}r_{k+2}+\dots,\quad k=1,\dots,n,
\]
where as usual we set $r_0=1$ and $r_i=0$ for $i>n$ or $i<0$, and where $t_k$ should be replaced by $t_k(\rho)$ if the algebra is $C(\frp)^K$ and with 0 if the algebra is $(\twedge\frp)^K$. This time a basis for each of our algebras is given by the monomials
\[
r_1^{\eps_1}r_2^{\eps_2}\dots r_n^{\eps_n},\quad \eps_i\in\{0,1\}.
\]
The filtration degree of $C(\frp)^K$ inherited from $C(\frp)$, and the gradation degree of $(\twedge\frp)^K$ inherited from $\twedge\frp$, are obtained by setting $\deg r_i=2i$ for $i=1,\dots,n$.

For $G/K=\SO(2n)/\UU(n)$ (Theorem \ref{thm basis D/A}) the situation is entirely analogous to the case $G/K=\Sp(n)/U(n)$, except that we get to eliminate $r_n$ from the list of generators.

For the group cases $G\times G/\Delta G\cong G$ where $G$ is $\SO(n)$, $\UU(n)$ or $\Sp(n)$ (Theorem \ref{gen rel basis group}), the algebras $C(\frp)^K\cong C(\frg)^\frg$ and $(\twedge\frp)^K\cong(\twedge\frg)^\frg$ are the Clifford respectively exterior algebra of the graded subspace $\cal P_\wedge(\frp)$ (Definition \ref{def samspace}). 

For the Clifford algebras, these cases were settled by Kostant in \cite{K97}. For the exterior algebras, the fact that $(\twedge \fp)^\fk$ is isomorphic to a graded subspace goes back to Cartan, Chevalley, Koszul and others \cite{CCCvol3}. 
The degrees are given in Table \ref{tab degrees prim}.

For the cases $G/K=\UU(2n)/\Sp(n)$ (Theorem \ref{gen rel basis A/C}) the algebra  $(\twedge\frp)^K$ is the exterior algebra of the graded subspace $\cal P_\wedge(\frp)$ (Definition \ref{def samspace}).
Again, the degrees are given in Table \ref{tab degrees prim}.

Our results are well known for complex Grassmannians, i.e., for $\Gr_p(\bbC^{p+q})\cong\UU(p+q)/\UU(p)\times\UU(q)$, $\Gr_2(\bbR^{2+q})\cong\SO(2+q)/S(\OO(2)\times\OO(q))$, $\LGr(\bbC^{2n})\cong\Sp(n)/\UU(n)$ and $\OLGr^+(\bbC^{2n})\cong\SO(2n)/\UU(n)$. Among many papers dealing with the complex Grassmannians and their Schubert calculus, we mention 
\cite{Fulton.YT.1997}, \cite{Fulton.Pragacz.1998}, \cite{Tamvakis2005,TamvakisArbeitTalk2001}, and  \cite{Pragacz.1991,Pragacz.1996,Pragacz.2003}.

\newpage

\begin{table}[hpbt]
    \centering
    \caption{\large\textbf{Table of Grassmannians and corresponding symmetric spaces.}}\label{tab:Grass_symspaces}
\resizebox{\linewidth}{!}{

\begin{tabular}{lllllll}
\makebox[0pt][l]{\textbf{General Linear} $n = p+q$}\\
$\mathbb{G}$ & $\mathbb{P}$ & $\bbL =\mathbb{G}^{\sigma}$   &    $\mathbb{K} = \mathbb{G}^\theta = G $ &  $ \mathbb{G}^{\sigma\theta}= G_\mathbb{R}$ & $\mathbb{K}^{\sigma} = G_\mathbb{R}^\theta= K$ & $\mathfrak{p}_0$
\\  \hline \\
$\GL_{n}(\mathbb{R})$  &  $\mathrm{Stab}(\mathbb{R}^p)$ & $\GL_p(\mathbb{R}) {\times} \GL_q(\mathbb{R})$ &   $\UU_{n}(\mathbb{R})$ & $\UU_{p,q}(\mathbb{R})$ & $\UU_p(\mathbb{R}) {\times} \UU_q(\mathbb{R})$ 
& $ \Mat_{p,q}(\bbR) $
\\

$\GL_n(\mathbb{C})$  & $\mathrm{Stab}(\mathbb{C}^p)$ & $\GL_p(\mathbb{C}) {\times} \GL_q(\mathbb{C})$ & 
   $\UU_n(\mathbb{C})$ & $\UU_{p,q}(\mathbb{C})$ & $\UU_p(\mathbb{C}) {\times} \UU_q(\mathbb{C})$ 
& $ \Mat_{p,q}(\bbC) $
\\

${\GL}_{n}(\mathbb{H})$  & $\mathrm{Stab}(\mathbb{H}^p) $ & $\GL_p(\mathbb{H}) {\times} \GL_q(\mathbb{H})$&    $\UU_{n}(\mathbb{H})$ & $\UU_{p,q}(\mathbb{H})$ & $\UU_p(\mathbb{H}) {\times} \UU_q(\mathbb{H})$
& $ \Mat_{p,q}(\bbH) $
\\

\quad \\
\makebox[0pt][l]{\textbf{Symplectic}}\\
$\mathbb{G}$ & $\mathbb{P}$ &  $\bbL =\mathbb{G}^{\sigma}$ &      $\mathbb{K} = \mathbb{G}^\theta = G $ &   $ \mathbb{G}^{\sigma\theta}= G_\mathbb{R}$ & $\mathbb{K}^{\sigma} = G_\mathbb{R}^\theta= K$ &$\mathfrak{p}_0$
\\ 
\hline \\

 $\Sp_{2n}(\mathbb{R})$  & $\mathrm{Stab}(L_0)$ &$\GL_n(\mathbb{R})$ & $\UU_n(\mathbb{C})$  & $\GL_n(\mathbb{R})$ & $\UU_n(\mathbb{R})$  
& $ \Sym_n(\bbR) $
\\

 $\Sp_{2n}(\mathbb{C})$  &   $\mathrm{Stab}(L_0)$ & $\GL_n(\mathbb{C})$ &$\UU_n(\mathbb{H})$ & $\Sp_{2n}(\mathbb{R})$ & $\UU_n(\mathbb{C})$ 
& $ \Sym_n(\bbC) $
\\

\quad \\
\makebox[0pt][l]{\textbf{Orthogonal}}\\
$\mathbb{G}$ & $\mathbb{P}$ &  $\bbL =\mathbb{G}^{\sigma}$ &      $\mathbb{K} = \mathbb{G}^\theta = G $ &   $ \mathbb{G}^{\sigma\theta}= G_\mathbb{R}$ & $\mathbb{K}^{\sigma} = G_\mathbb{R}^\theta= K$ &$\mathfrak{p}_0$
\\ 
\hline \\

$\OO_{2n}(\mathbb{C})$  & $\mathrm{Stab}(L_0)$ &$\GL_n(\mathbb{C})$  &     $\OO_{2n}(\mathbb{R})$  & $ \SO^*(2n, j \Id_n)$ & $  \UU_n(\mathbb{C})$ 
& $ \Alt_n(\bbC) $
\\

\quad \\
\makebox[0pt][l]{\textbf{Hermitian}}\\
$\mathbb{G}$ & $\mathbb{P}$ &  $\bbL =\mathbb{G}^{\sigma}$ &      $\mathbb{K} = \mathbb{G}^\theta = G $ &   $ \mathbb{G}^{\sigma\theta}= G_\mathbb{R}$ & $\mathbb{K}^{\sigma} = G_\mathbb{R}^\theta= K$ &$\mathfrak{p}_0$
\\ 
\hline \\

$\UU_{n,n}(\mathbb{R})$  &   $\mathrm{Stab}(L_0)$ &$\GL_n(\mathbb{R})$ &
  $\UU_n(\mathbb{R})^2$ & $ \OO_n(\mathbb{C})$& $\Delta(\UU_n(\mathbb{R}))$ 
& $ \mathrm{SHer}_n(\bbR) $
\\
$\UU_{n,n}(\mathbb{C})$  &  $\mathrm{Stab}(L_0)$ & $\GL_n(\mathbb{C})$ &
   $\UU_n(\mathbb{C})^2$  & $\GL_n(\mathbb{C}) $ & $\Delta(\UU_n(\mathbb{C}))$ 
& $ \mathrm{SHer}_n(\mathbb{C}) $
\\
 $\UU_{n,n}(\mathbb{H})$  &  $\mathrm{Stab}(L_0)$ &$\GL_n(\mathbb{H})$  &   $\UU_n(\mathbb{H})^2$ & $\Sp_{2n}(\mathbb{C})$& $\Delta(\UU_n(\mathbb{H}))$
& $  \mathrm{SHer}_n(\mathbb{H})$ 
\\
%

\quad \\
\makebox[0pt][l]{\textbf{Skew Hermitian}}\\
$\mathbb{G}$ & $\mathbb{P}$ &  $\bbL =\mathbb{G}^{\sigma}$ &      $\mathbb{K} = \mathbb{G}^\theta = G $ &   $ \mathbb{G}^{\sigma\theta}= G_\mathbb{R}$ & $\mathbb{K}^{\sigma} = G_\mathbb{R}^\theta= K$ &$\mathfrak{p}_0$
\\ 
\hline \\

$\SO^*(4n) $  &  $\mathrm{Stab}(L_0)$ &$ \GL_n(\bbH) $   &  $\UU_{2n}(\mathbb{C})$ &   $\GL_n(\mathbb{H})$ & $\UU_n(\mathbb{H})$ 
& $ \Her_n(\bbH) $
\\
%
\quad\\
\makebox[0pt][l]{\textbf{Quadric}}\\
$\mathbb{G}$ & $\mathbb{P}$  & $\bbL =\mathbb{G}^{\sigma}$  &   $\mathbb{K} = \mathbb{G}^\theta = G $  & $ \mathbb{G}^{\sigma\theta}= G_\mathbb{R}$ & $\mathbb{K}^{\sigma} = G_\mathbb{R}^\theta= K$  & $\mathfrak{p}_0$
\\ 
\hline \\

$\SO_{p+1,q+1}(\mathbb{R})_e$  &  $P_1(Q_n(\mathbb{R}))$ & $\SO_{1,1}(\mathbb{R}) {\times} \SO_{p,q}(\mathbb{R})$ &    $\SO_{p+1}(\mathbb{R}) {\times} \SO_{q+1}(\mathbb{R})$  & $\SO_{1,p}(\mathbb{R}){\times} \SO_{q,1}(\mathbb{R})$ & $S(\OO_{p}(\mathbb{R}) {\times} \OO_{q}(\mathbb{R})) $ 
& $\mathbb{R}^{p-1} \oplus \mathbb{R}^{q-1}$
\\

$\SO_{n+2}(\mathbb{C})$ & $P_1(Q_n(\mathbb{C}))$ & 
$ S(\OO_{2}(\mathbb{C}) {\times} \OO_n(\mathbb{C}))$&$\SO_{n+2}(\mathbb{R})$  & $\SO_{2,n}(\mathbb{R})$ & $S(\OO_2(\mathbb{R}) {\times} \OO_{n}(\mathbb{R}))$   &$\mathbb{C}^n = \Mat_{2,n}(\mathbb{R})$\\
\end{tabular}
}


\vskip .5cm
\hspace{-1cm}
\begin{minipage}{1\textwidth}

\[  
\begin{array}{l l}
&Key: \quad n = p + q,
\\&\bbP: \text{ maximal parabolic subgroup with abelian nilpotent radical}, \bbP = \bbL \ltimes \bbN : \text{Levi decomposition}, \quad
\\
&\bbL = \bbG^\sigma : \text{ Levi of } \bbP, \quad G = \bbK = \bbG^{\theta} \subset \bbG : \text{ maximal compact},    \quad G_{\bbR} = \bbG^{\sigma\theta} : \text{ noncompact real group},

\\ & K = G_{\bbR}^\theta = \bbG^{\sigma, \theta} \subset G_{\bbR} : \text{ maximal compact subgroup,}\quad   \mathrm{Lie}(\bbN) \simeq \mathfrak{p}_0 = \mathrm{Lie}(G)^{-\sigma},
\\
&
K = \bbK \cap \bbP \:(\text{except for }\bbG = \SO_{n+2}(\mathbb{C}) \text{ when }K_e = \bbK \cap \bbP \text{ and 
} \bbG /\bbP \simeq G /K_e) ,
\\& 
\text{Compact symmetric space: }   G / K = \bbK/K \simeq 
\bbG/ \bbP \text{ Grassmannian}, \quad


\\& L_0 : \text{ maximal Lagrangian subspace }, \quad
P_1(Q_n(\mathbb{F})) :  \text{ stabiliser of a point in quadric},  

\\ &\mathrm{Her}_n(\mathbb{F}) : \text{ Hermitian matrices}, \quad \mathrm{SHer}_n(\mathbb{F}) : \text{ Skew Hermitian matrices},

\\& \UU_{p,q}(\mathbb{F}) 
= \OO({p,q})  \text{ if } \mathbb{F}=\mathbb{R}, \:\:
\UU({p,q})   \text{ if } \mathbb{F} = \mathbb{C}, \: \: 
\Sp({p,q})  \text{ if } \mathbb{F} = \mathbb{H}.

\end{array}\]
\end{minipage}
\end{table}

%
%


%


\newpage

\section{Realization of certain Grassmannians as compact symmetric spaces}
\label{sec gras}

\subsection{Some general facts}
\label{structure} 
Let $\bbG$ be one of the groups in Table 1; note that the corresponding symmetric spaces $G/K$ and $G_\bbR/K$ exhaust the list of classical symmetric spaces given in \cite[Ch.9, Sec.4]{hel}. Let $\bbP$ be the parabolic subgroup of $\bbG$ described in Table 1. Then $\bbP$ has a Levi decomposition $\bbP=\bbL\bbN$ specified in Table 1.
(As we shall see in the case by case analysis in the subsequent subsections, $\bbP$ consists of the block upper triangular matrices in $\bbG$ with two diagonal blocks, while $\bbL$ consists of the block diagonal matrices in $\bbP$.)
Let $\frP=\frL\oplus \frN$ be the corresponding decomposition of the Lie algebra of $\bbP$. The opposite parabolic subalgebra is $\frP^-=\frL\oplus\frN^-=\frL\oplus\theta\frN$, where the differential of $\theta$ is still denoted by $\theta$. 

The parabolic subgroup $\bbP$ contains a minimal parabolic subgroup $\bbP_0=\bbM \bbA\bbN_0$ corresponding to an Iwasawa decomposition $\bbG=\bbK\bbA\bbN_0$. The Levi subgroup $\bbL$ of $\bbP$ contains $\bbM\bbA$, while the unipotent radical $\bbN$ of $\bbP$ is contained in $\bbN_0$. Let $\frG,\frA,\frM$ be the Lie algebras of $\bbG,\bbA,\bbM$. Recall that $\frP$ can be constructed by taking a subset of simple $(\frG,\frA)$ roots. Then one
generates a root subsystem by these simple roots, which defines $\frL$ as the span of $\frM\oplus\frA$ and the root spaces for the roots in this subsystem, and defines $\frN$ to be the span of the root spaces for the remaining positive roots.

\begin{lem}
	\label{nn-}
{\rm (1)}	With the above notation, suppose that $\gamma,\delta$ are roots of $\frN$ such that $\gamma+\delta$ is a root (hence a root of $\frN$). Then $[\frG_\gamma,\frG_\delta]\neq 0$.

{\rm (2)} Suppose $\frN$ is abelian. Then $[\frN,\frN^-]=[\frN,\theta\frN]$ is contained in $\frL$.
\end{lem}

\pf
(1) Let $\delta-k\gamma,\dots,\delta,\delta+\gamma,\dots,\delta+n\gamma$ be the $\gamma$-string of roots through $\delta$, with $k\geq 0$ and $n\geq 1$. Let $e\in\frG_\gamma$ be nonzero. By \cite[Proposition 6.52]{beyond} there is an $\frsl_2$-triple $e,h,f$ with $h\in\frA$ and $f\in\frG_{-\gamma}$. Now $\frG_{\delta-k\gamma}\oplus\dots\oplus\frG_\delta\oplus\frG_{\delta+\gamma}\oplus\dots\oplus\frG_{\delta+n\gamma}$ is a representation of the $\frsl_2$ spanned by $e,h,f$, and since $\frG_\delta$ and $\frG_{\delta+\gamma}$ are both nonzero, the action of $e$ between them can not be 0. This implies (1).

(2) Assume that $[\frN,\theta\frN]$ is not contained in $\frL$. Then there are root vectors $x\in\frG_\alpha,\, y\in\frG_\beta$ in $\frN$ such that $[x,\theta y]\notin \frL$. Since $[x,\theta y]\in \frG_{\alpha-\beta}$, it follows that $\alpha-\beta$ is a root either of $\frN$ or of $\frN^-$. If $\alpha-\beta$ is a root of $\frN$, then since $\alpha=(\alpha-\beta)+\beta$ is a root, (1) implies that $[\frG_{\alpha-\beta},\frG_\beta]\neq 0$, so $[\frN,\frN]\neq 0$ and $\frN$ is not abelian. If $\alpha-\beta$ is a root of $\frN^-$, then $\beta-\alpha$ is a root of $\frN$, so $\beta=(\beta-\alpha)+\alpha$ again implies that $[\frN,\frN]\neq 0$ and so $\frN$ is not abelian.
\epf

The following theorem was proved in \cite[\S 4]{TK1968}. 
See also \cite[Lemma 7.3.1]{kobayashi} and \cite{RRS}.
We present a short proof for the convenience of the reader.

\begin{thm}\cite{TK1968}.
\label{abel sym}
Let $\bbG$, $\bbK$ and $\bbP=\bbL\bbN$ be as above (i.e., as in Table 1). Then the following statements are equivalent:
\begin{enumerate}
    \item $\bbN$ (or equivalently $\frN$) is abelian;
    \item $\bbL$ is a symmetric subgroup of $\bbG$;
    \item $\bbP\cap\bbK=\bbL\cap\bbK$ is a symmetric subgroup of $\bbK$.
\end{enumerate}
	\end{thm}
\pf {\bf $\mathbf{(a)\Rightarrow (b)}$.} It is enough to show that in the decomposition $\frG=\frL\oplus(\frN\oplus\theta\frN)$ we have $[\frN\oplus\theta\frN,\frN\oplus\theta\frN]\subseteq\frL$. But Lemma \ref{nn-}(2) implies that 
\[
[\frN\oplus\theta\frN,\frN\oplus\theta\frN]=[\frN,\frN]+[\frN,\theta\frN]+[\theta\frN,\frN]+[\theta\frN,\theta\frN]=0+[\frN,\theta\frN]+0\subseteq \frL.
\]
Note that the associated involution $ \sigma $ is defined to be $ +1$ on $ \frL $ and $ (-1) $ on $ \frN\oplus\theta\frN $.  
In particular $ \bbP $ is $ \sigma $-stable.

{\bf $\mathbf{(b)\Rightarrow(c)}$.} Let $\sigma$ be an involution of $\bbG$ such that $\bbG^\sigma_e\subseteq \bbL\subseteq \bbG^\sigma$, where $\bbG^\sigma_e$ denotes the connected component of $\bbG^\sigma$. 
Since $\bbP$ is standard, we may assume $\sigma$ commutes with $\theta$. Then the restriction of $\sigma$ to $\bbK$ is an involution, and $\bbK^\sigma_e\subseteq \bbL\cap \bbK\subseteq \bbK^\sigma$. So $\bbL\cap\bbK$ is a symmetric subgroup of $\bbK$.
(We remark that in all the examples we consider we will have $\bbL=\bbG^\sigma$ and $\bbL\cap\bbK=\bbK^\sigma$.)

{\bf $\mathbf{(c)\Rightarrow(a)}$.}
Suppose that $\frN$ is not abelian. Then there are roots $\alpha,\beta$ of $\frN$ and $x\in\frG_\alpha$, $y\in\frG_\beta$ such that $[x,y]\neq 0$. Then $x+\theta x,y+\theta y\in (\frN\oplus\theta\frN)^\theta$, and we have
\[
[x+\theta x,y+\theta y]=[x,y]+[x,\theta y]+[\theta x,y]+[\theta x,\theta y],
\]
with
\[
[x,y]\in \frG_{\alpha+\beta},\ [x,\theta y] \in \frG_{\alpha-\beta},\ [\theta x,y]\in \frG_{-\alpha+\beta},\ [\theta x,\theta y]\in \frG_{-\alpha-\beta}.
\]
Since the root $\alpha+\beta$ is strictly greater than $\alpha-\beta$, $-\alpha+\beta$ and $-\alpha-\beta$ (in the usual lexicographical order), we see that $[x,y]\in\frN\setminus 0$ implies $[x+\theta x,y+\theta y]\notin\frL\cap\frK$. It follows that $\bbL\cap\bbK$ is not a symmetric subgroup of $\bbK$.
\epf

\begin{rem}
\label{max par}
Suppose that $[\frG,\frG]$ is simple and that $\frP=\frL\oplus\frN$ is a standard parabolic subalgebra as above. If $\frN$ is abelian, then $\frP$ is a maximal parabolic subalgebra. See \cite[Lemma 2.2, p.651]{RRS}.
\end{rem}

\begin{prop}
	\label{k trans}
 Let $\bbG$, $\bbK$ and $\bbP=\bbL\bbN$ be as above (i.e., as in Table 1). 
 Then $\bbK$ acts transitively on $\bbG/\bbP$, and therefore $\bbG/\bbP$ is diffeomorphic to $\bbK/\bbP\cap\bbK=\bbK/\bbL\cap\bbK$. In particular, $ \bbG/\bbP $ is diffeomorphic to a symmetric space.
\end{prop}
\pf
It is clear from the Iwasawa decomposition that $\bbK$ acts transitively on $\bbG/\bbP_0$, where $\bbP_0$ is a minimal parabolic subgroup of $\bbG$ contained in $\bbP$. 
Since $\bbP\supseteq\bbP_0$, there is a natural projection from $\bbG/\bbP_0$ to $\bbG/\bbP$, sending $g\bbP_0$ to $g\bbP$. This projection intertwines the $\bbG$-actions, hence also the $\bbK$-actions. It follows that $\bbK$ acts transitively on $\bbG/\bbP$. Indeed, if $g\bbP\in\bbG/\bbP$, let $k\in\bbK$ be such that $k\bbP_0=g\bbP_0$. Taking the projection we see that $k\bbP=g\bbP$, which implies transitivity of the $\bbK$-action.
\epf

\subsection{Ordinary Grassmannians}
\label{real gras}

Let $\bbF$ be $\bbR$, $\bbC$ or $\bbH$.
Let $\Gr_p(\bbF^{p+q})$ be the  Grassmannian of $p$-dimensional subspaces of the vector space $\bbF^{p+q}$. The group $\bbG=\GL(p+q,\bbF)$ clearly acts transitively on $\Gr_p(\bbF^{p+q})$, so $\Gr_p(\bbF^{p+q})=\bbG/\bbP$, where $\bbP$ is the stabilizer in $\bbG$ of the standard $p$-dimensional subspace
\eq
\label{std rp}
\bbF^p=\{(x_1,\dots,x_p,0,\dots,0)\bbar x_1,\dots,x_p\in\bbF\} \subseteq \bbF^{p+q}.
\eeq
In other words, 
\[
\bbP=\left\{\begin{pmatrix} A & B\cr 0 & C \end{pmatrix}\bbar A\in \GL(p,\bbF),\,C\in \GL(q,\bbF), \, B\in M_{pq}(\bbF)\right\}.
\]

Let $\sigma$ be the involution of $\bbG$ defined as in \eqref{def sigma}, i.e., $\sigma(g)=I_{p,q} g I_{p,q}$ where 
$I_{p,q}=\smat I_p&0\cr 0&-I_q\esmat$. Then $\bbG^\sigma$ is equal to the Levi subgroup $\bbL$ of $\bbP$ and it consists of block diagonal matrices in $\bbG$, i.e., 
\[
\bbG^\sigma=\bbL=\GL(p,\bbF)\times \GL(q,\bbF).
\]
The maximal compact subgroup $\bbK$ of $\bbG$ is the unitary group of $\bbF^{p+q}$ with respect to the standard inner product, denoted as $\UU(p+q,\bbF)$. In other words, $\bbK$ is $\OO(p+q)$ if $\bbF=\bbR$, $\UU(p+q)$ if $\bbF=\bbC$, and $\Sp(p+q)$ if $\bbF=\bbH$. $\bbP\cap\bbK=\bbK^\sigma$ is $\UU(p,\bbF)\times \UU(q,\bbF)$, embedded block diagonally. In other words, $\bbK^\sigma$ is $\OO(p)\times \OO(q)$ if $\bbF=\bbR$,
$\UU(p)\times \UU(q)$ if $\bbF=\bbC$, and $\Sp(p)\times \Sp(q)$ if $\bbF=\bbH$.

Now Proposition \ref{k trans} implies the following well known result which can be found for example in \cite[Ch. I, \S 4]{On}.

\begin{prop}
\label{prop real grass}
$\Gr_p(\bbF^{p+q}) = \bbG/\bbP$ is diffeomorphic to $\UU(p+q,\bbF)/\UU(p,\bbF)\times \UU(q,\bbF)$. In other words, $\Gr_p(\bbR^{p+q})$ is diffeomorphic to $\OO(p+q)/\OO(p)\times \OO(q)$, $\Gr_p(\bbC^{p+q})$ is diffeomorphic to $\UU(p+q)/\UU(p)\times \UU(q)$, and $\Gr_p(\bbH^{p+q})$ is diffeomorphic to $\Sp(p+q)/\Sp(p)\times \Sp(q)$
\end{prop}

For cohomology computation we need $G=\bbK$ to be connected, and it is connected if $\bbF$ is $\bbC$ or $\bbH$. If $\bbF=\bbR$, we note that
\[
\bbG/\bbP\cong \SO(p+q)/S(\OO(p)\times \OO(q)).
\]
This follows immediately from the fact that $\SL(p+q,\bbR)$ acts transitively on $\Gr_p(\bbR^{p+q})$. We can now conclude

\begin{cor}
\label{coho real grass}
The cohomology ring (with complex coefficients) of the Grassmannian $\Gr_p(\bbF^{p+q})$ is described by: 
     Theorem \ref{gen rel basis SOk+m} if $\bbF=\bbR$;
    Theorem \ref{gen rel basis} if $\bbF=\bbC$;
     Theorem \ref{gen rel basis B/BxB} if $\bbF=\bbH$.
\end{cor}

Since the involution $\sigma\theta$ of $\bbG$ is given by \eqref{sigma theta}, i.e., by $\sigma\theta(g)=I_{p,q}(\bar g^t)^{-1}I_{p,q}$, 
the group $G_\bbR=\bbG^{\sigma\theta}$ is equal to $\UU(p,q;\bbF)$. In other words, $G_\bbR$ is $\OO(p,q)$ if $\bbF=\bbR$; $\UU(p,q)$ if $\bbF=\bbC$; and $\Sp(p,q)$ if $\bbF=\bbH$.

\subsection{The symplectic Lagrangian Grassmannians}
\label{real lagr grass}

Let $\bbF=\bbR$ or $\bbC$, and let  $\LGr(\bbR^{2n})$ be the (symplectic) Lagrangian Grassmannian, i.e., the manifold of all Lagrangian subspaces of $\bbF^{2n}$ with respect to the standard symplectic form $\lara$ given by
\[
\lan x,y\ran=x^t J_n y,\qquad\text{where}\quad J_n=\begin{pmatrix} 0&I_n\cr -I_n&0\end{pmatrix}.
\]
Let $\bbG$ be the group $\Sp(2n,\bbF)$ of $2n\times 2n$ matrices over $\bbF$ preserving the form $\lara$, i.e., satisfying $g^tJ_ng=J_n$. Then $\bbG$ acts on $\LGr(\bbF^{2n})$, and this action is transitive by Witt's Theorem \ref{witt}. Thus $\LGr(\bbF^{2n})=\bbG/\bbP$, where 
$\bbP$ is the (Siegel) parabolic subgroup of $\Sp(2n,\bbF)$, defined as the stabilizer of the standard Lagrangian subspace 
\[
L_0=\{(x_1,\dots,x_n,0,\dots,0)\bbar x_1,\dots,x_n\in\bbF\} \subseteq \bbF^{2n}.
\]
Writing $g\in \bbG$ as a block matrix $\left(\begin{smallmatrix} A&B\cr C& D \end{smallmatrix}\right)$ with $n\times n$ blocks, the condition $g^tJ_ng=J_n$ implies
\eq
\label{cond sp2nR}
\bbG=\Sp(2n,\bbF)=\left\{\pmat A&B\cr C& D\epmat\in \GL(2n,\bbF)\bbar A^t C=C^tA,\ B^t D=D^t B,\ A^tD-C^t B=I_n\right\}.
\eeq
It follows that
\eq
\label{cond P sp2nR}
\bbP=\left\{\pmat A&B\cr 0& D\epmat\in \GL(2n,\bbF)\bbar B^t D=D^t B,\ A^tD=I_n\right\}.
\eeq

Let $\sigma$ be the involution of $\bbG$ given by \eqref{def sigma}, i.e., by $\sigma(g)=I_{n,n}gI_{n,n}$. Then $\bbG^\sigma$ is equal to the Levi subgroup $\bbL$ of $\bbP$ and it consists of block diagonal matrices in $\bbG$, i.e., 
\[
\bbG^\sigma=\bbL=\left\{\pmat A&0\cr 0& (A^t)^{-1}\epmat \bbar A\in \GL(n,\bbF)\right\}\cong \GL(n,\bbF)\quad\text{via}\ \pmat A&0\cr 0& (A^t)^{-1}\epmat\leftrightarrow A.
\]
Since $\bbK$ consists of the fixed points of the Cartan involution $\theta (g)=(\bar g^t)^{-1}$, and since $g\in\bbG$ is equivalent to $(g^t)^{-1}=J_n gJ_n^{-1}$, we see that $\theta(g)=g$ is equivalent to $J_n\bar g=gJ_n$. This implies
\eq
\label{descr K spn}
\bbK=\left\{\begin{pmatrix} A&-\bar C\cr C& \bar A \end{pmatrix}\bbar A^tC=C^tA,\ \bar A^tA+\bar C^tC=I_n\right\}.
\eeq
If $\bbF=\bbR$ (so the bars can be omitted), this is exactly the standard description of $\UU(n)$ inside $\GL(2n,\bbR)$; more precisely, $\bbK\cong \UU(n)$ via 
$\smat A&-C\cr C& A \esmat\leftrightarrow A+iC$. 

If $\bbF=\bbC$, then \eqref{descr K spn} is exactly the standard description of $\Sp(n)$ inside $\GL(2n,\bbC)$.  

We now also see that 
\[
\bbP\cap \bbK= \bbK^\sigma = \left\{\begin{pmatrix} A&0\cr 0& \bar A \end{pmatrix}\bbar A^tA=I\right\}\cong \UU(n,\bbF)\quad\text{via}\ \begin{pmatrix} A&0\cr 0& \bar A \end{pmatrix}\leftrightarrow A. 
\]
In other words, if $\bbF=\bbR$, then $\bbK^\sigma =\OO(n)$ via $\smat A&0\cr 0& A \esmat\leftrightarrow A$, and if $\bbF=\bbC$, then $\bbK^\sigma =\UU(n)$ via $\smat A&0\cr 0& \bar A \esmat\leftrightarrow A$.

Now Proposition \ref{k trans} implies 
\begin{prop}
\label{prop real lg}
The real symplectic Lagrangian Grassmannian $\LGr(\bbR^{2n})$ is diffeomorphic to $\UU(n)/\OO(n)$, while the complex symplectic Lagrangian Grassmannian $\LGr(\bbC^{2n})$ is diffeomorphic to $\Sp(n)/\UU(n)$.
\end{prop}

\begin{cor}
\label{coho lagr grass}
The cohomology ring (with complex coefficients) of the symplectic  Lagrangian Grassmannian $\LGr(\bbF^{2n})$ is described by: Theorem \ref{thm prim and aprim alg} if $\bbF=\bbR$; Theorem \ref{thm basis C/A} if $\bbF=\bbC$.

\end{cor}

Finally, we describe the group $\bbG^{\sigma\theta}$. Let first $\bbF=\bbR$. Then since $\sigma\theta(g)=I_{n,n}(g^t)^{-1}I_{n,n}$, and since any $g\in\bbG$ satisfies $(g^t)^{-1}=J_ngJ_n^{-1}$, $\sigma\theta(g)=g$ is equivalent to 
$gD_n=D_n g$ where $D_n=\smat 0&I_n\cr I_n&0\esmat$. It follows that 
\[
\bbG^{\sigma\theta}=\left\{g=\pmat A&B\cr B&A\epmat\bbar A^tB=B^tA,\,A^tA- B^tB=I_n \right\}.
\]
Conjugating $\smat A&B\cr B&A\esmat$ by $\smat I_n&I_n\cr I_n&-I_n\esmat$ we get the matrix $\smat A+B&0\cr 0&A-B\esmat$, and the conditions $A^tB=B^tA,\, A^tA- B^tB=I_n$ imply $(A+B)^t(A-B)=I_n$, so $A-B=((A+B)^t)^{-1}$. Conversely, starting from the matrix $\smat Z&0\cr 0& (Z^t)^{-1}\esmat$ and setting $A=\half(Z+(Z^t)^{-1})$, $B=\half (Z-(Z^t)^{-1})$, we get $A^tB= B^tA,\,A^tA-B^tB=I_n$. Thus
\[
\bbG^{\sigma\theta}\cong \left\{\pmat Z&0\cr 0&(Z^t)^{-1}\epmat\bbar Z\in \GL(n,\bbR)\right\} \cong \GL(n,\bbR) \ 
\text{via}\ \pmat Z&0\cr 0&(Z^t)^{-1}\epmat\leftrightarrow Z.
\]

Now let $\bbF=\bbC$. Since  
$\sigma\theta(g)=I_{n,n}(\bar g^t)^{-1}I_{n,n}$, and since any $g\in \bbG$ satisfies $(g^t)^{-1}=J_ngJ_n^{-1}$, we see that 
$\sigma\theta(g)=D_n\bar g D_n$. 
We claim that $\bbG^{\sigma\theta}\cong \Sp(2n,\bbR)$. To see this, we note that
\eq
\label{sqrt Dn}
C_n=\half\begin{pmatrix}
    (1+i)I_n& (1-i)I_n\cr (1-i)I_n&(1+i)I_n
\end{pmatrix}
\qquad\text{implies}\quad C_n^2=D_n\text{ and } C_n^{-1}=\bar C_n.
\eeq
This implies that 
\[
g\in\bbG^{\sigma\theta} \qquad\text{if and only if}\qquad \overline{C_ngC_n^{-1}}=C_n g C_n^{-1},
\]
i.e., that $\bbG^{\sigma\theta}=C_n^{-1} \Sp(2n,\bbR)C_n\cong \Sp(2n,\bbR)$.

\subsection{Orthogonal Lagrangian Grassmannians}
\label{real orth lagr grass}

Let $\OLGr(\bbC^{2n})$ be the (complex) orthogonal  Lagrangian Grassmannian. In other words, $\OLGr(\bbC^{2n})$ is the manifold of all Lagrangian subspaces of $\bbC^{2n}$ with respect to the symmetric bilinear form $\lara$ defined by 
\[
\lan x,y\ran=\sum_{r=1}^n x_ry_{n+r} +\sum_{r=1}^n x_{n+r}y_r=x^t D_n y,
\]
where as before, $D_n=\smat 0&I_n\cr I_n&0\esmat$.

Let $\bbG=\OO(2n,\bbC)$ be the group of $2n\times 2n$ complex matrices  preserving the form $\lara$, i.e., satisfying $g^tD_ng=D_n$. Writing $g=\smat A&B\cr C&D\esmat$ we see that 
\eq
\label{cond onn}
\bbG=\OO(2n,\bbC)=\left\{\pmat A&B\cr C&D \epmat\in \GL(2n,\bbC)\bbar C^tA=-A^tC,\, D^tB=- B^tD,\, A^tD+C^tB=I_n\right\}.
\eeq

The group $\bbG$ acts on $\OLGr(\bbC^{2n})$, and this action is transitive by Witt's Theorem \ref{witt}. Thus $\OLGr(\bbC^{2n})=\bbG/\bbP$, where 
$\bbP$ is the parabolic subgroup of $\bbG$, defined as the stabilizer of the standard Lagrangian subspace 
\[
L_0=\{(x_1,\dots,x_n,0,\dots,0)\bbar x_1,\dots,x_n\in\bbC\}\subset \bbC^{2n}.
\]
An element $g=\smat A&B\cr C&D\esmat$ of $\bbG$ stabilizes $L_0$ if and only if $C=0$, so
\eq
\label{descr P o2nc}
\bbP=\left\{\pmat A&B\cr 0&D\epmat\in \GL(2n,\bbC)\bbar   D^tB=- B^tD,\,  A^tD=I_n\right\}.
\eeq

Let $\sigma$ be the involution of $\bbG$ defined by \eqref{def sigma}, i.e., by $\sigma(g)=I_{n,n}gI_{n,n}$. Then $\bbG^\sigma$ is equal to the Levi subgroup $\bbL$ of $\bbP$ and it consists of block diagonal matrices in $\bbG$, i.e., 
\[
\bbG^\sigma=\bbL=\left\{\pmat A&0\cr 0& (A^t)^{-1}\epmat \bbar A\in \GL(n,\bbC)\right\}\cong \GL(n,\bbC)\quad\text{via}\ \pmat A&0\cr 0& (A^t)^{-1}\epmat\leftrightarrow A.
\]
Since $\theta(g)=(\bar g^t)^{-1}$ and since $g\in\bbG$ is equivalent to $(g^t)^{-1}=D_ngD_n$, we see that $\theta(g)=g$ is equivalent to $D_n\bar gD_n=g$, or $D_n\bar g=gD_n$. Thus  
\eq
\label{descr K onn}
\bbK=\left\{ \pmat A& \bar C\cr C& \bar A\epmat\in \GL(2n,\bbC)\bbar  C^tA=- A^tC,\  \bar A^tA+ \bar C^tC=I_n\right\}.
\eeq
To identify this subgroup, we connect our $\bbG$ with the more usual group $\bbG'=\OO(2n,\bbC)'$ given by $g^tg=I_{2n}$. The maximal compact subgroup $\bbK'$ of $\bbG'$ is given by the condition $(\bar g^t)^{-1}=g$, or equivalently $\bar g=g$, so $\bbK'=\OO(2n)$. Since $\bbG$ and $\bbG'$ are isomorphic, $\bbK\cong \OO(2n)$. In fact, explicit isomorphisms $\bbG\cong\bbG'$ and $\bbK\cong\bbK'$ are given by conjugation by the matrix $C_n$ of \eqref{sqrt Dn}.

We also see that
\[
\bbP\cap \bbK= \bbK^\sigma = \left\{ \pmat A& 0\cr 0&\bar A\epmat\bbar \bar A^tA=I_n\right\}\cong \UU(n)\quad\text{via}\quad \pmat A& 0\cr 0&\bar A\epmat\leftrightarrow A.
\]

Now Proposition \ref{k trans} implies

\begin{prop}
\label{prop onn}
The orthogonal Lagrangian Grassmannian $\OLGr(\bbC^{2n})$ is diffeomorphic to $\OO(2n)/\UU(n)$.
\end{prop}

Since our computation of cohomology of a compact symmetric space $G/K$ requires $G$ to be connected, we replace $\OO(2n)/\UU(n)$ by $\SO(2n)/\UU(n)$. The orbit of $\SO(2n)$ on $\OLGr(\bbC^{2n})$ is one of the two components of $\OLGr(\bbC^{2n})$, which we denote by $\OLGr^+(\bbC^{2n})$, and still call it the orthogonal Lagrangian Grassmannian. The other component of $\OLGr(\bbC^{2n})$ is diffeomorphic to $\OLGr^+(\bbC^{2n})$ and thus has the same cohomology.

\begin{cor}
\label{coho onn}
The cohomology ring (with complex coefficients) of the orthogonal Lagrangian Grassmannian $\OLGr^+(\bbC^{2n})\cong \SO(2n)/\UU(n)$ is described by 
Theorem \ref{thm basis D/A}.
\end{cor}

Finally, we describe the group $G_\bbR=\bbG^{\sigma\theta}$ in case 
$\bbG=\SO(2n,\bbC)$. Since 
$\sigma\theta(g)=I_{n,n}(\bar g^t)^{-1} I_{n,n}$, we see that 
$\bbG^{\sigma\theta}=\OO(2n,\bbC)\cap\UU(n,n)=\SO(2n,\bbC)\cap \SU(n,n)$, and this is exactly the description of $\SO^*(2n)$ given in \cite[1.141]{beyond}. 

\begin{rem}
    One could define the group $\OO^*(2n)$ as $\OO(2n,\bbC)\cap U(n,n)$, or alternatively, as the group of automorphisms of $\bbH^n$ preserving a skew Hermitian form; see Section \ref{quat lagr grass}. Conceivably, an element of this group could have determinant equal to $\pm 1$. However, we prove in Section \ref{quat lagr grass} that the maximal compact subgroup of this group is $\UU(n)$, so it follows that the group is connected and the determinant must be 1. In other words, $\OO^*(2n)=\SO^*(2n)$.
\end{rem}

\subsection{The Hermitian Lagrangian Grassmannians}
\label{herm lagr grass}

Let $\bbF$ be $\bbR$, $\bbC$ or $\bbH$ and let $\HLGr(\bbF^{2n})$ be the Hermitian   Lagrangian Grassmannian. In other words, $\HLGr(\bbF^{2n})$ is the manifold of all Lagrangian subspaces of $\bbF^{2n}$ with respect to the Hermitian form $\lara$ of signature $(n,n)$, defined by 
\[
\lan x,y\ran=\sum_{r=1}^n \bar x_ry_{n+r} +\sum_{r=1}^n \bar x_{n+r}y_r=\bar x^t D_n y,
\]
where as before, $D_n=\smat 0&I_n\cr I_n&0\esmat$.

Let $\bbG=\UU(n,n;\bbF)$ be the group of $2n\times 2n$ matrices over $\bbF$ preserving the form $\lara$, i.e., satisfying $\bar g^tD_ng=D_n$. So if $\bbF=\bbR$, $\bbG=\OO(n,n)$; if $\bbF=\bbC$, $\bbG=\UU(n,n)$; and if $\bbF=\bbH$, $\bbG=\Sp(n,n)$.

Writing $g=\smat A&B\cr C&D\esmat$ with $n\times n$ blocks, we see that 
\eq
\label{cond unn}
\bbG=\UU(n,n;\bbF)=\left\{\pmat A&B\cr C&D \epmat\in \GL(2n,\bbF)\bbar \bar C^tA=-\bar A^tC,\, \bar D^tB=-\bar B^tD,\, \bar A^tD+\bar C^tB=I_n\right\}.
\eeq
The group $\bbG$ acts on $\HLGr(\bbF^{2n})$, and this action is transitive by Witt's Theorem \ref{witt}.
Thus $\HLGr(\bbF^{2n})=\bbG/\bbP$, where 
$\bbP$ is the parabolic subgroup of $\bbG$ defined as the stabilizer of the standard Lagrangian subspace 
\[
L_0=\{(x_1,\dots,x_n,0,\dots,0)\bbar x_1,\dots,x_n\in\bbF\}\subset\bbF^{2n}.
\]
It follows that
\eq
\label{descr P unn}
\bbP=\left\{\pmat A&B\cr 0&D\epmat\in \GL(2n,\bbF)\bbar  \bar D^tB=-\bar B^tD,\, \bar A^tD=I_n\right\}.
\eeq

Let $\sigma$ be an involution of $\bbG$ defined by \eqref{def sigma}, i.e., by $\sigma(g)=I_{n,n}gI_{n,n}$. Then $\bbG^\sigma$ is equal to the Levi subgroup $\bbL$ of $\bbP$ and it consists of block diagonal matrices in $\bbG$, i.e., 
\[
\bbG^\sigma=\bbL=\left\{\pmat A&0\cr 0& (\bar A^t)^{-1}\epmat \bbar A\in \GL(n,\bbF)\right\}\cong \GL(n,\bbF)\quad\text{via}\ \pmat A&0\cr 0& (\bar A^t)^{-1}\epmat\leftrightarrow A.
\]
Since the Cartan involution is $\theta(g)=(\bar g^t)^{-1}$ and since $g\in\bbG$ is equivalent to $(\bar g^t)^{-1}=D_ngD_n$, we see that $\theta(g)=g$ is equivalent to $D_ngD_n=g$, or $D_ng=gD_n$. It follows that 
\eq
\label{descr K unn}
\bbK=\left\{ \pmat A& B\cr B& A\epmat\in \GL(2n,\bbF)\bbar \bar B^tA=-\bar A^tB,\ \bar A^tA+\bar B^tB=I_n\right\}.
\eeq
To identify this subgroup, we conjugate the matrix $\smat A&B\cr B&A\esmat$ by $\smat I_n&I_n\cr I_n&-I_n\esmat$ and get the matrix $\smat A+B&0\cr 0&A-B\esmat$. The conditions $\bar B^tA=- \bar A^tB,\, \bar A^tA+\bar B^tB=I_n$ imply $\overline{(A+B)}^t(A+B)=I_n$ and $\overline{(A-B)}^t(A-B)=I_n$, so $A+B$ and $A-B$ are in $\UU(n,\bbF)$ (i.e., in $\OO(n)$ if $\bbF=\bbR$; in $\UU(n)$ if $\bbF=\bbC$; and in $\Sp(n)$ if $\bbF=\bbH$). Conversely, starting from matrices $Z$ and $W$ in $\UU(n,\bbF)$, we can reconstruct $A$ and $B$ as $A=\half(Z+W)$, $B=\half (Z-W)$, and the matrix $\smat A&B\cr B&A\esmat$ will satisfy the conditions $\bar B^tA=- \bar A^tB,\, \bar A^tA+\bar B^tB=I_n$. 
So we found an explicit isomorphism $\bbK\cong \UU(n,\bbF)\times \UU(n,\bbF)$. 
Under this isomorphism, the subgroup $\bbP\cap\bbK=\bbK^\sigma$ corresponds to the diagonal $\Delta \UU(n,\bbF)\subset \UU(n,\bbF)\times \UU(n,\bbF)$.

Now Proposition \ref{k trans} implies
\begin{prop}
\label{prop unn}
The Hermitian Lagrangian Grassmannian $\HLGr(\bbF^{2n})$ is diffeomorphic to 
$\UU(n,\bbF)\times \UU(n,\bbF)/\Delta \UU(n,\bbF)$. In other words, $\HLGr(\bbF^{2n})$ is diffeomorphic to
$\OO(n)\times \OO(n)/\Delta \OO(n)$ if $\bbF=\bbR$; to $\UU(n)\times \UU(n)/\Delta \UU(n)$ if $\bbF=\bbC$; and to $\Sp(n)\times \Sp(n)/\Delta \Sp(n)$ if $\bbF=\bbH$ .
\end{prop}

To compute the cohomology of $G/K$ we need $G$ to be connected, and in case $\bbF=\bbR$ the group $G=\bbK=\OO(n)\times \OO(n)$ is not connected. Thus in the real case we replace $\OO(n)\times \OO(n)/\Delta \OO(n)$ by $\SO(n)\times \SO(n)/\Delta \SO(n)$. This amounts to replacing $\HLGr(\bbR^{2n})$ by the orbit $\HLGr^+(\bbR^{2n})$ of $\SO(n)\times \SO(n)$ which is one of the two
connected components of $\HLGr(\bbR^{2n})$.

\begin{cor}
\label{coho unn}
The cohomology rings (with complex coefficients) of the Hermitian Lagrangian Grassmannians $\HLGr^+(\bbR^{2n})$, $\HLGr(\bbC^{2n})$, and $\HLGr(\bbH^{2n})$, are described by  Theorem \ref{gen rel basis group}.
\end{cor}

Finally, we describe the group $G_\bbR=\bbG^{\sigma\theta}$. Since 
$\sigma\theta(g)=I_{n,n}(\bar g^t)^{-1} I_{n,n}$ and since any $g\in\bbG$ satisfies $(\bar g^t)^{-1}=D_ngD_n$, we see that 
\[
\sigma\theta(g)=I_{n,n}D_ngD_nI_{n,n}=J_ngJ_n^{-1}.
\]
Thus $\sigma\theta(g)=g$ is equivalent to $J_ng=gJ_n$, and it follows that  
\eq
\label{cond gst}
\bbG^{\sigma\theta}=\left\{\pmat A&-C\cr C&A\epmat\bbar \bar C^tA=-\bar A^tC,\,\bar A^tA-\bar C^tC=I_n\right\}.
\eeq 

If $\bbF=\bbR$, recall that $A+iC\mapsto\smat A&-C\cr C&A\esmat$ is the standard embedding of $\GL(n,\bbC)$ into $\GL(2n,\bbR)$, and note that the conditions $C^tA=-A^tC,\,A^tA-C^tC=I_n$ correspond to $(A+iC)^t(A+iC)=I_n$. It follows that $\bbG^{\sigma\theta}=\OO(n,\bbC)$.

If $\bbF=\bbC$, we conjugate $\smat A&-C\cr C&A\esmat$ by $\smat 1&i\cr 1&-i\esmat$ and get $\smat A+iC&0\cr 0&A-iC\esmat$. The conditions $\bar C^tA=-\bar A^tC,\,\bar A^tA-\bar C^tC=I_n$ imply $\overline{(A+iC})^t(A-iC)=I_n$. Conversely, given $\smat Z&0\cr 0&(\bar Z^t)^{-1}\esmat$ we can reconstruct $A$ and $C$ as $A=\half(Z+(\bar Z^t)^{-1})$ and $C=\frac{1}{2i}(Z-(\bar Z^t)^{-1})$ and get 
$\smat A&-C\cr C&A\esmat\in\bbG^\sigma$. It follows that
\[
\bbG^{\sigma\theta}\cong \GL(n,\bbC), \quad \text{via}\ \pmat A&-C\cr C&A\epmat\leftrightarrow \pmat A+iC&0\cr 0&A-iC\epmat\leftrightarrow A+iC.
\]

If $\bbF=\bbH$, we claim that $\bbG^{\sigma\theta}$ is isomorphic to $\Sp(2n,\bbC)$. To see this, we consider the map
\[
\pmat A&-C\cr C&A\epmat =\pmat A_1+jA_2&-C_1-jC_2\cr C_1+jC_2&A_1+jA_2\epmat\mapsto \pmat A_1+iC_1&-\bar A_2-i\bar C_2\cr A_2+iC_2&\bar A_1+i\bar C_1\epmat.
\]
It is a tedious but straightforward  computation to check that the conditions \eqref{cond gst} imply the conditions in 
\eqref{cond sp2nR}, so our map sends $\bbG^{\sigma\theta}$ into $\Sp(2n,\bbC)$. Conversely, if $\smat X&Y\cr Z&T\esmat\in\Sp(2n,\bbC)$, then we can reconstruct an element of $\bbG^{\sigma\theta}$ mapping to $\smat X&Y\cr Z&T\esmat$; it is given by
\[
A_1=\half(X+\bar T),\quad A_2=\half(Z-\bar Y),\quad C_1=\frac{1}{2i}(X-\bar T),\quad C_2=\frac{1}{2i}(Z+\bar Y).
\]
(Another tedious computation shows that this element does satisfy the conditions of \eqref{cond gst}.)

\subsection{The skew Hermitian quaternionic Lagrangian Grassmannians}
\label{quat lagr grass}
Consider the skew Hermitian form on $\bbH^{2n}$ given by
\eq
\label{skew herm form}
(x\sbar y)=\bar x^t J_n y=\sum_{r=1}^{n} \bar x_r y_{n+r}-\sum_{r=1}^{n} \bar x_{n+r}y_r,
\eeq
where as before, $J_n=\smat 0&I_n\cr -I_n&0\esmat$.
(Note that the form $(\sbar)$ is different from, but equivalent to, the form considered in \cite{rossmann}. Also, recall that $\bbH$ acts on $\bbH^{2n}$ by right scalar multiplication, and that the form $(\sbar)$ satisfies the condition $(x\alpha\sbar y\beta)=\bar\alpha(x\sbar y)\beta$ for $x,y\in\bbH^{2n}$ and $\alpha,\beta\in\bbH$.)

The group $\bbG=\SO^*(4n)$ is the group of automorphisms of $\bbH^{2n}$ preserving the form $(\sbar)$, i.e., 
\[
\bbG=\left\{g\in \GL(2n,\bbH)\bbar \bar g^tJ_n g=J_n \right\}.
\]
The operations bar and transpose are defined by passing to the complex matrices: if $X$ is any $2n\times 2n$ quaternionic matrix, we write it as $X=U+jV$ with $U,V$ complex and identify $X$ with the $4n\times 4n$ complex matrix $\smat U&-\bar V\cr V&\bar U\esmat$. Then 
\[
 \bar X= \overline{\pmat U&-\bar V\cr V&\bar U\epmat}=\bar U+j\bar V;\  X^t=\pmat U&-\bar V\cr V&\bar U\epmat^t=U^t-j\bar V^t.
\]
Then one has $\overline{XY}=\bar X\bar Y$ and $(XY)^t=Y^tX^t$. 
The reader is cautioned that the operations bar and transpose cannot be performed directly on the quaternionic matrix in the usual way, but their composition can, since
\[
\overline{(U+jV)}^t=(\bar U+j\bar V)^t=\bar U^t-jV^t=\bar U^t +\bar V^t\bar j.
\]

Upon writing $g\in\bbG$ as $\smat A&B\cr C&D\esmat$ with $n\times n$ (quaternionic) blocks and writing out the condition $\bar g^tJ_n g=J_n$, we see
\eq
\label{cond so*}
\bbG=\SO^*(4n)=\left\{\pmat A&B\cr C&D\epmat\bbar \bar A^tC=\bar C^tA,\,\bar B^tD=\bar D^tB,\,\bar A^tD-\bar C^t B=I_n \right\}.
\eeq

Clearly, $\bbG$ acts on the skew Hermitian quaternionic Lagrangian Grassmannian $\LGr^*(\bbH^{2n})$, the manifold of all Lagrangian subspaces of $\bbH^{2n}$ with respect to $(\sbar)$, and this action is transitive by Witt's Theorem \ref{witt}. Thus 
$\LGr^*(\bbH^{2n})=\bbG/\bbP$, where 
$\bbP$ is the stabilizer in $\bbG$ of the standard Lagrangian subspace
\[
L_0=\{(x_1,\dots,x_n,0,\dots,0)\bbar x_1,\dots,x_n\in\bbH\}\subset \bbH^{2n}.
\]
It follows that 
\eq
\label{cond P so*}
\bbP=\left\{\pmat A&B\cr 0&D\epmat\in \GL(2n,\bbH)\bbar \bar B^tD=\bar D^tB,\,\bar A^tD=I_n \right\}.
\eeq

Let $\sigma$ be an involution of $\bbG$ given by \eqref{def sigma}, i.e., by $\sigma(g)=I_{n,n}gI_{n,n}$. Then $\bbG^\sigma$ is equal to the Levi subgroup $\bbL$ of $\bbP$ and it consists of block diagonal matrices in $\bbG$, i.e., 
\[
\bbG^\sigma=\bbL=\left\{\pmat A&0\cr 0& (\bar A^t)^{-1}\epmat \bbar A\in \GL(n,\bbH)\right\}\cong \GL(n,\bbH)\quad\text{via}\ \pmat A&0\cr 0& (\bar A^t)^{-1}\epmat\leftrightarrow A.
\]
Since $\bbK$ consists of the fixed points of the Cartan involution $\theta(g)=(\bar g^t)^{-1}$, and since 
 $g\in\bbG$ is equivalent to  
$(\bar g^t)^{-1}=J_ngJ_n^{-1}$, we see that $g\in\bbK$ if and only if $gJ_n=J_ng$. This implies 
\eq
\label{cond K so*}
\bbK=\left\{\pmat A&-C\cr C&A\epmat\in \GL(2n,\bbH)\bbar \bar A^tC=\bar C^tA,\,\bar A^tA+\bar C^tC=I_n \right\}.
\eeq
We now also see that
\[
\bbP\cap\bbK=\bbK^\sigma=\left\{\pmat A&0\cr 0&A\epmat\in \GL(2n,\bbH)\bbar \bar A^tA=I_n \right\}\cong \Sp(n)\quad\text{via}\ \pmat A&0\cr 0&A\epmat\leftrightarrow A.
\]
We claim that $\bbK\cong \UU(2n)$. To see this, we consider a different copy $\bbG'$ of $\SO^*(4n)$ inside $\GL(2n,\bbH)$, the one preserving the skew Hermitian form
\[
\lan x,y\ran = \bar x^t i y=\sum_{r=1}^{2n}\bar x_riy_r.
\]
So $\bbG'$ is the subgroup of $\GL(2n,\bbH)$ consisting of matrices $g$ such that $\bar g^tig=iI_{2n}$, and upon writing $g=U+jV$ with $U,V$ complex, we see
\eq
\label{cond so+ bis}
\bbG'=\left\{U+jV\in \GL(2n,\bbH)\bbar \bar U^tU+\bar V^tV=I_{2n},\, U^tV=V^t U\right\}.
\eeq
Since $g\in\bbG'$ is equivalent to $(\bar g^t)^{-1}=-igi$, $\theta g=g$ is equivalent to $ig=gi$. Writing $g=U+jV$ with $U,V$ complex, we see
\eq
\label{cond k so+ bis}
\bbK'=\left\{U+j0\in \GL(2n,\bbH)\bbar \bar U^tU=I_{2n} \right\}.
\eeq
 So $\bbK'$ is the usual $\UU(2n)$, embedded into $\GL(2n,\bbH)$ as $U(n)+j0$. 

To show an explicit connection between $\bbG$ and $\bbG'$ and also between $\bbK$ and $\bbK'$, we note that the matrix
\[
T_n=\frac{1}{\sqrt{2}}\pmat I_n&-iI_n\cr jI_n&-kI_n\epmat
\]
satisfies $\bar T_n^t i T_n=J_n$. It follows that 
\[
T_n\bbG T_n^{-1} = \bbG',
\]
and since $\theta T_n=T_n$, also $T_n\bbK T_n^{-1} = \bbK'=\UU(2n)$. Moreover, $T_n(\bbP\cap\bbK) T_n^{-1}$ is the standard $\Sp(n)$ inside $\UU(2n)$, embedded as matrices of the form $\smat U&-\bar V\cr V&\bar U\esmat$.

Now Proposition \ref{k trans} implies the following result, which can also be found in \cite{CN}. 

\begin{prop}
\label{prop u2n spn}
The skew Hermitian quaternionic Lagrangian Grassmannian $\bbG/\bbP=\LGr^*(\bbH^{2n})$ is diffeomorphic to $\UU(2n)/\Sp(n)$.
\end{prop}

\begin{cor}
\label{coho u2n spn}
The cohomology ring (with complex coefficients) of the skew Hermitian quaternionic Lagrangian Grassmannian $\LGr^*(\bbH^{2n})\cong \UU(2n)/\Sp(n)$ is described by Theorem \ref{gen rel basis A/C}.
\end{cor}

Finally, we describe the group $G_\bbR=\bbG^{\sigma\theta}$. Since $\sigma\theta(g)=I_{n,n}(\bar g^t)^{-1}I_{n,n}$ and since $g\in\bbG$ is equivalent to $(\bar g^t)^{-1}=J_ngJ_n^{-1}$, 
\[
\sigma\theta (g)=I_{n,n}J_ngJ_n^{-1}I_{n,n}=D_n g D_n,\qquad g\in\bbG.
\] 
Thus $\sigma\theta(g)=g$ is equivalent to $gD_n=D_ng$. It follows that
\[
\bbG^{\sigma\theta}=\left\{g=\pmat A&B\cr B&A\epmat\bbar \bar A^tB=\bar B^tA,\,\bar A^tA-\bar B^tB=I_n\right\}.
\]
Conjugating $\smat A&B\cr B&A\esmat$ by $\smat I_n&I_n\cr I_n&-I_n\esmat$ we get the matrix $\smat A+B&0\cr 0&A-B\esmat$, and the conditions $\bar A^tB=\bar B^tA,\,\bar A^tA-\bar B^tB=I_n$ imply $\overline{(A+B)^t}(A-B)=I_n$, so $A-B=(\overline{(A+B)^t})^{-1}$. Conversely, starting from the matrix $\smat Z&0\cr 0& (\bar Z^t)^{-1}\esmat$ and setting $A=\half(Z+(\bar Z^t)^{-1})$, $B=\half (Z-(\bar Z^t)^{-1})$, we get $\bar A^tB=\bar B^tA,\,\bar A^tA-\bar B^tB=I_n$. Thus
\[
\bbG^{\sigma\theta}\cong \left\{\pmat Z&0\cr 0&(\bar Z^t)^{-1}\epmat\bbar Z\in \GL(n,\bbH)\right\} \cong \GL(n,\bbH)\quad
\text{via}\ \pmat Z&0\cr 0&(\bar Z^t)^{-1}\epmat\leftrightarrow Z.
\]

\subsection{The quadric cases}
In this subsection $\mathbb{F}$ is equal to $\mathbb{R}$ or $\mathbb{C}$. The following is taken from \cite[Chapter 4.4]{Thor}

Let $f$ be a nondegenerate symmetric bilinear form on $\mathbb{F}^{n+2}$. The quadric $Q_f(\mathbb{F})$ is defined to be the subset of the projective space $P^{n+1}(\mathbb{F})$:
\[Q_f(\mathbb{F}) = \{ x = (x_1 : \ldots : x_{n+2}) \in P^{n+1}(\bbF) \:| \:f(x,x) =0 \}. \] 
If $\mathbb{F}=\mathbb{R}$ then $f$ has normal form with matrix $I_{p+1,q+1}$ 
(if the form is definite, $Q_f(\mathbb{R})$ does not contain projective lines) and we denote $Q_f(\mathbb{R})$ by $Q_{p,q}(\mathbb{R})$. If $\mathbb{F}=\mathbb{C}$ then all $f$ are equivalent and we denote $Q_f(\mathbb{C})$ by $Q_{n}(\mathbb{C})$. The group $\bbG=\SO(p+1,q+1)_e$ acts transitively on the quadric $Q_{p,q}(\mathbb{R})$ with parabolic $\bbP = \mathrm{Stab}(1:0: \ldots :1:0: \ldots:0 )$. The maximal compact subgroup of $\bbG$ is $G=\bbK= \SO(p+1) \times \SO(q+1)$ and $K =\bbK \cap \bbP = S(\OO(p) \times \OO(q))$. In this setting $Q_{p,q}(\bbR)$ is equal to the symmetric space 
\[ G/K =  \SO(p+1) \times \SO(q+1) /  S(\OO(p) \times \OO(q)), \qquad p+q > 2. \]
The symmetric space is not an irreducible symmetric space, it has a double cover by spheres $S^p \times S^q$, however it is indecomposable.
When $\mathbb{F}=\mathbb{C}$ then $\SO_{n+2}(\mathbb{C})$ acts transitively on $Q_{n+2}(\mathbb{C})$ and the parabolic $\bbP = \mathrm{Stab}(1:i: \ldots :0)$ has abelian unipotent radical, furthermore
\[Q_n(\bbC) =\SO_{n+2}(\mathbb{C}) /\mathrm{Stab}(1:i: \ldots :0) \simeq \SO(n+2)/\SO(n) \times \SO(2) \quad n \geq 3.  \]

 As a compact symmetric space, $Q_n(\bbC)$ coincides with $\SO(n + 2)/\SO(n) \times \SO(2)$, which is equal to a double cover of the Grassmannian of $2$-planes in $\bbR^{n+2}$.  We leave the calculation of the cohomology of double covers of Grassmannians and these quadrics to future work.

\section{The structure of $C(\frp)^K$ and $(\twedge\frp)^K$}
\label{sec spin} 

\subsection{The decomposition of the spin module}
\label{K dec general}

Let $G/K$ be a compact symmetric space corresponding to an involution $\sigma$ of $G$ and let $\frg=\frk\oplus\frp$ be the decomposition of the complexified Lie algebra $\frg$ of $G$ into eigenspaces of $\sigma$. In particular, $\frk$ is the complexified Lie algebra of $K$. We note that if $G_\bbR/K$ is the noncompact dual of $G/K$, then $\sigma$ corresponds to the Cartan involution of $G_\bbR$ 
($ \sigma $ coincides with $ \theta $ on $ G_{\bbR} $). 

We define the Clifford algebra $C(\frp)$ using the nondegenerate invariant symmetric bilinear form $B$ on $\frg$ obtained by extending the Killing form over the center of $\frg$. Alternatively, for matrix groups from our table, we can replace $B$ by the trace form $\tr XY$. Recall that $C(\frp)$ is the associative unital algebra generated by $\frp$, with relations $XY+YX=2B(X,Y)$, $X,Y\in\frp$.

Let $S$ be a spin module for $C(\frp)$. Recall that $S$ is constructed as follows. Let $\frp^+$ and $\frp^-$ be two maximal isotropic subspaces of $\frp$, dual under $B$. Let $S=\twedge\frp^+$, with elements of $\frp^+\subset C(\frp)$ acting by wedging, and elements of $\frp^-\subset C(\frp)$ acting by contracting.
If $\dim\frp$ is even, $\frp=\frp^+\oplus\frp^-$ and hence this determines $S$ completely. Moreover, $S$ is the only irreducible $C(\frp)$-module, and $C(\frp)=\End S$. If $\dim\frp$ is odd, then $\frp=\frp^+\oplus\frp^-\oplus\bbC Z$ where $Z$ is an element of $\frp$ not contained in $\frp^+\oplus\frp^-$, such that $B(Z,Z)=1$. Now we can make $Z$ act on $\twedge\frp^+$ in two different ways; it can act by $1$ on $\twedge^{\even}\frp^+$ and by $-1$ on $\twedge^{\odd}\frp^+$, or by
$-1$ on $\twedge^{\even}\frp^+$ and by $1$ on $\twedge^{\odd}\frp^+$. In this way we get two inequivalent $C(\frp)$-modules $S_1$ and $S_2$. These are the only irreducible $C(\frp)$-modules and $C(\frp)=\End S_1\oplus\End S_2$. In the following, $S$ denotes either one of these two modules. For more details about Clifford algebras, spin modules, and also pin and spin groups, see \cite[Ch.2]{HP}.

Since the pin group $\Pin(\frp)$ is contained in $C(\frp)$, the pin double cover $\Kt$ of $K$ acts on $S$. Recall that $\Kt$ is obtained from the following pullback diagram
\[
\begin{CD}
\Kt @>>> \Pin(\frp) \\
@VVV @VVV \\
K @>>> \OO(\frp)
\end{CD}
\]
where the map $K\to \OO(\frp)$ is given by the adjoint action of $K$ on $\frp$.

It now follows that 
\eq
\label{end k s}
C(\frp)^K=C(\frp)^{\Kt}=\left\{\begin{matrix} \End_{\Kt}S,\qquad\qquad\quad\ \quad \dim\frp\text{ even}\cr
\End_{\Kt}S_1\oplus\End_{\Kt}S_2,\quad \dim\frp\text{ odd}.
\end{matrix}\right.
\eeq

Since the algebra $C(\frp)^K$ and its graded version $(\twedge\frp)^K$ are of  our primary interest, we are led to study the $\Kt$-decomposition of $S$. We first study the decomposition of $S$ under the complexified Lie algebra $\frk$ of $\Kt$. This Lie algebra acts on $S$ through the map $\alpha:\frk\to C(\frp)$, which is defined as the action map
$\frk\to\frso(\frp)$ followed by the Chevalley map (i.e., the skew symmetrization) $\frso(\frp)\cong\twedge^2\frp\hookrightarrow C(\frp)$. Explicitly, if $b_i$ is a basis of $\frp$ with dual basis $d_i$ with respect to the form $B$, then
\eq
\label{alpha def}
\alpha(X)=\frac{1}{4}\sum_i[X,b_i]d_i,\quad X\in\frk.
\eeq
(See \cite[\S 2.3.3]{HP}; the difference in sign comes from using different conventions to define the Clifford algebra.)

Let $\frt_0$ be a Cartan subalgebra of the (real) Lie algebra $\frk_0$ of $K$ and let $\frt=(\frt_0)_\bbC$. Let $\Delta^+(\frg,\frt)\supseteq\Delta^+(\frk,\frt)$ be compatible choices of positive roots for $(\frg,\frt)$ respectively $(\frk,\frt)$. Let $\rho$ respectively $\rho_\frk$ be the corresponding half sums of positive roots. Let $W_\frg$ respectively $W_\frk$ be the Weyl groups of $\Delta(\frg,\frt)$ respectively $\Delta(\frk,\frt)$.

Let $W^1_{\frg,\frk}$ be the set of minimal length representatives of $W_{\frk}$-cosets in $W_{\frg}$. Alternatively,
\[
W^1_{\frg,\frk}=\{\sigma\in W_\frg\bbar \sigma\rho\text{ is $\frk$-dominant}\}.
\]
It is well known (\cite{Par}; see also \cite{HP}) that the decomposition of $S$ under the action of $\frk$ is given by
\eq
\label{decomp S k}
S=m\cdot\bigoplus_{\sigma\in W^1_{\frg,\frk}} E_{\sigma},
\eeq
where $E_\sigma$ denotes the irreducible finite-dimensional $\frk$-module with highest weight $\sigma\rho-\rho_\frk$, and
the multiplicity $m$ is equal to $2^{[\half\tiny{\dim}\fra]}$ where $\fra$ is the centralizer of $\frt$ in $\frp$ (so that $\frh=\frt\oplus\fra$ is a Cartan subalgebra of $\frg$). 

Since $m$ is exactly the dimension of the spin module for the Clifford algebra $C(\fra)$, \eqref{end k s} and \eqref{decomp S k} imply that
\eq
\label{abstract dec}
C(\frp)^\frk\cong C(\fra)\otimes\proj(S),
\eeq
where $\proj(S)$ is the algebra spanned by the $\frk$-equivariant projections $\pr_\sigma:S\to m\cdot E_\sigma$, $\sigma\in W^1_{\frg,\frk}$.

If $K$ is connected, then the adjoint action map maps $K$ into $\SO(\frp)$, so $\Kt$ is the spin double cover of $K$,
\[
\begin{CD}
\Kt @>>> \Spin(\frp) \\
@VVV @VVV \\
K @>>> \SO(\frp).
\end{CD}
\]
If the double covering map $\Kt\to K$ does not split, then $\Kt$ is connected 
and \eqref{decomp S k} gives a decomposition of $S$ with respect to $\Kt$. If the covering $\Kt\to K$ splits, then $\Kt=K\times\bbZ_2$,
where the generator $z$ of $\bbZ_2$ maps to $1\in K$ under the covering map. This implies that $z$ maps to the preimage in $\Spin(\frp)$ of $1\in \SO(\frp)$, that is to $\pm 1\in C(\frp)$.  
Thus $z$ acts by the scalar $1$ or $-1$ on $S$, in particular it preserves the decomposition \eqref{decomp S k}, and hence this decomposition is also a decomposition of the $\Kt$-module $S$ into irreducibles. To conclude:

\begin{prop}
\label{K-dec of S conn}
If $K$ is connected, then the $\Kt$-decomposition of $S$ into irreducibles is the same as the $\frk$-decomposition \eqref{decomp S k}.
\end{prop}

In general, the $\frk$-decomposition is the same as the decomposition under $\Kt_e$, the connected component of the identity in $\Kt$, but the $\Kt$-action may combine several irreducible $\frk$-modules into one irreducible $\Kt$-module. More precisely, the component group $\Kt/\Kt_e\cong K/K_e$ acts by permuting the components $E_\sigma$ of $S$, and the $E_\sigma$ combining to produce an irreducible $\Kt$-module belong to the same orbit of $K/K_e$.
(Recall that $K_e$ is a normal subgroup of $K$ and that $K/K_e$ is a finite group.) 
We will treat the case of disconnected $K$ in Subsections \ref{subsec so/soo} and \ref{subsec u/o} below. Before that we describe the structure of $C(\frp)^\frk$ more precisely. 

\subsection{Top degree element and Poincar\'e duality}
\label{top elt}
We identify $C(\frp)$ and $\twedge\frp$ using the Chevalley map, and think of them as one vector space with two different multiplications.

\begin{prop}
\label{hodge *}
Let $T$ be the unique (up to a scalar multiple) element of the top wedge of $\frp$; let $d=\dim\frp$ denote the degree of $T$.

{\rm (1)} $T$ squares to a nonzero constant  with respect to Clifford multiplication; consequently we can rescale $T$ and assume that $T^2=1$.

{\rm (2)} If $d$ is odd,
$T$ is in the center of $C(\frp)$. If $d$ is even, $T$ commutes with $C(\frp)_{\even}$.

{\rm (3)} $T$ is $\frk$-invariant (with respect to the adjoint action).

{\rm (4)} Clifford multiplication by $T$ from the left is a linear isomorphism from $\twedge^j\frp$ to $\twedge^{d-j}\frp$, for any $j=0,1,\dots,d$. This isomorphism, denoted by $*$, preserves the $\frk$-invariants and therefore gives an isomorphism from $(\twedge^j\frp)^\frk$ to $(\twedge^{d-j}\frp)^\frk$ for any $j$.

{\rm (5)} The isomorphism $*$ can up to sign be expressed as $x\mapsto \iota_xT$.

{\rm (6)} For any $x,y\in \twedge^j\frp$,
\[
x\wedge *y = B(x,y)T.
\]
In other words, $*$ is the usual Hodge star operator.
\end{prop}

\pf Let $Z_1,\dots,Z_d$ be an orthonormal basis of $\frp$. Then, up to a nonzero scalar, $T=Z_1\cdots Z_d$. It follows that $T^2$ is a nonzero constant, since $Z_1\cdots Z_d$ squares to $\pm 1$. This proves (1). (2) is a straightforward computation: one checks that $Z_1\cdots Z_d$ commutes with all $Z_j$ if $ d $ is odd, and with all $Z_jZ_k$ if $d$ is even.

To prove (3), we note that the adjoint action of $X\in\frk$ on $C(\frp)$ is the same as the Clifford commutator with $\alpha(X)$. Since $\alpha$ maps $\frk$ into $C(\frp)_{\even}$, the claim follows from (2). 

To prove (4), we note that if we set $Z_I=Z_{i_1}\dots Z_{i_a}$ for 
$I=\{i_1,\dots,i_a\}\subseteq\{1,\dots,d\},$
then $T Z_I=\pm Z_{I^c},$
where $I^c=\{1,\dots,d\}\setminus I$. This implies (4). (5)  follows from the fact that Clifford multiplication by $y\in\frp$ equals $\iota_y+\eps_y$ where $\eps_y$ denotes wedging by $y$. Since $T$ is of top degree, it is annihilated by all $\eps_y$, $y\in\frp$, and this implies the claim.

To prove (6), we note that both sides of the equation are bilinear in $x$ and $y$, so we can assume $x=Z_I$, $y=Z_J$ for some $I,J\subseteq\{1,\dots,d\}$. If $I\neq J$, both sides of the equation are zero. Finally, if $I=J$, then we are to check that $Z_I\wedge *Z_I=T$, which is a straightforward computation.
\epf

\begin{lem}
    \label{T cliff}
    Let $x$ be any element of $C(\frp)$ such that $x^2=1$ (with respect to Clifford multiplication). Then $B(x,x)=1$, where $B$ is the extended Killing form on $C(\frp)\cong\twedge\frp$. 
    Consequently the elements $1$ and $x$ span a subalgebra of $C(\frp)$ isomorphic to the Clifford algebra on the one-dimensional space $\bbC x$.
    
    In particular, if $T$ is the top element of $C(\frp)$ as above, rescaled so that $T^2=1$, then $B(T,T)=1$ and $\vspan_\bbC(1,T)$ is a subalgebra of $C(\frp)$ is isomorphic to the Clifford algebra $C(\bbC T)$. 
\end{lem}

\pf This follows from the fact that the constant term of $x^2$ is $\iota_x x$, so
\[
1=B(1,1)=B(\iota_x x,1)=B(x,x\wedge 1)=B(x,x).
\]
\epf

\subsection{$C(\frp)^\frk$ and $(\twedge\frp)^\frk$ in the equal rank cases}
\label{cpk eq rk}

The equal rank cases on our list are:
\[
\begin{aligned}
& G/K=\UU(p+q)/\UU(p)\times\UU(q)\quad &\text{ (Subsection \ref{real gras})};\\ 
& G/K=\Sp(p+q)/\Sp(p)\times\Sp(q)\quad &\text{ (Subsection \ref{real gras})};\\ 
& G/K=\SO(2p+2q)/\mathrm{S}(\OO(2p)\times\OO(2q))\quad &\text{ (Subsection \ref{real gras})}; \\ 
& G/K=\SO(2p+2q+1)/\mathrm{S}(\OO(2p)\times\OO(2q+1))\quad &\text{ (Subsection \ref{real gras})}; \\ 
& G/K=\Sp(n)/\UU(n)  \quad &\text{ (Subsection \ref{real lagr grass})}; \\
& G/K=\SO(2n)/\UU(n)  \quad &\text{ (Subsection \ref{real orth lagr grass})}. 
\end{aligned}
\]
In each of these cases the situation is as in Kostant's email (see the introduction). In other words, the spin module $S$ is multiplicity free under $\frk$ and since $\dim\frp$ is even, $C(\frp)=\End S$. Therefore Schur's lemma implies that  $C(\frp)^\frk=\End_\frk S=\proj(S)$, the algebra spanned by the $\frk$-equivariant projections to the $\frk$-irreducible constituents of the spin module. The map $\alpha:\frk\to C(\frp)$ from \eqref{alpha def} extends to $U(\frk)$ and its restriction to the center $Z(\frk)$ of $U(\frk)$ is the algebra homomorphism
\[
\alpha_\frk:Z(\frk)\to C(\frp)^\frk.
\]
(The notation $\alpha_\frk$ is to distinguish this map from the analogous map $\alpha_K$ on the level of $K$-invariants; for connected $K$, there is no difference between these two maps.)

Since the $\frk$-infinitesimal character of $E_\sigma$ corresponds to $\sigma\rho$ under the Harish-Chandra isomorphism $Z(\frk)\cong\bbC[\frt^*]^{W_\frk}$, we can identify $\alpha_\frk$ with
\eq
\label{descr alpha k}
\alpha_\frk:\bbC[\frt^{*}]^{W_\frk}\to C(\frp)^\frk,\qquad \alpha_\frk(P)=\sum_{\sigma\in W^1_{\frg,\frk} }P(\sigma\rho) \pr_\sigma,
\eeq
where $\pr_\sigma$ denotes the $\frk$-equivariant projection $S\to E_\sigma$. 

\begin{prop} 
\label{prop al k}
The map $\alpha_\frk$ of \eqref{descr alpha k} is a filtered algebra homomorphism, which doubles the degree. Here the filtration on 
the algebra $\bbC[\frt^*]^{W_\frk}$ is induced by the grading, while the filtration on the algebra $C(\frp)^\frk$ is inherited from $C(\frp)$. 
\end{prop}
\pf
The claim follows from the fact that $\alpha_\frk$ is the restriction of $\alpha:U(\frk)\to C(\frp)$ given by extending \eqref{alpha def}.
\epf

In the next subsection we consider a more general setting. We will in particular prove that the map \eqref{descr alpha k} is onto; consequently, the map  $\gr\alpha_\frk:\bbC[\frt^*]^{W_\frk}\to (\twedge\frp)^\frk$  is also onto.
 Moreover, we will give a description of $\ker\alpha$ and of $\ker\gr\alpha=\gr\ker\alpha$. It is clear from \eqref{descr alpha k} that $\ker\alpha_\frk$ 
consists of polynomials vanishing at all $\sigma\rho$, $\sigma\in W^1_{\frg,\frk}$. 
Thus $\ker\alpha$ contains all $W_\frg$-invariant polynomials on $\frt^*$ that vanish at $\rho$ (and thus automatically on all $\sigma\rho$); we will prove that these polynomials in fact generate $\ker\alpha_\frk$. Likewise, we will see that $\ker\gr\alpha$ is generated by $W_\frg$-invariant polynomials on $\frt^*$ vanishing at $0$.

\subsection{Relative coinvariant algebra and filtered deformations}
\label{rel coinv}

Let $W$ be a finite group inside $\GL(\ft)$, with subgroup $H \subset W$. Let $\nu \in \ft^*$ be a point such that $\mathrm{Stab}_W(\nu) = \{ \Id \} $. 
 Let $\mathbb{C}[W]$ denote the algebra of functions from $W$ to $\mathbb{C}$ with pointwise multiplication. Give $\mathbb{C}[W]$  basis $\{ f_w: w \in W\}$, $f_w(w') = \delta_{w,w'}1$.

\begin{Def}

Define $\Ev_\nu: \bbC[\frt^*] \to \mathbb{C}[W]$ by 
\[ \Ev_\nu(p) = \sum_{w \in W} p(w \nu) f_w.\]

Restricting $\Ev_\nu$ to $\bbC[\frt^*]^H$ we define $\Ev_\nu^H: \bbC[\frt^*]^H \to \mathbb{C}[W]^H=\mathbb{C}[W/H]$.

\end{Def}

\begin{lem}
\label{ev surj}
The map $\Ev_\nu$ is a surjective $W$-module and algebra homomorphism and $\Ev_\nu^H$ is a surjective algebra homomorphism.
\end{lem}

\begin{proof}
Clearly $\Ev_\nu$ is a $W$ module homomorphism. Let $p_\nu$ be a linear polynomial that is zero on $\nu$ and non-zero on all $w \nu$, for all $w \neq 1.$ 
The polynomial $\prod_{w \in W \setminus \{ 1 \}} w p_\nu$ evaluated on $\nu$ is non-zero and is zero on every other element in the orbit of $\nu$. Suitably scaled,  $\Ev_\nu(\prod_{w \in W \setminus \{ 1 \}} w p_\nu) = f_{1}.$ Since $f_{1}$ is a cyclic generator for the module $\bbC[W]$ and is in the image of $\Ev_\nu$ then the homomorphism is surjective.  
Since $f_{w}f_{w'} = \delta_{ww'}f_w$, then
$\Ev_\nu(p)\Ev_\nu(q) = \sum_{w\in W}p(w\nu)f_w\sum_{w'\in W}q(w'\nu)f_{w'}$ is equal to $\sum_{w\in W}pq(w\nu)f_w = \Ev_\nu(pq)$. Taking $H$ invariants
on both sides proves that $\Ev_\nu^H$ is a surjective algebra homomorphism.
\end{proof}

We define two ideals of $\bbC[\frt^*]$, $I_{W,+}$ is the  two sided graded ideal generated by 
\[
\{p \in \bbC[\frt^*]^W: \deg p >0 \} = \{p \in \bbC[\frt^*]^W: p(0) = 0\}
\]
and $I_{W,\nu}$ is the filtered ideal generated by $\{p \in \bbC[\frt^*]^W: p(\nu)=0 \}$.

\begin{lem}\label{l::grI}
The ideal $I_{W,+}$ is equal to $ \gr I_{W, \nu}$.
\end{lem}

\begin{proof}
    $I_{W,+}$ is generated by $\bbC[\frt^*]^W_+=\{p \in \bbC[\frt^*]^W: p(0) = 0\}$ and $I_{W,\nu}$ is generated by $\bbC[\frt^*]^W_\nu =\{p \in \bbC[\frt^*]^W: p(\nu) = 0\}$. Both of which are codimension $1$ in $\bbC[\frt^*]^W$. Furthermore, $\gr (\bbC[\frt^*]^W_\nu)$ is a codimension one graded ideal of $\bbC[\frt^*]^W$. Since $\bbC[\frt^*]^W_+$ is the only graded ideal of codimension one, then $\gr \bbC[\frt^*]^W_\nu = \bbC[\frt^*]^W_+$. 
\end{proof}

\begin{lem}\label{l::kerev}
    The kernel of $\Ev_\nu$ is equal to $I_{W,\nu}$ and the $\ker \Ev_\nu^H = I_{W,\nu}^H$.  
\end{lem}

\begin{proof}
    The kernel of $\Ev_\nu$ is precisely polynomials that evaluate to zero on the full $W$-orbit of $\nu$. The kernel $I_{W,\nu}$ is generated by $W$-invariant polynomials that evaluate to zero on $\nu$. Since they are $W$-invariant they also evaluate to zero on the full $W$-orbit. Hence $\Ker \Ev_\nu \supset I_{W,\nu}$. The quotient of $\bbC[\frt^*]$ by $I_{W,+}$ is the coinvariant algebra, which, in particular, is of dimension $|W|$. Lemma \ref{l::grI} then shows that the codimension of $I_{W,\nu}$ is $|W|$. Since $\Ev_\nu$ is surjective onto $\bbC[W]$ then $\ker \Ev_\nu$ is also of codimension $|W|$ and hence $\Ker \Ev_\nu = I_{W,\nu}$. The second statement follows by taking $H$ invariants of both sides. 
\end{proof}

The polynomials $\bbC[\frt^*]$ (resp. $\bbC[\frt^*]^H$) are graded, therefore the map $\Ev_\nu$  gives $\bbC[W]$ (resp. $\mathbb{C}[W/H]$) a filtration.

\begin{thm}
\label{thm coinv}
Let $W \subset \GL(\ft)$ be such that $\bbC[\frt^*]$ is a free $\bbC[\frt^*]^W$ module of rank $|W|$ and let $H$ be any subgroup of $W$. With the filtration on $\bbC[W]$ endowed by $\Ev_\nu$, the associated graded algebra of $\mathbb{C}[W]$ is the coinvariant algebra $\bbC[\frt^*] / \langle \bbC[\frt^*]^W_+\rangle$, similarly 
$\gr (\mathbb{C}[W/H] )\cong \bbC[\frt^*]^H / \langle \bbC[\frt^*]^W_+\rangle_{\bbC[\frt^*]^H} .$
\end{thm}

\begin{proof}
    Lemmas \ref{l::grI} and \ref{l::kerev} prove that $\gr\ker \Ev_\nu = I_{W,+}$ and $\gr(\ker \Ev_\nu^H) = I_{W,+}^H$. Since $\bbC[W] = \bbC[\frt^*] /\ker(\Ev_\nu)$ then $\gr \bbC[W] = \bbC[\frt^*] / \gr(\ker(\Ev_\nu)) = \bbC[\frt^*] / I_{W,+}$ which is the definition of the coinvariant algebra of $W$ acting on $\bbC[\frt^*]$. An identical statement holds for $\Ev_\nu^H$.
\end{proof}

Hence for any finite group $W$ acting by complex reflections on $\ft$ with  any subgroup $H$ we can define a filtration on $\mathbb{C}[W/H]$ such that the associated graded algebra is isomorphic to the relative coinvariant algebra of $W$ and $H$.

\begin{cor}
\label{cor alpha surj} With notation of Subsections \ref{K dec general} and \ref{cpk eq rk}, the map $\alpha_\frk$ is surjective onto $C(\frp)^\frk$ and the kernel of $\alpha_\frk$ is generated by the $W_\frg$ invariant polynomials in $\bbC[\frt^*]^{W_\frk}$ which evaluate to zero at $\rho$, $(I_{W,\rho})^{W_\frk}$. Furthermore, the map $\gr \alpha_\frk$ is surjective onto $(\twedge\frp)^\frk$ and the kernel of $\gr \alpha_\frk$ is $I_{W_\frg,+}^{W_\frk}$.
\end{cor}

\begin{proof}
    The map $\alpha_\frk$ defined by \eqref{descr alpha k} is given by evaluation of polynomials in $\bbC[\frt^*]^{W_\frk}$ at $\sigma\rho$ for $\sigma \in W^1_{\frg,\frk}$, 
    defining an isomorphism between $C(\frp)^\frk=\proj(S)$ and $\bbC[W_G/W_K]$
   
    where $\pr_{\sigma}$ maps to $\sum_{w \in \sigma W_\frk} f_{w}$. So the map $\alpha_\frk$ is equal to the restriction of $\Ev_{\rho}:\bbC[\frt^*] \to \bbC[W_\frg]$ to $W_\frk$ invariants
    \[ 
    \Ev_{\rho}^{W_\frk}: \bbC[\frt^*]^{W_\frk} \to \bbC[W_\frg/W_\frk] = \bbC[W^1_{\frg,\frk}].
    \]
    Lemma \ref{ev surj} then states that $\alpha_\frk$ is surjective onto $C(\frp)^\frk=\proj(S)$ and Lemma \ref{l::kerev} describes the kernel. Theorem \ref{thm coinv} provides the statement for $\gr \alpha_\frk$.
    \end{proof}

\subsection{$C(\frp)^\frk$ and $(\twedge\frp)^\frk$ in the almost equal rank case: $G/K=\SO(2p+2q+2)/S(\OO(2p+1)\times\OO(2q+1))$}
\label{subsec so odd}
We call this case ``almost equal rank", because $\dim\fra=1$ for all $p$ and $q$. To see that indeed $\dim\fra=1$, and also for later purposes, we first describe a Cartan subalgebra $\frh=\frt\oplus\fra$ of $\frg$.

For the Cartan subalgebra $\frt$ of $\frk$ we choose
block diagonal matrices with diagonal blocks
\eq
\label{t so odd}
t_1J,\dots,t_pJ,0,t_{p+1}J,\dots,t_{p+q}J,0, 
\eeq
where $J=J_1=\smat 0&1\cr -1&0\esmat$, and  $t_1,\dots,t_{p+q}$ are (complex) scalars. The centralizer $\fra$ of $\frt$ in $\frp$ is one-dimensional, spanned by $E_{k\,k+m}-E_{k+m\,k}$. We identify $\frt$ with $\bbR^{p+q}\times 0\subset\bbR^{p+q+1}$ by sending the matrix \eqref{t so odd} to $(t_1,\dots,t_{p+q},0)$, and we identify $\fra$ with $0\times \bbR\subset\bbR^{p+q+1}$ by sending $E_{k\,k+m}-E_{k+m\,k}$ to $(0,\dots,0,1)$. 

Since $\dim\frp$ is odd, $C(\frp)=\End S_1\oplus \End S_2$, where $S_1$ and $S_2$ are the two spin modules. These spin modules are not isomorphic as $C(\frp)$-modules, but they are isomorphic as modules over $C(\frp)_{\even}$, in particular they are isomorphic as 
$\frk$-modules. Moreover, the $\frk$-module $S=S_1=S_2$ is multiplicity free (since the multiplicity $m=2^{[\half\tiny{\dim}\fra]}=1$).

To understand the decomposition $C(\frp)=\End S_1\oplus \End S_2$ more explicitly, we first note that the top element $T$ of $C(\frp)$, which is central in $C(\frp)$ since $\dim\frp$ is odd, acts as 1 on $S_1$ and as $-1$ on $S_2$. Therefore the central idempotents 
\[
\pr_1=\half (1+T),\qquad \pr_2=\half (1-T)
\]
satisfy the following: $\pr_1$ is 1 on $S_1$ and 0 on $S_2$, while $\pr_2$ is 0 on $S_1$ and 1 on $S_2$. It follows that 
\[
\End S_1=C(\frp)\pr_1\qquad\text{ and }\qquad \End S_2=C(\frp)\pr_2.
\]
By Proposition \ref{hodge *}, multiplication by $T$ is an isomorphism between  $C(\frp)_{\even}$ and $C(\frp)_{\odd}$, and moreover
\eq
\label{dec almost eq}
C(\frp)\cong C(\bbC T)\otimes C(\frp)_{\even},
\eeq
with the isomorphism implemented by the multiplication. It follows that 
\[
\End S_1=C(\frp)_{\even}\pr_1\qquad\text{ and }\qquad \End S_2=C(\frp)_{\even}\pr_2.
\]
Namely, $\End S_i$ corresponds to $\pr_i\otimes C(\frp)_{\even}$ under the decomposition \eqref{dec almost eq}.

Since the $\frk$-action on $C(\bbC T)$ is trivial, and since $\End S_1=\End S_2=\End S$ as $\frk$-modules, we see that for any $c\in C(\bbC T)$, $c\otimes C(\frp)_{\even}$ is a copy of $\End S$. In particular, $1\otimes C(\frp)_{\even}=C(\frp)_{\even}$ is isomorphic to $\End S$, and in the following when we write $\End S$ we mean this particular copy. It follows that $\End_\frk S=C(\frp)_{\even}^\frk$; 
this is also the image of the map 
$\alpha_\frk$ of \eqref{descr alpha k}, which now sends $\bbC[\frt^*]^{W_\frk}$ onto $\proj(S)=\End_\frk S\subset C(\frp)^\frk$. ($\End_\frk S$ is equal to the algebra $\proj(S)$ of projections onto isotypic components since $S$ is multiplicity free.)
Furthermore, an analogue of Proposition \ref{prop al k} holds, with $C(\frp)^\frk$ replaced by $\proj(S)$.
Finally, the above discussion shows that
\[
C(\frp)^\frk=C(\bbC T)\otimes \proj(S).
\]

We now go back to our Cartan subalgebra $\frh=\frt\oplus\fra$. 
Since $(\frg,\frt)$-roots are the restrictions of $(\frg,\frh)$-roots to $\frt$, we see that $\Delta(\frg,\frt)$ is of type $B_{p+q}$, while $\Delta(\frk,\frt)$ is of type $B_p\times B_q$. (On the other hand, $\Delta(\frg,\frh)$ is of type $D_{p+q+1}$.)

\begin{lem} 
\label{g' k'}
The filtered algebra $\proj(S)$ is isomorphic to 
the filtered algebra $C(\frp')^{\frk'}=\proj(S')$ for the equal rank symmetric space
\[
G'/K'=\Sp(p+q)/\Sp(p)\times \Sp(q).
\]
This algebra is isomorphic to the algebra 
$\bbC[\frt^*]^{W_\frk'}$ modulo the ideal generated by $\bbC[\frt^*]^{W_\frg'}$. (We identify the isomorphic spaces $\frt$ and $\frt'$.) It can be identified with the space 
$\bbC[r_1,\dots,r_p]_{\leq q}$,  spanned by monomials of degree at most $q$ in the elementary symmetric functions $r_1,\dots,r_p$ of the squares of the variables $x_1,\dots,x_p$, as in Subsection \ref{subsec B/BxB} below. The degrees of these monomials as functions of the $x_i$ range from 0 to $4pq$ and are divisible by 4.
\end{lem}
\pf
It will be shown in Subsection \ref{subsec B/BxB} that $C(\frp')^{\frk'}$ for $G'/K'$ is 
 $\bbC[r_1,\dots,r_p]_{\leq q}$ as in the statement of the lemma. Since types $B$ and $C$ have the same Weyl group, $W_\frg$ and $W_\frk$ are the same for $G/K$ and $G'/K'$. Moreover, $\rho=\rho'=(p+q,p+q-1,\dots,1)\in\frt^*$. Since the algebra $\proj(S)$ is isomorphic to the algebra $\bbC[\frt^*]^{W_\frk}$ modulo the ideal generated by $\bbC[\frt^*]^{W_\frg}_\rho$, 
this implies the lemma.
\epf

The degree of the top element $T$ of $C(\frp)$ is $d=\dim\frp=4pq+2p+2q+1$.
The algebra $\proj(S)$ contains a unique element $t$ of degree $4pq$. Let $e=Tt$ be the corresponding odd element. Then $e$ is the unique element of lowest odd degree; this degree is 
\[
\deg e=d-4pq=2p+2q+1.
\]

\begin{lem}
    \label{squares}
    The elements $t$ and $e$ square to nonzero constants in $C(\frp)$. Therefore, we can rescale these two elements and assume that $t^2=e^2=1$.
\end{lem}
\pf
By Lemma \ref{g' k'}, the filtered algebra $\proj(S)$ is isomorphic to the filtered algebra $C(\frp')^{\frk'}$ for $(\frg',\frk')=(\frsp(p,q),\frsp(p)\times\frsp(q))$. By Proposition \ref{hodge *}, the top degree element of $C(\frp')^{\frk'}$ squares to a nonzero constant. So $t^2$ is a nonzero constant.

Since $T^2=1$, it follows that $e=Tt$ squares to the same constant.
\epf

\begin{cor}
    \label{t T cliff}
    The element $e\in C(\frp)^\frk$ satisfies $B(e,e)=1$ (here $B$ is the extended Killing form on $C(\frp)\cong\twedge\frp$). 
    Consequently the elements $1$ and $e$ span a subalgebra of $C(\frp)^\frk$ isomorphic to the Clifford algebra on the one-dimensional space $\bbC e$. The same is true if we replace $e$ by $t$.
\end{cor}
\pf This follows from Lemma \ref{squares}  and Lemma \ref{T cliff}. 
\epf

\begin{thm}
    There are tensor product decompositions
    \begin{eqnarray*}
    &C(\frp)^\frk\cong C(\bbC e)\otimes \proj(S);\\   
    &(\twedge\frp)^\frk\cong\twedge \bbC e\otimes \gr\proj(S),
    \end{eqnarray*}
    with the isomorphisms implemented by the multiplication.
\end{thm}

\pf
To show that $C(\frp)^\frk\cong C(\bbC e)\otimes \proj(S)$, we first note that by \eqref{dec almost eq}, $\proj(S)=C(\frp)_{\even}$ is a subalgebra of $C(\frp)^\frk$ of half the dimension.
Moreover, by Corollary \ref{t T cliff}, $\vspan \{1,e\}$ is a subalgebra of $C(\frp)^\frk$ isomorphic to the Clifford algebra $C(\bbC e)$ of the space $\bbC e$.
It is thus enough to show that (Clifford) multiplication by $e$ is injective on $\proj(S)$. This follows immediately from $e^2=1$. Since $\proj(S)$ commutes with $C(\frp)^\frk$, this concludes the proof of $C(\frp)^\frk\cong C(\bbC e)\otimes \proj(S)$.

To prove $(\twedge\frp)^\frk\cong\twedge \bbC e\otimes \gr\proj(S)$, we again start from the fact that 
$\gr\proj(S)$ is a subalgebra of $(\twedge\frp)^\frk$ of half the dimension. Moreover, $e\wedge e=0$, since the degree of $e\wedge e$ is $4p+4q+2$, which is even but not divisible by 4. It thus suffices to see that wedging by $e$ is injective on $\gr\proj(S)$. 

We first note that $e\wedge t=T$ up to (nonzero) scalar. Indeed, since the degrees match, it is enough to see that $e\wedge t\neq 0$. But
\[
B(e\wedge t,T)=-B(e,\iota_tT)=-B(e,e)\neq 0.
\]
(We have already seen that $B(e,e)=1$. Alternatively, since $e$ is the only $\frk$-invariant in its degree, and since $B$ is nondegenerate on $\frk$-invariants, $B(e,e)\neq 0$.)

Assume now that $p\in\gr\proj(S)$ is nonzero; we want to show that $e\wedge p\neq 0$. We can assume $p$ is homogeneous. We will be done if we can show that there is $p'\in\gr\proj(S)$ such that $p\wedge p'=t$; then
\[
(e\wedge p)\wedge p'=e\wedge t\neq 0,
\]
so also $e\wedge p\neq 0$.

By Lemma \ref{g' k'}, 
the algebra $\proj(S)$ is isomorphic (as a filtered algebra) to the algebra $C(\frp')^{\frk'}$, where $(\frg',\frk')=(\frsp(p+q),\frsp(p)\times\frsp(q))$. The corresponding graded algebras are thus also isomorphic. The claim now follows from Proposition \ref{hodge *}.(6).
\epf

\subsection{$(\twedge\frp)^K$ in the primary case and almost primary case} 
\label{subsec primary and aprim}

In this section we cite results from \cite{CCCvol3} that describe the structure of $(\twedge\frp)^\frk$, extend this description to $(\twedge \frp)^K$ and explicitly give a generating subspace when $G/K$ is primary or almost primary.

\begin{Def}
    Let $ \lambda_\frk =\mathrm{gr}\alpha_\frk:U(\frk)^\frk \to (\twedge \frp)^\frk$ and let $\lambda_K$ be the restriction of $\lambda_\frk$ to  $U(\frk)^K$.
\end{Def}

\begin{thm}\cite[X.4 Th VII]{CCCvol3} \label{th: cccstruct}
    Let $(\frg,\frk)$ be a symmetric pair of Lie algebras, with Cartan involution $\theta$. Let $\Cal P_\frg$ be a graded subspace that generates $\twedge(\frg)^\frg$ and define the Samelson subspace $\Cal P_\fra = \cal P _\frg^{-\theta}$ then there is an isomorphism of graded algebras 
    \[ (\twedge \frp)^\frk \cong \twedge \cal P _ \fra \otimes \im \lambda_\frk.\]
\end{thm}
We extend the graded algebra description of $(\twedge \frp)^\frk$ from \cite{CCCvol3} to $(\twedge \frp)^K$ in the below proposition.

\begin{prop}
Let $G/K$ be a symmetric space with $G$ connected ($K$ may be disconnected). Then, with notation as in Theorem \ref{th: cccstruct}, there is an isomorphism of graded algebras 
\[ (\twedge \frp)^K \cong \twedge \cal P _ \fra \otimes \im \lambda_K.\] 
\end{prop}

\begin{proof}
When one identifies the de Rham cohomology $H(G/K_e)$ of the space $G/K_e$ with $\twedge\fp^\fk$ and the de Rham cohomology $H(G)$ of $G$ with $\twedge(\fg)^\fg$, then the map on cohomology from $H(G/K_e)$ to $H(G)$ is given by the inclusion of $(\twedge\frp)^\frk$ into $\twedge\frg$ followed by the projection of $\twedge\frg$ onto $(\twedge\frg)^\frg$ along $\frg\cdot\twedge\frg$ (see \cite[X.4]{CCCvol3} for extra details). Both of these maps are $K$-module homomophisms and the image of the composition is $\twedge\cal P_\fra \subset \twedge\cal P_ \frg$  \cite[X.4 Th VII (2)]{CCCvol3}. Since $(\twedge\frg)^\frg$ is $G$-fixed (and hence $K$-fixed) and the above map is a $K$-module homomorphism, we can conclude that the subspace of $(\twedge\fp)^\fk$ congruent to $\twedge\cal P_\fra$ is $K$-fixed. 
 The space $\twedge(\fp)^K$ is equal to
 \[ (\twedge(\fp)^\fk)^K \cong (\twedge(\cal P_\fa) \otimes \im \lambda_\frk)^K = \twedge(\cal P_\fa) \otimes (\im \lambda_\frk)^K,  \] the second equality following from the fact that the first tensorand is entirely $K$-fixed. To finish, note that $(\im \lambda_\frk)^K = \im \lambda_K$,  hence 
 \[ (\twedge \fp)^K \cong \twedge \cal P_\fa \otimes \im \lambda_K.\]
\end{proof}

\begin{Def}
    The symmetric space $G/K$ is primary if $W_{\frg,\frt} = W_\frk$ and almost primary if $W_{\frg,\frt}  = W_K \neq W_\frk$.
\end{Def}

\begin{Def}\label{def samspace}
  Define $\cal P_\wedge(\frp)$ to be the subspace of $(\twedge\frp)^K$ orthogonal to square of the augmentation ideal $((\twedge\frp)^K_+)^2$.  
\end{Def}

\begin{prop}\cite[proposition 4 p 105]{On}\label{prop: graded gens} Suppose that the algebra $A$ is isomorphic to an exterior algebra, and let $S$ be a subset of $A$. Then the following are equivalent: 
\begin{enumerate}
\item the algebra $A$ is generated by $S$ and $1$; 
\item the augmentation ideal $A_+$ is equal to  the $A$-submodule generated by $S$; 
\item the augmentation ideal $A_+$ is equal to $\mathrm{span}_\bbC(S) \oplus A_+^2$.
\end{enumerate}
\end{prop}
In particular, Proposition \ref{prop: graded gens} shows that modifying a generating set $S$ by any elements from $A_+^2$ retains the property of generating $A$; we will use this to show that $P_\wedge(\frp)$ generates $(\twedge\frp)^K$ in the (almost) primary case.

\begin{cor}\label{cor: gen orthog}
    Suppose that $A$ is a graded algebra with non-degenerate bilinear form such that $A$ is isomorphic to an exterior algebra, and different graded components are orthogonal to each other.  Let $R$ be the subspace of $A_+$ orthogonal to $A_+^2$ then $A$ is generated by $R$.
\end{cor}

\begin{proof} Since $A$ is isomorphic to an exterior algebra, suppose $P$ is any graded subspace of $A$ such that $A = \twedge P$. The proof follows by performing a Gram-Schmidt algorithm on $P$ modifying by elements in $A_+^2$ until the generating subspace is orthogonal to $A_+^2$. By Proposition \ref{prop: graded gens}, at each step of the Gram-Schmidt process the new subspace still generates $A$ and is graded since the different graded components are orthogonal. The end result is a graded subspace $P'$ orthogonal to $A_+^2$ that generates $A$. Hence we have a direct sum of orthogonal components 
    \[ A = \bbC \oplus P' \oplus A_+^2,\]
    and $P'$ is contained in $R$. Any element in $R \setminus P'$ would be orthogonal to $P'$, $\bbC$ and $A_+^2$ thus contradicting the fact that the form on $A$ is non-degenerate, hence $R = P'$ and $R$ generates $A$.
\end{proof}

\begin{thm}\label{thm prim and aprim alg}
Let $G/K$ be primary or almost primary,  then the inclusion $\cal P_\wedge(\frp) \hookrightarrow (\twedge\frp)^K$ extends to an isomorphism of graded algebra
\[ (\twedge \frp)^K = \twedge\cal P_\wedge(\frp).\]
\end{thm}

\begin{proof}
    If $G/K$ is primary then $\im \lambda_\frk$ is $\bbC$ and $(\twedge\frp)^\frk \cong \twedge\cal P_\fra$, if $G/K$ is almost primary then $\im \lambda_K$ is $\bbC$ and $(\twedge \frp)^K \cong \twedge\cal P_\fra$. Hence, in both cases $(\twedge\frp)^K$ is isomorphic to an exterior algebra,  denote this isomorphism by $f: \twedge \cal P_\fra \cong (\twedge\frp)^K$. Then $(\twedge \frp)^K$ is generated by the graded subspace $f(\cal P_\fra)$ and the form on $(\twedge \frp)^K$ induced by the Killing form on $\frp$ is non degenerate on $(\twedge \frp)^K$ with differing graded components orthogonal.  Hence Corollary \ref{cor: gen orthog} proves that $P_\wedge(\frp)$ (Definition \ref{def samspace}) generates $(\twedge \frp)^K$; $(\twedge \frp)^K = \twedge\cal P_\wedge (\frp)$.
\end{proof}

The degrees of $\cal P_\wedge(\frp)$ are the same as the degrees of $\cal P_\fa = \cal P_\frg ^{-\theta}$ which are given in \cite[Table I,II, III p. 492-496]{CCCvol3} and repeated below for reference. There is paper in preparation \cite{CGKP} that will prove a transgression theorem when $G/K$ is primary or almost primary and will directly give the degrees of $\cal P_\wedge(\frp)$.

\begin{table}[htbp]
\caption{Degrees of $\cal P_\wedge(\fp)$}\label{table:degPa}
    \begin{tabular}{lll} \label{tab degrees prim}
    $G$ & $K$ & degrees of $\cal P_\wedge(\fp)$ \\
   Group & &\\
   \hline
   $\UU_{2n+1}(\mathbb{R})^2$ & $\UU_{2n+1}(\mathbb{R})$ & $ 4p-1 : 1 \leq p \leq n$ \\
   $\UU_{2n}(\mathbb{R})^2$ & $\UU_{2n}(\mathbb{R})$ & $ 4p-1 : 1 \leq p \leq n -1 $ and $ 2n-1$ \\
   $\UU_n(\mathbb{C})^2$ & $\UU_n(\mathbb{C})$ & $ 2p-1 : 1 \leq p \leq n $\\
   $\UU_n(\mathbb{H})^2$ & $\UU_n(\mathbb{H})$ &  $ 4p-1 : 1 \leq p \leq n $ \vspace{5pt}\\
   
   Primary \\ 
   \hline 
   $\UU_{2n+1}(\mathbb{C})$ & $\UU_{2n+1}(\mathbb{R})$ &  $ 4p-3 : 1 \leq p \leq n+1 $ \vspace{5pt}\\
   $\UU_{2n}(\bbC)$ & $\UU_n(\bbH)$ & $4p-3 : 1 \leq p \leq n$\\
   Almost Primary \\
   \hline
    $\UU_{2n}(\mathbb{C})$ & $\UU_{2n}(\mathbb{R})$ & $ 4p-3 : 1 \leq p \leq n$  \vspace{5pt}\\
   
    \end{tabular} 

(Recall $\UU_n(\bbR) = \OO(n)$, $\UU_n(\bbC) = \UU(n)$, and $\UU_n(\bbH) = \Sp(n)$.)
\end{table}

\subsection{$C(\frp)^K$ and $(\twedge\frp)^K$ for disconnected $K$: the case $G/K=\SO(k+m)/S(\OO(k)\times\OO(m))$}\label{subsec so/soo}
To understand what happens with the decomposition \eqref{decomp S k} when $K$ is disconnected, we consider the group theoretic Weyl group $W_K$ defined as
\eq
\label{gp W gp}
W_{K}=N_{K}(\frt)/Z_{K}(\frt),
\eeq
where $N_{K}(\frt)$ denotes the normalizer in $K$ of the Cartan subalgebra $\frt$ of $\frk$, while $Z_K(\frt)$ denotes the centralizer of $\frt$ in $K$. 
It is well known (see \cite[Theorem~4.54]{beyond}) that for connected $K$, $W_{K}=W_\frk$, the Weyl group of the root system of $(\frk,\frt)$. Analogously, we define $W_G$, which is equal to $W_\frg=W_{\frg,\frt}$ since $G$ is assumed connected.

We note that it looks like we should consider the groups $W_{\Kt}$ and $W_{\Kt_e}$, but these are in fact the same as $W_K$ respectively $W_{K_e}$. This follows from the fact that the adjoint action of $k\in\Kt$ is the same as the adjoint action of $\pi(k)$ where $\pi$ denotes the covering map from $\Kt$ to $K$.

Let now $G/K=\SO(k+m)/S(\OO(k)\times \OO(m))$. 
In these cases $K$ is disconnected and the group theoretic Weyl group $W_K$ may be different from $W_\frk$. In the following we describe $W_K$ explicitly.

The group $K=S(\OO(k)\times \OO(m))$ has two connected components, and we choose the following explicit representative $s$ of the disconnected component:
\begin{eqnarray}
\label{def s}
s&=&\diag(\underbrace{1,\dots,1,-1}_k,\underbrace{1,\dots,1,-1}_m)\quad\text{if }k\text{ and }m\text{ are both even};\\
\nonumber
s&=&\diag(\underbrace{1,\dots,1,-1}_k,\underbrace{-1,\dots,-1}_m)\quad\text{if }k\text{ is even and }m\text { is odd};\\
\nonumber
s&=&\diag(\underbrace{-1,\dots,-1}_k,\underbrace{-1,\dots,-1}_m)\quad\text{if }k\text{ and }m\text{ are both odd}.
\end{eqnarray}
For the Cartan subalgebra $\frt$ of $\frk$ we choose
block diagonal matrices with diagonal blocks
\begin{eqnarray*}
&t_1J,\dots,t_pJ,t_{p+1}J,\dots,t_{p+q}J & \qquad  \text{if }(k,m)=(2p,2q); \\
&t_1J,\dots,t_pJ,t_{p+1}J,\dots,t_{p+q}J,0 & \qquad  \text{if }(k,m)=(2p,2q+1); \\
&t_1J,\dots,t_pJ,0,t_{p+1}J,\dots,t_{p+q}J,0 & \qquad  \text{if }(k,m)=(2p+1,2q+1);, 
\end{eqnarray*}
where $J=J_1=\smat 0&1\cr -1&0\esmat$, and  $t_1,\dots,t_{p+q}$ are (complex) scalars. 

In the equal rank cases $(k,m)=(2p,2q)$ and $(k,m)=(2p,2q+1)$, $\frt$ is also a Cartan subalgebra of $\frg$, while for $(k,m)=(2p+1,2q+1)$, as noted in Subsection \ref{subsec so odd}, a Cartan subalgebra for $\frg$ is $\frh=\frt\oplus\fra$ where  $\fra$ is one-dimensional, spanned by $E_{k\,k+m}-E_{k+m\,k}$. For $(k,m)=(2p,2q)$ and $(k,m)=(2p,2q+1)$, we identify $\frt$ with $\bbR^{p+q}$ by sending the above described matrix to $(t_1,\dots,t_{p+q})$. For $(k,m)=(2p+1,2q+1)$ we identify $\frt$ with $\bbR^{p+q}\times 0\subset\bbR^{p+q+1}$ by sending the above described matrix to $(t_1,\dots,t_{p+q},0)$, and we identify $\fra$ with $0\times \bbR\subset\bbR^{p+q+1}$ by sending $E_{k\,k+m}-E_{k+m\,k}$ to $(0,\dots,0,1)$. In this last case we see that $\Delta(\frg,\frt)$ is of type $B_{p+q}$, while $\Delta(\frk,\frt)$ is of type $B_p\times B_q$. (On the other hand, $\Delta(\frg,\frh)$ is of type $D_{p+q+1}$.)

Since
\eq
\label{compute s}
\begin{pmatrix} 1 & 0 \cr 0 & -1\end{pmatrix}
\begin{pmatrix} 0 & t \cr -t & 0\end{pmatrix}
\begin{pmatrix} 1 & 0 \cr 0 & -1\end{pmatrix}=
\begin{pmatrix} 0 & -t \cr t & 0\end{pmatrix},
\eeq
we see that in each of the cases $s$ given in \eqref{def s} normalizes $\frt$ and acts on it as follows:

\begin{lem}
\label{s on t}
For $(k,m)=(2p,2q)$, 
\[
s(t_1,\dots,t_{p-1},t_p,t_{p+1},\dots,t_{p+q-1},t_{p+q})=
(t_1,\dots,t_{p-1},-t_p,t_{p+1},\dots,t_{p+q-1},-t_{p+q}).
\]
For $(k,m)=(2p,2q+1)$, 
\[
s(t_1,\dots,t_{p-1},t_p,t_{p+1},\dots,t_{p+q})=
(t_1,\dots,t_{p-1},-t_p,t_{p+1},\dots,t_{p+q})
\]
For $(k,m)=(2p+1,2q+1)$, $s$ centralizes $\frt$ (i.e., $s$ acts trivially on $\frt$).

In particular, $s$ normalizes $\frt$ in all cases.
\end{lem}

Since in the cases $(k,m)=(2p,2q)$ and $(k,m)=(2p,2q+1)$ we have $K=K_e\rtimes\{1,s\}$, and we know by Lemma \ref{s on t} that $s$ normalizes $\frt$, we conclude that
\[
W_K=W_{K_e}\rtimes\{1,s\}=W_\frk\rtimes\{1,s\}.
\]
For $(k,m)=(2p+1,2q+1)$, $s$ centralizes $\frt$ and thus $W_K=W_\frk$.
To conclude:

\begin{prop}
\label{w sok+m}
For $G/K=\SO(k+m)/S(\OO(k)\times \OO(m))$,  
\[
W_K=\left\{ \begin{matrix} S(B_p\times B_q) & \qquad \text{if}\ (k,m)=(2p,2q);\cr B_p\times B_q & \qquad \text{if}\  (k,m)=(2p,2q+1);\cr
B_p\times B_q & \qquad \text{if}\  (k,m)=(2p+1,2q+1).
\end{matrix}\right.
\]
Here $B_p\times B_q$ is the group consisting of permutations and sign changes of the first $p$ and the last $q$ coordinates, while 
$S(B_p\times B_q)$ is the group consisting of permutations and sign changes of the first $p$ and the last $q$ coordinates, so that the total number of sign changes is even. 

The group $W_G$ is equal to $D_{p+q}$ if $(k,m)=(2p,2q)$, and to $B_{p+q}$ if $(k,m)=(2p,2q+1)$ or if $(k,m)=(2p+1,2q+1)$.
\end{prop}

In Lemma \ref{s on t} we have described the adjoint action of $s$ (and hence also of $\st$) on $\frt$. It follows from that and from passing to the dual $\frt^*$ that 

\begin{lem}
\label{s on roots}
Let $s\in K$ be defined by \eqref{def s}, and let $\st$ be a lift of $s$ in $\Kt$. Then the coadjoint action of $s$ and $\st$ permutes the positive roots of $(\frk,\frt)$.
\end{lem} 
\pf
As already noted, $\Ad(\st)=\Ad(s)$ so the coadjoint actions of $s$ and $\st$ are also the same. To compute the action of $s$, we note that the formulas in Lemma \ref{s on t} imply:

(1) For $G/K=SO(2p+2q)/S(O(2p)\times O(2q))$, $s$ interchanges the roots $\eps_i-\eps_p$ and $\eps_i+\eps_p$ ($1\leq i<p$), and the roots $\eps_{p+j}-\eps_{p+q}$ and $\eps_{p+j}+\eps_{p+q}$ ($1\leq j<q$), while fixing all the other positive $(\frk,\frt)$-roots. 

(2) For $G/K=SO(2p+2q+1)/S(O(2p)\times O(2q+1))$, $s$ interchanges the roots $\eps_i-\eps_p$ and $\eps_i+\eps_p$ ($1\leq i<p$), while fixing all the other positive $(\frk,\frt)$-roots. 

(3) For $G/K=SO(2p+2q+2)/S(O(2p+1)\times O(2q+1))$, $s$ fixes all the positive $(\frk,\frt)$-roots.
\epf

The formulas in Lemma \ref{s on t} also describe the action of $s$ (and hence of $\st$) on weights $\la\in\frt^*$. In coordinates, the action is exactly the same as on coordinates of elements of $\frt$: 

(1) For $G/K=SO(2p+2q)/S(O(2p)\times O(2q))$, $s$ acts by changing the sign of the $p$-th and the $(p+q)$-th coordinate; 

(2) For $G/K=SO(2p+2q+1)/S(O(2p)\times O(2q+1))$, $s$ acts by changing the sign of the $p$-th coordinate;

(3) For $G/K=SO(2p+2q+2)/S(O(2p+1)\times O(2q+1))$, $s$ acts trivially.
\smallskip

We are now ready to describe the $\Kt$-decomposition of the spin module $S$ in each of the cases. Recall that the $\frk$-decomposition of $S$, or equivalently the $\Kt_e$-decomposition where $\Kt_e$ is the spin double cover of the connected component $K_e$ of the identity in $K$, is multiplicity free (since $\dim\fra\leq 1$), and given by

(1) For $G/K=SO(2p+2q)/S(O(2p)\times O(2q))$, the infinitesimal characters of the irreducible $\frk$-submodules of $S$ are the $\frk$-dominant $W_G$-conjugates of $\rho=(p+q-1,p+q-2,\dots,1,0)$. These are the $(p,q)$-shuffles of $\rho$, and the $(p,q)$-shuffles of $\rho$ with the sign of the $p$-th and the $(p+q)$-th coordinates changed to negative (note that one of these coordinates is zero, so the sign change does not affect it). The highest weights are obtained from these infinitesimal characters by subtracting $\rho_\frk$.

(2) For $G/K=SO(2p+2q+1)/S(O(2p)\times O(2q+1))$, the infinitesimal characters of the irreducible $\frk$-submodules of $S$ are the $\frk$-dominant $W_G$-conjugates of $\rho=(p+q-\half,p+q-\frac{3}{2},\dots,\frac{3}{2},\half)$. These are the $(p,q)$-shuffles of $\rho$, and the $(p,q)$-shuffles of $\rho$ with the sign of the $p$-th coordinate changed to negative. The highest weights are obtained from these infinitesimal characters by subtracting $\rho_\frk$.

(3) For $G/K=SO(2p+2q+2)/S(O(2p+1)\times O(2q+1))$, the infinitesimal characters of the irreducible $\frk$-submodules of $S$ are the $\frk$-dominant $W_G$-conjugates of $\rho=\rho_{(\frg,\frt)}=\rho_{(\frg,\frh)}\big|_\frt=(p+q,p+q-1,\dots,2,1)$. These are the $(p,q)$-shuffles of $\rho$. The highest weights are obtained from these infinitesimal characters by subtracting $\rho_\frk$.

It is now clear that in Cases (1) and (2) the action of $s$ (or $\st$) interchanges the infinitesimal characters that differ only by the sign of one of the coordinates. Since in each of the cases $s$ permutes $\Delta^+(\frk,\frt)$, it fixes $\rho_\frk$ and thus also interchanges the highest weights corresponding to the above infinitesimal characters.

In Case (3), $s$ (and $\st$) fix all the $\frk$-infinitesimal characters and hence also all the highest weights in $S$. 

\begin{prop}
\label{K dec spin}
For $G/K=\SO(2p+2q)/S(\OO(2p)\times \OO(2q))$ or $G/K=\SO(2p+2q+1)/S(\OO(2p)\times \OO(2q+1))$, each irreducible $\Kt$-module in $S$, when viewed as a $\frk$-module, decomposes into two irreducible $\frk$-modules. The highest weights of these two modules are interchanged by $\st$ and $s$, and differ from each other by one sign change.

For $G/K=\SO(2p+2q+2)/S(\OO(2p+1)\times \OO(2q+1))$, the $\Kt$-decomposition of the spin module $S$ is the same as the $\frk$-decomposition.
\end{prop}
\pf 
Let $v$ be any highest weight vector for $\frk$ in $S$, and let its weight be $\la$ (it is one of the weights described above). We claim that $\st v$ is a highest weight vector of weight $s\la$.

To see this, we first note that $\Kt$-equivariance of the $\frk$-action on $S$ implies that for any $X\in\frk$, 
\eq
\label{equi}
X(\st v)=\st (\Ad(\st)^{-1} X) v=\st (\Ad(s)^{-1} X) v=\st (\Ad(s)X) v.
\eeq
By Lemma \ref{s on roots}, if $X$ is a positive root vector, then $\Ad(s)X$ is also a positive root vector. It follows that
\[
X(\st v)=\st (\Ad(s) X) v=0,
\]
so $\st v$ is a highest weight vector. 
To see the weight of $\st v$, we apply \eqref{equi} for $X\in\frt$. Recall from Lemma \ref{s on t} that then also $\Ad(s) X\in\frt$, so we have
\[
X(\st v)=\st (\Ad(s) X) v=\la(\Ad(s) X)\st v= s\la(X)\st v,
\]
so $\st v$ is of weight $s\la$.

Let now $\Kt_e$ be the spin double cover of the identity component $K_e$ of $K$. Then by Proposition \ref{K-dec of S conn} the $\Kt_e$-decomposition of $S$ is the same as the $\frk$-decomposition. Since 
$\Kt_e$ and $\st$ generate $\Kt$, 
and since we have described the action of $\st$, the proposition follows. 
\epf

We now describe a version of the Harish-Chandra isomorphism for the disconnected case. Let 
\[
\gamma_0:Z(\frk)=U(\frk)^\frk=U(\frk)^{K_e}\to S(\frt)^{W_\frk}=S(\frt)^{W_{K_e}}
\]
be the usual Harish-Chandra isomorphism. 

\begin{prop}
\label{HC K}
For $K=S(\OO(k)\times \OO(m))$, the Harish-Chandra map $\gamma_0:U(\frk)^{K_e}\to S(\frt)^{W_{K_e}}$ restricts to an isomorphism
\[
\gamma:U(\frk)^K\to S(\frt)^{W_K}.
\]
\end{prop}
\pf
We have seen that $K=K_e\rtimes\{1,s\}$, where $s$ is defined as above (the product is direct if $k,m$ are both odd). Since $s$ normalizes $K_e$, $\frk$, $\frt$  and $W_{K_e}$, it acts on $U(\frk)^{K_e}$ and on  $S(\frt)^{W_{K_e}}$, and the map $\gamma_0$ intertwines these actions. The claim now follows by taking $s$-invariants.
\epf

\begin{rem}
    It is possible to generalize the above proposition to the case when $K$ is an arbitrary compact Lie group. 
\end{rem}

We now define 
\[
\alpha_K:U(\frk)^K\cong \bbC[\frt^*]^{W_K}\to \proj(S)\subseteq C(\frp)^K,
\]
as the restriction of the map $\alpha_\frk$ of \eqref{descr alpha k} to the $K$-invariants. Here $\proj(S)$ denotes the algebra of $\Kt$-equivariant projections of the $K$-module $S$ to its isotypic components $E_\sigma$, where $\sigma\in W_{G,K}^1$, the set of minimal length representatives of $W_K$-cosets in $W_G$. Since the $E_\sigma$ are of multiplicity 1, $\proj(S)=\End_{\Kt} S$. The map $\alpha_K$ is given by the analogue of \eqref{descr alpha k}, i.e.,
\[
\alpha_K(P)=\sum_{\sigma\in  W_{G,K}^1} P(\sigma\rho)\pr_\sigma,\qquad P\in \bbC[\frt^*]^{W_K},
\]
where $\pr_\sigma:S\to E_\sigma$ is the $K$-equivariant projection. From the above considerations we conclude

\begin{cor}
\label{cor o/oo}
{\rm (1)} For $G/K=\SO(2p+2q)/S(\OO(2p)\times\OO(2q))$, 
the algebra $C(\frp)^K$ is isomorphic to $\proj(S)$, which is isomorphic to $\bbC[\frt^*]^{S(B_p\times B_q)}$ modulo the ideal generated by  $D_{p+q}$-invariants in $\bbC[\frt^*]$ evaluating to $0$ at $\rho=(p+q-1,\dots,1,0)$. The algebra  $(\twedge\frp)^K$ is isomorphic to $\gr\proj(S)$, which is isomorphic to $\bbC[\frt^*]^{S(B_p\times B_q)}$ modulo the ideal generated by  $D_{p+q}$-invariants in $\bbC[\frt^*]$ evaluating to $0$ at $0$.

{\rm (2)} For $G/K=\SO(2p+2q+1)/S(\OO(2p)\times\OO(2q+1))$, the algebra $C(\frp)^K$ is isomorphic to $\proj(S)$, which is isomorphic to $\bbC[\frt^*]^{B_p\times B_q}$ modulo the ideal generated by $B_{p+q}$-invariants in $\bbC[\frt^*]$ evaluating to $0$ at $\rho=(p+q-\half,\dots,\frac{3}{2},\half)$. The algebra  $(\twedge\frp)^K$ is isomorphic to $\gr\proj(S)$, which is isomorphic to $\bbC[\frt^*]^{B_p\times B_q}$ modulo the ideal generated by  $B_{p+q}$-invariants in $\bbC[\frt^*]$ evaluating to $0$ at $0$.

{\rm (3)} For $G/K=SO(2p+2q+2)/S(\OO(2p+1)\times\OO(2q+1))$, the algebra $C(\frp)^K$ is isomorphic to $C(\bbC e)\otimes\proj(S)$, 
where $e$ is a generator of degree $2p+2q+1$ squaring to $1$, and $\proj(S)$ is isomorphic to $\bbC[\frt^*]^{B_p\times B_q}$ modulo the ideal generated by  $B_{p+q}$-invariants in $\bbC[\frt^*]$ evaluating to $0$ at $\rho=(p+q,\dots,2,1)$. The algebra  $(\twedge\frp)^K$ is isomorphic to $\twedge\bbC e\otimes \gr\proj(S)$, where $e$ is a generator of degree $2p+2q+1$ squaring to $0$, and where $\gr\proj(S)$ 
is isomorphic to $\bbC[\frt^*]^{B_p\times B_q}$ modulo the ideal generated by  $B_{p+q}$-invariants in $\bbC[\frt^*]$ evaluating to $0$ at $0$.
\end{cor}

\subsection{$(\twedge \frp)^K$ for disconnected $K$: the case $G/K=\UU(n)/\OO(n)$}
\label{subsec u/o}

    We use the standard matrix realizations of $G=\UU(n)$ and $K=\OO(n)$. 
Since $K$ is disconnected, the group $W_K$ may be different from $W_\frk$ and we want to describe it explicitly. We prove that $G/K = \UU(n) /\OO(n)$ is primary (resp. almost primary) when $n$ is even (resp. odd), we then apply the results of Section \ref{subsec primary and aprim}.

The group $K=\OO(n)$ has two connected components and we choose the following representative for the disconnected component:
\begin{eqnarray*}
s=\diag(1,\dots,1,-1)&\qquad&\text{if}\ n=2k;\\
s=\diag(-1,\dots,-1)&\qquad&\text{if}\ n=2k+1.   
\end{eqnarray*}

For the Cartan subalgebra $\frt$ of $\frk$ we take the space of block diagonal matrices with diagonal blocks
\begin{eqnarray*}
t_1J,\dots, t_kJ&\qquad&\text{if  }\ n=2k;\\
t_1J,\dots, t_kJ,0&\qquad&\text{if  }\ n=2k+1,
\end{eqnarray*}
where as before $J=J_1=\smat 0&1\cr -1&0\esmat$ and $t_1,\dots,t_k$ are complex scalars. We extend $\frt$ to a Cartan subalgebra $\frh=\frt\oplus\fra$ of $\frg$, where $\fra$ is the space of block diagonal matrices with diagonal blocks
\begin{eqnarray*}
a_1I,\dots, a_kI&\qquad&\text{if  }\ n=2k;\\
a_1I,\dots, a_kI,a_{k+1}&\qquad&\text{if  }\ n=2k+1,
\end{eqnarray*}
where $I=I_2=\smat 1&0\cr 0&1\esmat$ and $a_1,\dots,a_{k+1}$ are complex scalars.

We can identify $\frt\subset\frh\cong\bbC^n$ with
\begin{eqnarray*}
\{(t_1,\dots,t_k,-t_k,\dots,-t_1)\bbar t_1,\dots,t_k\in\bbC\}&\qquad&\text{if  }\ n=2k;\\
\{(t_1,\dots,t_k,0,-t_k,\dots,-t_1)\bbar t_1,\dots,t_k\in\bbC\}&\qquad&\text{if  }\ n=2k+1,
\end{eqnarray*}
and 
$\fra\subset\frh\cong\bbC^n$ with
\begin{eqnarray*}
\{(a_1,\dots,a_k,a_k,\dots,a_1)\bbar a_1,\dots,a_k\in\bbC\}&\qquad&\text{if  }\ n=2k;\\
\{(a_1,\dots,a_k,a_{k+1},a_k,\dots,a_1)\bbar a_1,\dots,a_{k+1}\in\bbC\}&\qquad&\text{if  }\ n=2k+1,
\end{eqnarray*}
This corresponds to the action of the involution $\sigma$ on $\frh$ being
\[
\sigma(h_1,\dots,h_n)=(-h_n,\dots,-h_1).
\]
Here $\sigma$ is the restriction to $G=\UU(n)$ of the involution in Subsection \ref{real lagr grass}; in particular, $\UU(n)^\sigma=\OO(n)$.

The above identification enables us to identify the $(\frg,\frt)$ roots easily: $\Delta(\frg,\frt)$ is of type $C_k$ if $n=2k$ and of type $BC_k$ if $n=2k+1$. Thus 
\[
W_G=W_\frg=B_k;
\]
recall that the Weyl group $B_k$, which is the same as the Weyl group of type $C_k$ or $BC_k$, consists of permutations and sign changes of coordinates of $\frt\cong \bbC^k$.

Of course, $W_\frk$ is $D_k$ if $n=2k$ and $B_k$ if $n=2k+1$; here $D_k$ is the Weyl group of type $D_k$, consisting of permutations of the coordinates and sign changes of an even number of coordinates of $\frt\cong \bbC^k$. We however want to identify the group theoretic Weyl group $W_K$.

To do this, we identify $\frt$ with $\bbC^k$ by sending $(t_1,\dots,t_k,-t_k,\dots,-t_1)$ to $(t_1,\dots,t_k)$.
Using \eqref{compute s} we see that
for $n=2k$ the element $s$ normalizes $\frt$ and sends $(t_1,\dots,t_{k-1},t_k)\in\frt$ to $(t_1,\dots,t_{k-1},-t_k)$. Thus $W_K=W_\frk\rtimes\{1,s\}=B_k$. For $n=2k+1$, $s$ centralizes $\frt$, hence $W_K=W_\frk=B_k$. 

We have proved that $G/K = U(n)/O(n)$ is primary for odd $n$ and almost primary for even $n$, therefore Theorem \ref{thm prim and aprim alg} gives the graded algebra structure of $(\twedge \frp)^K$.

We conclude

\begin{cor}
\label{u/o even and odd}
For $G/K=\UU(n)/\OO(n)$, either $G/K$ is primary or almost primary. Hence, by Theorem \ref{thm prim and aprim alg},
\[
(\twedge \frp)^K\cong \twedge(\cal P_\wedge(\frp)),
\]
where $\cal P_\wedge(\frp)$ is the subspace defined in Definition \ref{def samspace} and the degrees are given in Table \ref{tab degrees prim}. 
\end{cor}

\section{Cohomology rings of compact symmetric spaces}
\label{coho symm}

\subsection{Some general facts}
\label{subsec coho symm general}

Let $G/K$ be a compact symmetric space, with $G$ a compact connected Lie group and $K$ a closed symmetric subgroup. Then the de Rham cohomology (with real coefficients) of $G/K$ can be identified, as an algebra, with $(\twedge\frp_0^*)^K$, where $\frp_0$ stands for the tangent space $\frg_0/\frk_0$ to $G/K$ at $eK$ ($\frg_0$ and $\frk_0$ are the Lie algebras of $G$ respectively $K$).  In the examples we are interested in, $\frp_0$ can be $K$-equivariantly embedded into $\frg_0$, and also $\frp_0^*\cong\frp_0$.

As mentioned in the introduction, this fact is well known, but it is difficult to find an appropriate reference. We present here a proof we learned from Sebastian Goette \cite{goette}.

Any $g\in G$ acts on $G/K$ by a map that
is homotopic to the identity. Indeed, since $G$ is connected, there is a smooth path $g(t)$, $t\in [0,1]$, from $g$ to the unit element $e\in G$. Then $H:G/K\times [0,1]\to G/K$, $H(x,t)=g(t)x$, is a smooth homotopy from $g:G/K\to G/K$ to the identity map on $G/K$. 

It now follows that if $\omega$ is a closed form, then it represents the same cohomology class as $g^*\omega$, for any $g\in G$. Namely, \cite[Proposition 15.5]{Lee} says that $g^*$ and $e^*=\id$ induce the same map on cohomology, so the class of $\omega$ is the same as the class of $g^*\omega$.
Since $G$ is compact, we can average
over $g$ and get a $G$-invariant differential form that represents the same cohomology class as $\omega$.

Next, assume that $\omega$ is $G$-invariant and $\omega=d\mu$. Even if $\mu$
is not $G$-invariant itself, we know that $d\mu=g^*d\mu=dg^*\mu$. So we
can average over $g\in G$ again to get a $G$-invariant differential form $\bar\mu$ such that $\omega=d\bar\mu$. So we see that the de Rham cohomology is captured by the subcomplex of $G$-invariant differential forms.

Now we recall that the differential forms on $G/K$ are sections of the homogeneous vector bundle 
\[
G\times_K \twedge\frp_0^* \to G/K.
\]
The bundle $G\times_K \twedge\frp_0^*$ is defined as
$(G\times\twedge\frp_0^*)/\sim$, where $\sim$ is the equivalence relation defined by
\[
(gk,\nu)\sim(g,\Ad^*(k)\nu),\qquad g\in G,k\in K, \nu\in\twedge\frp_0^*.
\]
The differential forms, or sections of the above bundle, are maps
\[
\omega:G\to\twedge\frp_0^*\quad \text{such that }\quad \omega(gk) = \Ad^*(k^{-1})\omega(g),\quad g\in G,\ k\in K.
\]
The group $G$ acts on such $\omega$ by left translation, i.e.,
\[
(g\omega)(g')=\omega(g^{-1}g'),\qquad g,g'\in G.
\]
Thus $\omega$ is $G$-invariant if and only if it is constant as a function on $G$, i.e., $\omega(g)=\omega(e)$ for any $g\in G$.

For such an invariant form $\omega$, set $\bar\omega=\omega(e)\in\twedge\frp_0^*$. We claim that $\bar\omega\in(\twedge\frp_0^*)^K$. Indeed, for any $k\in K$ we have
\begin{multline*}
\Ad^*(k)\bar\omega=\Ad^*(k)\omega(e)=(\text{since $\omega$ is a section})=\\
\omega (ek^{-1})=\omega(k^{-1})=(\text{since $\omega$ is $G$-invariant})=\omega(e)=\bar\omega.
\end{multline*}
Conversely, if $\bar\omega\in(\twedge\frp_0^*)^K$, then $\omega(g)=\bar\omega$ defines a $G$-invariant form $\omega$ on $G/K$.

So we see that for any compact homogeneous space $G/K$, with $G$ a compact connected Lie group and $K$ a closed subgroup of $G$, the de Rham cohomology $H(G/K)$ is the cohomology of the complex $((\twedge\frp_0^*)^K,d)$, where the differential $d$ is induced by the de Rham differential. Let us describe the differential $d$ more explicitly.
We first recall a coordinate free formula for the de Rham differential $d$ on a manifold $M$. Any differential $q$-form is determined if we know how to evaluate it on any $q$-tuple of (smooth) vector fields. In this interpretation, the de Rham differential of a $q$-form $\omega$ is the $(q+1)$-form given by
\begin{eqnarray}
\label{de rham}
& \qquad d\omega(X_1\wedge ... \wedge X_{q+1})=
\sum_i(-1)^{i-1} X_i(\omega(X_1\wedge\dots\wedge\widehat X_i\wedge\dots\wedge X_{q+1})) + \\ \nonumber
& \qquad \sum_{i<j} (-1)^{i+j}\omega ([X_i,X_j] \wedge X_1 \wedge \dots\wedge \widehat X_i\wedge\dots\wedge\widehat X_j\wedge\dots\wedge X_{q+1}),
\end{eqnarray}
where $X_1,\dots,X_{q+1}$ are vector fields on $M$, the bracket denotes the Lie bracket of vector fields, and the hat over a variable means this variable is omitted. See e.g. \cite[Proposition 12.19]{Lee}.

If $M=G/K$ and if the form $\omega$ is $G$-invariant (as we saw we may assume), then we know $\omega$ at any point $gK$ if we know it at the base point $eK$. More precisely, if $Y_1,\dots,Y_q$ is any $q$-tuple of tangent vectors at a point $gK$ in $G/K$, then 
\[
\omega(gK)(Y_1,\dots,Y_q)=\omega(eK)(g^{-1}_*Y_1,\dots,g^{-1}_*Y_q).
\]
It follows that it is enough to know the value of $\omega$ at $q$-tuples of $G$-invariant vector fields, which correspond to the tangent space $\frg_0/\frk_0\cong\frp_0$ to $G/K$ at $eK$. The $G$-invariant vector fields can in turn be obtained as push-forwards of left invariant vector fields on $G$ under the projection $\Pi:G\to G/K$. Since $\Pi_*$ is compatible with Lie brackets, and since the left invariant vector fields on $G$ can be identified with the Lie algebra $\frg_0$ of $G$, we see that if $\tilde X,\tilde Y$ are invariant vector fields on $G/K$ corresponding to tangent vectors $X,Y\in\frp_0$, then the vector field $[\tilde X,\tilde Y]$ corresponds to the tangent vector $\pi([X,Y])$, where the bracket $[X,Y]$ is taken in $\frg_0$ and $\pi=d\Pi:\frg_0\to \frp_0$ is the canonical projection. 
Thus the formula \eqref{de rham} can be rewritten with $X_1,\dots,X_{q+1}$ in $\frp_0$ and with $[X_i,X_j]$ replaced by $\pi([X_i,X_j])$. Moreover, the first sum in the formula vanishes, since the action of a vector field involves the differentiation, and $\omega$ is constant. So the de Rham 
differential becomes
\eq
\label{de rham bis}
d\omega(X_1\wedge ... \wedge X_{q+1})=
\sum_{i<j} (-1)^{i+j}\omega (\pi([X_i,X_j]) \wedge X_1 \wedge \dots\wedge \widehat X_i\wedge\dots\wedge\widehat X_j\wedge\dots\wedge X_{q+1}),
\eeq
with $X_i$ in $\frp_0$. 

Now let us assume in addition that $K$ is a symmetric subgroup of $G$, i.e., that $(\frg_0,\frk_0)$ is a symmetric pair. Then  $[\frp_0,\frp_0]\subseteq\frk_0$, so $\pi([X_i,X_j])=0$ for any $X_i,X_j\in\frp_0$, and we see that $d\omega=0$. 
Hence 
\eq
\label{coho sym sp}
H(G/K)=(\twedge\frp_0^*)^K\cong(\twedge\frp_0)^K.
\eeq

In the following we will pass to the cohomology of $G/K$ with complex coefficients:
\[
H(G/K;\bbC)=H(G/K)\otimes_\bbR\bbC=(\twedge\frp_0)^K\otimes_\bbR\bbC= (\twedge\frp)^K,
\]
where $\frp$ stands for the complexification of $\frp_0$. We will also denote by $\frg$ respectively $\frk$ the complexifications of $\frg_0$ respectively $\frk_0$.

Recall from Section \ref{sec spin}, 
the paragraph above Corollary \ref{cor o/oo}, that there is a surjection $\alpha=\alpha_K: \bbC[\frt^*]^{W_K}\to \Proj(S)$,  given by
 \[
\alpha_K(P)=\sum_{\tilde\sigma\in W^1_{G,K}}P(\tilde\sigma\rho)\pr_{\tilde\sigma},
    \]
with kernel equal to the ideal $\langle\bbC[\frt^*]^{W_G}_\rho\rangle$ in 
$\bbC[\frt^*]^{W_K}$ generated by the $W_G$-invariants that vanish at $\rho$. 
Likewise, the kernel of the surjection $\gr\alpha:\bbC[\frt^*]^{W_K}\to\gr \Proj(S)\subseteq(\twedge\frp)^K$ is the ideal $\langle\bbC[\frt^*]^{W_G}_+\rangle$ generated by the $W_G$-invariants that vanish at $0$. In this section we will only use the obvious inclusions
\eq 
\label{ker incl}
\langle \bbC[\frt^*]^{W_G,}_\rho\rangle\subseteq\ker\alpha
\quad\text{ and }\quad \langle\bbC[\frt^*]^{W_G}_+\rangle\subseteq\ker\gr\alpha.
\eeq 
The opposite inclusions will get reproved as a byproduct of our analysis. Namely, in each of the cases we will have a candidate set for a basis of the quotient $\bbC[\frt^*]^{W_K}/\ker\alpha$ respectively $\bbC[\frt^*]^{W_K}/\ker\gr\alpha$ consisting of certain monomials, of the cardinality equal to
\[
\dim \Proj(S)=\dim\gr \Proj(S)=|W^1_{G,K}|.
\]
We will use the relations coming from $\bbC[\frt^*]^{W_G}_\rho$ respectively $\bbC[\frt^*]^{W_G}_+$ to show that the images of these candidate monomials span the respective quotients. It will follow that they form a basis, and also that there can be no additional relations outside of  $\langle \bbC[\frt^*]^{W_G}_\rho\rangle$ respectively $\langle\bbC[\frt^*]^{W_G}_+\rangle$.

\subsection{The case $G/K=\UU(p+q)/\UU(p)\times \UU(q)$}
\label{subsec A/AxA}

This is an equal rank case, so $\Proj(S)=C(\frp)^K$ and $\gr \Proj(S)=(\twedge\frp)^K$. 
Since $G$ and $K$ are both connected, the Weyl groups are $W_G=W_\frg=S_{p+q}$ and $W_K=W_\frk=S_p\times S_q$. Let $x_1,\dots,x_{p+q}$ be coordinates for the Cartan subalgebra $\frt$. 
Let
\begin{itemize}
\item $r_1,\dots,r_p$ be the elementary symmetric functions in variables $x_1,\dots,x_p$;
\item $s_1,\dots,s_q$ be the elementary symmetric functions in variables $x_{p+1},\dots,x_{p+q}$;
\item $t_1,\dots,t_{p+q}$ be the elementary symmetric functions in variables $x_1,\dots,x_{p+q}$.
\end{itemize}
Then the space of $W_K$-invariants in $S(\frt)$ is generated by $r_1,\dots,r_p$  and $s_1,\dots,s_q$.

\begin{thm}
\label{gen rel basis} 
For $G/K=\UU(p+q)/\UU(p)\times \UU(q)$, 
$ 1 \leq p \leq q $, 
the algebra $C(\frp)^K$ is isomorphic to the algebra $\frH(p,q;c)$ of Definition \ref{def alg h}, where $c=(t_1(\rho),\dots,t_{p+q}(\rho))$. The algebra $(\twedge\frp)^K$ is isomorphic to $\frH(p,q;0)$. 

In other words, the algebras $C(\frp)^K$ and $(\twedge\frp)^K$ are 
both generated by $r_1,\dots,r_p$ (or more precisely, their images in the respective quotients), with relations generated by 
\eq
\label{rel}
\sum_{i,j\geq 0;\ i+j=k}r_is_j=t_k= \left\{\begin{matrix} t_k(\rho) & \text{ for the case of } C(\frp)^K  \cr
0 & \text{ for the case of } (\twedge\frp)^K \end{matrix}
\right.
\eeq
for $k=1,\dots,p+q$, where we set $r_0=s_0=1$ and $r_i=0$ if $i>p$, $s_j=0$ if $j>q$.

To summarize,
\[ 
\begin{aligned}
C(\frp)^K &=
\frH(p,q;c) = \frac{\C[r_1,\dots,r_p; s_1,\dots,s_q]}{\Bigl(\sum_{i,j\geq 0;\ i+j=k}r_is_j=t_k(\rho)\Bigr)} \\
(\twedge\frp)^K &=
\frH(p,q;0)
= \frac{\C[r_1,\dots,r_p; s_1,\dots,s_q]}{\Bigl(\sum_{i,j\geq 0;\ i+j=k}r_is_j=0\Bigr)} 
\end{aligned}
\]
The latter algebra is isomorphic to $H^*(\Gr_p(\bbC^{p+q}), \bbC) $. 

A basis for each of these algebras is given by the monomials $r^\alpha=r_1^{\alpha_1}\dots r_p^{\alpha_p}$ of degree $|\alpha|\leq q$, so that our algebras can be identified with the space
\[
\bbC[r_1,\dots,r_p]_{\leq q}.
\]
In particular, each monomial in $r_1,\dots,r_p$ of degree $q+1$ can be expressed as a linear combination of lower degree monomials in $r_1,\dots,r_p$.  Such expressions follow from \eqref{rel explic 1} and \eqref{rel explic other} below, and they provide another set of defining relations for each of our algebras.

The filtration degree of $C(\frp)^K$ inherited from $C(\frp)$, and the gradation degree of $(\twedge\frp)^K$ inherited from $\twedge\frp$, are obtained by setting $\deg r_i=2i$ for $i=1,\dots,p$.
\end{thm}

\pf
Let us first note that it is clear that the algebras $C(\frp)^K$ and  $(\twedge\frp)^K$ are generated by the $r_i$ and the $s_j$. Also, the inclusions \eqref{ker incl} imply that the relations for these algebras include the relations \eqref{rel}. We are going to see that these relations in fact generate all the relations.

Using the first $q$ of the relations \eqref{rel}, we can express all $s_j$ as polynomials in the $r_i$. Indeed, the first relation is
\[
r_1+s_1=t_1
\]
and we see that $s_1=t_1-r_1$. Now the second relation is
\[
r_2+r_1s_1+s_2=t_2, 
\]
so $s_2$ can be expressed as a polynomial in the $r_i$ since we have already expressed $s_1$. We continue inductively.
So each of our algebras is generated by the (images of the) polynomials in the $r_i$.

We now prove that every monomial in the $r_i$ of degree $q+1$ can be expressed as a linear combination of lower degree monomials. This will finish the proof. Namely, this will show that the monomials in the $r_i$ of degree at most $q$ span each of our algebras, and their number is $\binom{p+q}{p}$, the same as the dimension of both algebras. So they have to form a basis. Since we only use the relations \eqref{rel}, it follows that these relations generate all the relations, otherwise the dimension would be lower which is impossible.

We order the monomials in the $r_i$ first by degree, and then inside each degree by the reverse lexicographical order. We will show by induction on this order that all degree $q+1$ monomials can be expressed by lower degree monomials.

We start with the first monomial, $r_1^{q+1}$. We express the $s_j$ in terms of the $r_i$ from relations $2,3,\dots,q+1$. These relations are linear in the $s_j$, with coefficients that are either constant or the $r_i$. In matrix form, this system of equations is
\eq
\label{matrix}
\begin{pmatrix}
r_1 & 1 & 0 & 0 & 0 & \dots & 0 \cr
r_2 & r_1 & 1& 0 & 0 & \dots & 0 \cr
\vdots & \vdots & \vdots & \vdots & \vdots & \vdots & \vdots \cr
r_{p-1} &  \dots & r_1 & 1 & 0 & \dots & 0\cr
r_{p} & r_{p-1} & \dots & r_1 & 1 & 0 & \dots \cr
0 & r_p & r_{p-1} & \dots & r_1 & 1 &  \dots \cr
\vdots & \vdots & \vdots & \vdots & \vdots & \vdots & \vdots \cr
0 & \dots & 0 & r_p & \dots & r_2 & r_1
\end{pmatrix}
\begin{pmatrix}
s_1 \cr
s_2 \cr
\vdots \cr
s_q
\end{pmatrix}
=
\begin{pmatrix}
t_2-r_2 \cr
t_3-r_3 \cr
\vdots \cr
t_p-r_p \cr
t_{p+1} \cr
\vdots \cr
t_{q+1}
\end{pmatrix}
\eeq

The determinant $D$ of this system contains a unique monomial $r_1^q$ of degree $q$, since upon expanding successively along the first row we can always pick either $r_1$, or $1$ which leads to lower degree. In particular $D\neq 0$, so we can solve the system by Cramer's rule and obtain 
\eq
\label{sol}
D s_i=D_i,\quad i=1,\dots,q,
\eeq
where $D_i$ is obtained from $D$ by replacing the $i$-th column by the right hand side of \eqref{matrix}.

Now we multiply the first equation, $r_1+s_1=t_1$, by $D$, and use \eqref{sol} to get
\eq
\label{rel explic 1}
r_1D+D_1=t_1 D.
\eeq
Since all monomials in $D_1$ and in $t_1D$ are of degree $\leq q$, and since all the monomials of $r_1D$ are of degree $\leq q$ except for the monomial $r_1^{q+1}$, we have expressed $r_1^{q+1}$ as a linear combination of lower degree monomials.

We now do the induction step. Let $1\leq i_1<i_2<\dots< i_a\leq p$ be integers and let
\eq
\label{mono}
r_{i_1}^{m_1} r_{i_2}^{m_2}\dots r_{i_a}^{m_a}
\eeq 
be a monomial of degree $q+1$ different from $r_1^{q+1}$. Suppose that we have already expressed all the degree $q+1$ monomials that are before \eqref{mono} in reverse lexicographical order.

Let us consider the degree $q$ monomial 
\eq 
\label{smaller mono}
r_{i_1}^{m_1-1} r_{i_2}^{m_2}\dots r_{i_a}^{m_a}.
\eeq 
Notice that for each $i$ and $j$, $r_is_j$ appears in exactly one of the equations (the $(i+j)$-th one). We first assume that $m_1>1$ and pick the equations that contain respectively
\begin{multline}
\label{picked}
r_{i_1}s_1,r_{i_1}s_2,\dots r_{i_1}s_{m_1-1};\ r_{i_2}s_{m_1},\dots,r_{i_2}s_{m_1+m_2-1};\ \dots \\
\dots; \ r_{i_a}s_{m_1+\dots + m_{a-1}},\dots,r_{i_a}s_{m_1+ \dots + m_a-1} (=r_{i_a}s_q).
\end{multline}
We view these equations as a linear system for $s_1,\dots,s_q$ and note that the diagonal coefficients are exactly the coefficients of the terms \eqref{picked}, i.e., 
\[
r_{i_1},\dots,r_{i_1},r_{i_2},\dots,r_{i_2},\dots,r_{i_a},\dots,r_{i_a},
\]
with $r_{i_1}$ repeating $m_1-1$ times, $r_{i_2}$ repeating $m_2$ times, ..., $r_{i_a}$ repeating $m_a$ times. Thus the determinant of the system, denoted again by $D$, contains the monomial \eqref{smaller mono}, and we claim this is the leading term of the expanded determinant $D$. (We warn the reader not to confuse the present $D$ with the one in \eqref{sol}.)

The first of the picked equations is
\[
r_{i_1}s_1+r_{i_1-1}s_2+\dots +r_1s_{i_1}+s_{i_1+1}=t_{i_1+1}-r_{i_1+1}.
\]
(This covers all the cases since we defined $r_i=0$ for $i>p$ and $s_j=0$ for $j>q$.)
The first row of $D$ is thus 
\[
r_{i_1}\ r_{i_1-1}\ \dots \ r_1\ 1\ 0\ \dots \ 0
\]
with 1 and/or zeros possibly missing.
When we expand $D$ along the first row and then write out all the lower order determinants as combinations of monomials, the terms containing the first row elements $r_{i_1-1},r_{i_1-2},\dots$ are all either of lower degree or before the term \eqref{smaller mono} in our ordering. To get the leading term we thus have to pick $r_{i_1}$ and cross the first row and column. The remaining determinant (if $m_1>2$) has the first row equal to
\[
r_{i_1}\ r_{i_1-1}\ r_{i_1-2}\ \dots,
\]
and we use the same argument to conclude that we should pick $r_{i_1}$ to obtain the leading term.

After we go over all the rows containing $r_{i_1}$, we continue with the next row
\[
r_{i_2}\ r_{i_2-1}\ r_{i_2-2}\ \dots
\]
We again see that to obtain the leading term we have to pick $r_{i_2}$ from this row.

The conclusion is that the leading term of $D$ is indeed the monomial \eqref{smaller mono}; in particular, $D\neq 0$. We now again write the Cramer's rule 
\eq
\label{soli}
Ds_i=D_i, \quad i=1,\dots,q.
\eeq
We multiply the equation containing $r_{i_1}$ with coefficient 1, i.e., the equation
\[
r_{i_1}+r_{i_1-1}s_1+\dots +r_1 s_{i_1-1}+ s_{i_1}=t_{i_1},
\]
by $D$, and use \eqref{soli} to get
\eq
\label{rel explic other}
r_{i_1}D+r_{i_1-1}D_1+\dots +r_1 D_{i_1-1}+ D_{i_1}=t_{i_1}D.
\eeq
The leading term of $r_{i_1}D$ is \eqref{mono}, and the other terms in the above equation are either of lower degree, or of the same degree but of lower order with respect to the reverse lexicographical order. Expressing these last terms by lower degree terms using the inductive assumption, we see that we have expressed \eqref{mono} as a linear combination of lower degree terms.

This finishes the proof if $m_1>1$. If $m_1=1$, we proceed analogously, starting by picking the equation containing $r_{i_2}s_1$. The argument is entirely similar.

The statement about degrees follows from the fact that $r_i$ is of degree $i$ as a polynomial in the variables $x_j$, and that the map $\alpha:U(\frk)\to C(\frp)$ doubles the degree. 
\epf

\begin{rem}
In the course of the proof of Theorem \ref{gen rel basis} we 
have obtained explicit relations \eqref{rel explic 1} and \eqref{rel explic other} for the generators $r_1,\dots,r_p$. It is clear that $C(\frp)^K$ is the algebra generated by the $r_i$ with these relations if we set $t_i=t_i(\rho)$ and $(\twedge\frp)^K$ is the algebra generated by the $r_i$ with the same relations if we set $t_i=0$.
\end{rem}

\begin{rem}
\label{rem schur}
The monomials in Theorem \ref{gen rel basis} span the same space as the Schur polynomials $s_\lambda$ for $\lambda$ in the $p\times q$ box. In particular, these Schur polynomials also form a basis of our algebra(s), since their number is equal to the dimension of each of the two algebras.

To pursue this relationship in more detail, we first recall the well known Jacobi-Trudi formulas that express Schur polynomials as polynomials in the elementary symmetric functions: if $\la$ is a partition with Young diagram inside the $p\times q$ box, let $\la^t=(\la^t_1,\dots,\la^t_l)$ be the transpose of $\la$ ($l\leq q$ is the length of $\la^t$). Then
\eq
\label{jacobi}
s_\la=\det\pmat r_{\la^t_1} & r_{\la^t_1+1}&\dots& r_{\la^t_1+l-1}\cr
r_{\la^t_2-1}&r_{\la^t_2}&\dots& r_{\la^t_2+l-2}\cr  \vdots&\vdots & \vdots &\vdots\cr
 r_{\la^t_l-l+1}& r_{\la^t_l-l+2}&\dots& r_{\la^t_l}
\epmat
\eeq
Here $r_j$ is the $j$th elementary symmetric function on $x_1,\dots,x_p$ if $1\leq j\leq p$, $r_0=1$, and $r_j=0$ if $j<0$ or $j>p$.

This is an expression of $s_\la$ as a linear combination of monomials in $r_1,\dots,r_p$  of degree at most $q$. We claim that in this way we obtain a triangular change of basis between the $s_\la$ and the monomials in $r_1,\dots,r_p$  of degree at most $q$. To see this, we order the monomials first be degree (in the $r_j$), and then by reverse lexicographical order inside each degree. We claim that upon expanding the determinant \eqref{jacobi} the leading term is the diagonal monomial $r_{\la^t_1}r_{\la^t_2}\dots r_{\la^t_l}$. Indeed, let us expand the determinant along the first row. Since $\la^t_1\geq\la^t_2\geq\dots\geq\la^t_l$, the diagonal monomial has no $r_j$ with $j>\la^t_1$, but if we pick any element of the first row other than $r_{\la^t_1}$, all monomials in the corresponding piece of the expansion will contain $r_j$ with $j>\la^t_1$. So to obtain the leading term we must pick $r_{\la^t_1}$. We now repeat this argument inductively, always expanding along the first row.

The main advantage of the Schur polynomials is the fact that their multiplication table is well understood, using Littlewood-Richardson coefficients. While the computation of the LR coefficients is only algorithmic, computer programs for computing them are widely known and available; for example, there is an online calculator available from the web page of Joel Gibson \cite{Gibson}. 
Our approach using monomials in the elementary functions and the relations between them that we obtained can also lead to a multiplication table, as illustrated by the following example. In this way, we get an alternative way of computing the LR coefficients.
\end{rem}

\begin{ex} 
{\rm 
Let $p=2$ and $q=3$. The expressions of the Schur polynomials for $\la$ in the $2\times 3$ box in terms of monomials in the elementary symmetric functions $r_1=x_1+x_2$, $r_2=x_1x_2$ are
\begin{equation*}
\begin{array}{lllll}
    s_{(0,0)}=1; \qquad
    &s_{(1,0)}=r_1; \qquad
    &s_{(2,0)}=r_1^2-r_2; \qquad
    &s_{(3,0)}=r_1^3-2r_1r_2;\\
    &s_{(1,1)}=r_2; \qquad
    &s_{(2,1)}=r_1r_2; \qquad
    &s_{(3,1)}=r_1^2r_2-r_2^2; \\
    &s_{(2,2)}=r_2^2; \qquad
    &s_{(3,2)}=r_1r_2^2; \qquad
    &s_{(3,3)}=r_2^3.
\end{array}
\end{equation*}
Our relations expressing monomials of degree four in terms of monomials of degree at most three are
\begin{equation*}
     r_1^4 = 3r_1^2r_2-r_2^2;  \quad \
     r_1^3r_2= 2r_1r_2^2;  \quad \
     r_1^2r_2^2= r_2^3;  \quad \
     r_1r_2^3= 0;  \quad \
     r_2^4 = 0.
\end{equation*}
The multiplication table for the monomials is
\bigskip

\begin{center}
    \begin{tabular}{|c||c|c|c|c|c|c|c|c|c|}
   \hline  &$r_1$&$r_2$&$r_1^2$&$r_1r_2$&$r_2^2$&$r_1^3$&$r_1^2r_2$&$r_1r_2^2$&$r_2^3$\\
    \hline
    \hline
    $r_1$ & $r_1^2$ & $r_1r_2$ & $r_1^3$ & $r_1^2r_2$ & $r_1r_2^2$ & $3r_1^2r_2-r_2^2$ & $2r_1r_2^2$ & $r_2^3$ & 0 \\
    \hline
    $r_2$ && $r_2^2$ & $r_1^2r_2$ & $r_1r_2^2$ & $r_2^3$ & $2r_1r_2^2$ & $0$ & 0 & 0\\
    \hline
    $r_1^2$ & && $3r_1^2r_2-r_2^2$ & $2r_1r_2^2$ & $r_2^3$ & $5r_1r_2^2$ & $2r_2^3$ & 0 & 0\\
    \hline
    $r_1r_2$ &&&& $r_2^3$ & 0 & $2r_2^3$ & 0&0&0\\
    \hline
    $r_2^2$ &&&&&0&0&0&0&0\\
    \hline
    $r_1^3$ & &&&&& $5r_2^3$ & 0&0&0\\
    \hline
    $r_1^2r_2$ & &&&&&& 0&0&0\\
    \hline
    $r_1r_2^2$ & &&&&&&& 0&0\\
    \hline
    $r_2^3$ & &&&&&&&& 0\\
    \hline
    \end{tabular}
\end{center}

\bigskip

The reader is invited to compare this with the multiplication table for the Schur polynomials obtained from \cite{Gibson}; to use the online calculator one has to remember that the Schur polynomials for $\la$ outside of the $2\times 3$ box have to be replaced by zeros.
}
\end{ex}

\subsection{The cases $G/K=\Sp(p+q)/\Sp(p)\times \Sp(q)$}
\label{subsec B/BxB}

Since $G$ and $K$ have equal rank, $\Proj(S)=C(\frp)^K$ and $\gr \Proj(S)=(\twedge\frp)^K$.
Since $G$ and $K$ are both connected, the Weyl group $W_G$ is equal to $W_\frg$ which is isomorphic to $B_{p+q}$, while the Weyl group $W_K=W_\frk$ is $B_p\times B_q$. (Recall that type $B$ and type $C$ have the same Weyl group. It consists of permutations and sign changes of the variables.)

As in type A, the set $W^1_{G,K}$ consists of $(p,q)$-shuffles. In particular,
\[
|W^1_{G,K}|=\dim C(\frp)^K=\dim(\twedge\frp)^K =\binom{p+q}{p}.
\]
It is well known (see \cite[p.{67}]{HumphreysRefBook}) that the algebra of $B_k$-invariants is a polynomial algebra generated by symmetric functions of the squares of the variables. 
Thus $S(\frt)^{W_K}$ is generated by
\[
r_1=x_1^2+\dots+x_p^2,\quad r_2=x_1^2x_2^2+\dots+x_{p-1}^2x_p^2,\quad \dots,\quad r_p=x_1^2x_2^2\dots x_p^2
\]
\[
s_1=x_{p+1}^2+\dots+x_{p+q}^2,\ s_2=x_{p+1}^2x_{p+2}^2+\dots+x_{p+q-1}^2x_{p+q}^2,\ \dots,\ s_q=x_{p+1}^2\dots x_{p+q}^2
\]
and $S(\frt)^{W_G}$ is generated by
\[
t_1=x_1^2+\dots+x_{p+q}^2,\quad t_2=x_1^2x_2^2+\dots+x_{p+q-1}^2x_{p+q}^2,\quad \dots,\quad t_{p+q}=x_1^2x_2^2\dots x_{p+q}^2
\]

As in the type A case, the relations for $C(\frp)^K$ respectively $(\twedge\frp)^K$ include the relations \eqref{rel}.
Moreover, we have 

\begin{thm}
\label{gen rel basis B/BxB} 
For $G/K=\Sp(p+q)/\Sp(p)\times \Sp(q)$, $ 1 \leq p \leq q $, 
the algebra $C(\frp)^K$ is isomorphic to the algebra $\frH(p,q;c)$ of Definition \ref{def alg h}, with generators the above $r_i$, and with $c=(t_1(\rho),\dots,t_{p+q}(\rho))$. The algebra $(\twedge\frp)^K$ is isomorphic to $\frH(p,q;0)$.

In other words,
\[
\begin{aligned}
C(\frp)^K &=
\frH(p,q;c) = \frac{\bbC[r_1,\dots,r_p; s_1,\dots,s_q]}{\Bigl(\sum_{i,j\geq 0;\ i+j=k}r_is_j=t_k(\rho)\Bigr)} \\
(\twedge\frp)^K &=
\frH(p,q;0)
= \frac{\bbC[r_1,\dots,r_p; s_1,\dots,s_q]}{\Bigl(\sum_{i,j\geq 0;\ i+j=k}r_is_j=0\Bigr)} 
\end{aligned}
\]
The latter algebra is isomorphic to $H^*(\Gr_p(\bbH^{p+q}), \bbC) $.

Both algebras can be identified with the space
\[
\bbC[r_1,\dots,r_p]_{\leq q}.
\]
The filtration degree of $C(\frp)^K$ inherited from $C(\frp)$, and the gradation degree of $(\twedge\frp)^K$ inherited from $\twedge\frp$, are obtained by setting $\deg r_i=4i$ for $i=1,\dots,p$.
\end{thm}

\pf
The same as the proof of Theorem \ref{gen rel basis}. (The statement about degrees follows from the fact that $r_i$ is of degree $2i$ as a polynomial in the variables $x_j$, and from the fact that the map $\alpha:U(\frk)\to C(\frp)$ doubles the degree.) 
\epf

\begin{rem}
\label{rmk schur B/BxB}
    As in Remark \ref{rem schur}, we can replace the monomials in the $r_i$ by the Schur polynomials $s_\lambda$ for $\lambda$ in the $p\times q$ box. This allows us to write the multiplication table in the usual way.
\end{rem}

\subsection{The cases $G/K=\SO(k+m)/S(\OO(k)\times \OO(m))$}
\label{subsec SOk+m}

If $(k,m)=(2p,2q)$ or $(k,m)=(2p,2q+1)$ then $G$ and $K$ have equal rank, so $C(\frp)^K=\Proj(S)$ and $(\twedge\frp)^K=\gr \Proj(S)$. 

If $(k,m)=(2p,2q+1)$, then by Proposition \ref{w sok+m}, $W_G=B_{p+q}$ and $W_K=B_p\times B_q$, so $C(\frp)^K$ and $(\twedge\frp)^K$ are described by Theorem \ref{gen rel basis B/BxB}. 

If $(k,m)=(2p,2q)$, then by Proposition \ref{w sok+m}, $W_G=D_{p+q}$ and $W_K=S(B_p\times B_q)$. The invariants in $S(\frt)$ under $B_p\times B_q\supset W_K$ are generated by the symmetric functions $r_1,\dots,r_p$ of $x_1^2,\dots,x_p^2$ and  the symmetric functions $s_1,\dots,s_q$ of $x_{p+1}^2,\dots,x_{p+q}^2$, while the invariants under $D_p\times D_q\subset W_K$ are generated by $r_1,\dots,r_{p-1}$, $s_1,\dots,s_{q-1}$, and the Pfaffians $\bar r_p=x_1\dots x_p$, $\bar s_q=x_{p+1}\dots x_{p+q}$ (see \cite[p.{68}]{HumphreysRefBook}).
It follows that the invariants under $W_K$ are generated by 
\[
r_1,\dots,r_p;\ s_1,\dots,s_q;\ \bar r_p\bar s_q.
\]
Of course, these generators are not independent, as $(\bar r_p\bar s_q)^2=r_ps_q$. 

Since our algebras $C(\frp)^K$ and $(\twedge\frp)^K$ are quotients of $\bbC[\frt^*]^{W_K}$ by the ideal generated by $\bbC[\frt^*]^{W_G}_\rho$ respectively $\bbC[\frt^*]^{W_G}_+$, and since $\bar r_p\bar s_q=x_1\dots x_{p+q}$ is $W_G$-invariant, we can remove $\bar r_p\bar s_q$ from the list of generators. (Note that the value of $\bar r_p\bar s_q$ at $\rho$ is 0, since 0 is a coordinate of $\rho$.) It follows that the algebras $C(\frp)^K$ and $(\twedge\frp)^K$ are again described by Theorem \ref{gen rel basis B/BxB}.

Finally, suppose that $(k,m)=(2p+1,2q+1)$. 
This is an unequal (but almost equal) rank case and as we saw in Subsection \ref{subsec so odd}, 
the fundamental Cartan subalgebra is $\frh=\frt\oplus\fra\cong\bbC^{p+q+1}$ with 
\[
\frt=\{(t_1,\dots t_{p+q},0)\bbar t_i\in\bbC\};\qquad \fra=\{(0,\dots,0,a)\bbar a\in\bbC\}.
\]
The root system $\Delta(\frg,\frt)$ is $B_{p+q}$ while  $\Delta(\frk,\frt)=B_p\times B_q$. 
(Note that $\Delta(\frg,\frh)$ is $D_{p+q+1}$.) 

By Corollary \ref{cor o/oo} (3), $C(\frp)^K=C(\Cal P_\fra)\otimes \Proj(S)$; since $\fra$ is one-dimensional, $C(\Cal P_\fra)$ is two-dimensional, spanned by 1 and a generator $e$ squaring to $1$. Likewise, $(\twedge\frp)^K=\twedge\Cal P_\fra\otimes \Proj(S)$, with $\twedge\Cal P_\fra$ spanned by 1 and by a generator $e$ squaring to 0. 

We know from Proposition \ref{w sok+m} that $W_G=B_{p+q}$ and $W_K=B_p\times B_q$. It follows that the algebras $\Proj(S)$ and $\gr\Proj(S)$ are described by Theorem \ref{gen rel basis B/BxB}, with notation given by that theorem and the text above it. 

\begin{thm}
\label{gen rel basis SOk+m} 
Let $G/K=\SO(k+m)/S(\OO(k)\times \OO(m))$.

{\rm (a)} If $(k,m)=(2p,2q)$ or $(k,m)=(2p,2q+1)$, then the algebra $C(\frp)^K$ is isomorphic to the algebra $\frH(p,q;c)$ of Definition \ref{def alg h}, with generators $r_1,\dots,r_p$ as above, and with $c=(t_1(\rho),\dots,t_{p+q}(\rho))$. The algebra $(\twedge\frp)^K$ is isomorphic to $\frH(p,q;0)$. 
In other words,
\[ 
\begin{aligned}
C(\frp)^K &=
\frH(p,q;c) = \frac{\bbC[r_1,\dots,r_p; s_1,\dots,s_q]}{\Bigl(\sum_{i,j\geq 0;\ i+j=k}r_is_j=t_k(\rho)\Bigr)} \\
(\twedge\frp)^K &=
\frH(p,q;0) 
= \frac{\bbC[r_1,\dots,r_p; s_1,\dots,s_q]}{\Bigl(\sum_{i,j\geq 0;\ i+j=k}r_is_j=0\Bigr)} 
\end{aligned}
\]
The latter algebra is isomorphic to $H^*(\Gr_k(\bbR^{k + m}), \bbC)$.

Both algebras can be identified with the space 
\[
\bbC[r_1,\dots,r_p]_{\leq q}.
\]
The filtration degree of $C(\frp)^K$ inherited from $C(\frp)$, and the gradation degree of $(\twedge\frp)^K$ inherited from $\twedge\frp$, are obtained by setting $\deg r_i=4i$ for $i=1,\dots,p$.

{\rm (b)} If $(k,m)=(2p+1,2q+1)$, then the algebra $C(\frp)^K$ contains the algebra $\frH(p,q;c)$ as in (a), and an additional generator $e$ squaring to $1$. 
The algebra $(\twedge\frp)^K$ contains the algebra $\frH(p,q;0)$ as in (a), and an additional generator $e$ squaring to 0. 
$$ 
\begin{aligned}
C(\frp)^K &= \frac{\bbC[r_1,\dots,r_p; s_1,\dots,s_q; e]}{\Bigl(\sum_{i,j\geq 0;\ i+j=k}r_is_j=t_k(\rho),\, e^2 = 1 \Bigr)} \\
(\twedge\frp)^K & 
= \frac{\bbC[r_1,\dots,r_p; s_1,\dots,s_q; e]}{\Bigl(\sum_{i,j\geq 0;\ i+j=k}r_is_j=0 ,\, e^2 = 0 \Bigr)} 
\end{aligned}
$$
The latter algebra is isomorphic to
$H^*(\Gr_k(\bbR^{k+m}), \bbC))$.

Each of the algebras can be identified with 
\[
\bbC[r_1,\dots,r_p]_{\leq q} \oplus \bbC[r_1,\dots,r_p]_{\leq q}\, e.
\]
The filtration degree of $C(\frp)^K$ inherited from $C(\frp)$, and the gradation degree of $(\twedge\frp)^K$ inherited from $\twedge\frp$, are obtained by setting $\deg r_i=4i$ for $i=1,\dots,p$ and $\deg e=2p+2q+1$.
\end{thm}

\pf
This follows from the discussion above and from Corollary \ref{cor o/oo}.
\epf

\begin{rem}
\label{rmk schur SOk+m}
    As in Remark \ref{rem schur}, we can replace the monomials in the $r_i$ by the Schur polynomials $s_\lambda$ for $\lambda$ in the $p\times q$ box. This allows us to write the multiplication table in the usual way.
\end{rem}

\subsection{The case $G/K=\Sp(n)/\UU(n)$}
\label{subsec C/A}

Since $G$ and $K$ have equal rank, $\Proj(S)=C(\frp)^K$ and $\gr \Proj(S)=(\twedge\frp)^K$.

Since $G$ and $K$ are both connected, the Weyl groups are $W_G=W_\frg=B_n$ and $W_K=W_\frk=A_{n - 1}$. In other words, $W_K$ consists of the permutations of the variables $x_1,\dots,x_n$, while $W_G$
consists of permutations and sign changes of $x_1,\dots,x_n$. 
The set $W^1_{G,K}$ has $2^n$ elements and can be identified with the sign changes.
The algebra $S(\frt)^{W_K}$ is generated by
\[
r_1=x_1+\dots+x_n,\quad r_2=x_1x_2+x_1x_3+\dots + x_{n-1}x_n,\ \dots,\quad  r_n=x_1x_2\dots x_n, 
\]
and $S(\frt)^{W_G}$ is generated by
\[
t_1=x_1^2+\dots+x_n^2,\quad t_2=x_1^2x_2^2+\dots+x_{n-1}^2x_n^2,\ \dots,\  t_n=x_1^2x_2^2\dots x_n^2.
\]
To write down the relations coming from $S(\frt)^{W_G}$, let $z$ be a formal variable and  note that
\begin{multline*}
\sum_{k=0}^n (-1)^k t_k z^{2k}=\prod_{k=1}^n (1-x_k^2z^2)=\prod_{i=1}^n (1-x_iz)\prod_{i=1}^n (1+x_iz) \\
= \left(\sum_{i=0}^n(-1)^ir_iz^i\right)\left(\sum_{j=0}^n r_jz^j\right)=\sum_{k=0}^n\ \left(\sum_{i+j=2k}(-1)^ir_ir_j\right)z^{2k}.
\end{multline*}
It follows that the relations are
\eq
\label{rel C/A}
\sum_{i+j=2k} (-1)^ir_ir_j=(-1)^kt_k,\quad k=1,\dots,n.
\eeq
Equivalently,
\eq
\label{rel C/A bis}
r_k^2=t_k+2r_{k-1}r_{k+1}-2r_{k-2}r_{k+2}+\dots,\quad k=1,\dots,n,
\eeq
where as usual we set $r_0=1$ and $r_i=0$ for $i>n$.

\begin{thm}
\label{thm basis C/A}
For $G/K=\Sp(n)/\UU(n)$, the algebras $C(\frp)^K$ and $(\twedge\frp)^K$ are both generated by $r_1,\dots,r_p$ (or more precisely, their images in the respective quotients), with relations generated by \eqref{rel C/A bis} 
with $ t_k$ replaced by $t_k(\rho) $ for $C(\frp)^K$, and by $0$ for $(\twedge\frp)^K$.

In other words,
$$ 
\begin{aligned}
C(\frp)^K &=
\frac{\bbC[r_1,\dots,r_n]}{\Bigl(\sum_{i+j=2k} (-1)^ir_ir_j=(-1)^k t_k(\rho)\ (1 \leq k \leq n)\Bigr)} \\
(\twedge\frp)^K & 
= \frac{\bbC[r_1,\dots,r_n]}{\Bigl(\sum_{i+j=2k} (-1)^ir_ir_j= 0\  (1 \leq k \leq n)\Bigr)} \\
\end{aligned}
$$
The latter algebra is isomorphic to $H^*(\LGr(\bbC^{2n}), \bbC)$.

A basis for each of the algebras is represented by the monomials
\eq
\label{basis C/A}
r_1^{\eps_1}r_2^{\eps_2}\dots r_n^{\eps_n},\quad \eps_i\in\{0,1\}.
\eeq
The filtration degree of $C(\frp)^K$ inherited from $C(\frp)$, and the gradation degree of $(\twedge\frp)^K$ inherited from $\twedge\frp$, are obtained by setting $\deg r_i=2i$ for $i=1,\dots,n$.
\end{thm}

\pf
Note that the cardinality of the set \eqref{basis C/A} is correct, $2^n$, so it is enough to show that every monomial can be written as a linear combination of the monomials in \eqref{basis C/A}.

We proceed by induction on degree. If the degree is 0, the only possible monomial is 1, and it is on the list \eqref{basis C/A}. If an arbitrary monomial contains either $r_1$ or $r_n$ with degree $\geq 2$, then we can use the relations \eqref{rel C/A bis} for $k=1$ ($r_1^2=t_1+r_2$) or for $k=n$ ($r_n^2=t_n$) to write this monomial as a combination of smaller degree monomials.

Assume now that a monomial 
\eq
\label{monos C/A}
x^d = x_1^{d_1}\dots x_n^{d_n},\quad d_i\in\bbZ_+
\eeq
has $d_k\geq 2$, for some $1<k<n$.
We identify monomials \eqref{monos C/A} with the strings of exponents $(d_1,\dots,d_n)\in\bbZ_+^n$. Let $f:[1,n]\to\bbR^+$ be a concave function taking integer values on $[1,n]\cap\bbZ$; for example, we can take $f(x)=x(n+1-x)$. We define $F:\bbZ_+^n\to\bbZ_+$ by 
\[
F(d_1,\dots,d_n)=\sum_{k=1}^n
f(k)d_k.
\]
By relations \eqref{rel C/A bis}, the monomial $ x^d$ is, up to lower degree monomials, equal to a linear combinations of monomials with exponents of the form
\[
(\dots, d_{k-i}+1,\dots d_k-2,\dots,d_{k+i}+1,\dots)=d+e_{k-i}-2e_k+e_{k+i},
\]
with $i$ a positive integer such that $k-i\geq 1$ and $k+i\leq n$. Here $e_1,\dots,e_n$ is the usual standard basis of $\bbR^n$.

We now have
\begin{multline*}
    F(d+e_{k-i}-2e_k+e_{k+i})-F(d)=\\
    f(k-i)[(d_{k-i}+1)-d_{k-i}] +f(k)[(d_k-2)-d_k]+f(k+i)[(d_{k+i)}+1)-d_{k+i}]=\\
    f(k-i)-2f(k)+f(k+i),
\end{multline*}
which is negative since $f$ is concave. So all the monomials in the expression for $d$ using the relations have values of $F$ lower than $F(d)$. 

We can repeat this procedure as long as we have some $d_k\geq 2$, $1<k<n$ (recall that we already handled the cases $d_1\geq 2$ and $d_n\geq 2$). Since the value of $F$ gets strictly smaller each time, and since these values are positive integers, the process has to stop, meaning that there are no $d_k\geq 2$, hence we have arrived at a monomial of the form \eqref{basis C/A}.

The statement about degrees follows from the fact that $r_i$ is of degree $i$ as a polynomial in the variables $x_j$, and that the map $\alpha:U(\frk)\to C(\frp)$ doubles the degree. 
\epf

\subsection{The case $G/K=\SO(2n)/\UU(n)$}
\label{subsec D/A}

Since $G$ and $K$ have equal rank, $\Proj(S)=C(\frp)^K$ and $\gr \Proj(S)=(\twedge\frp)^K$.

Since $G$ and $K$ are both connected, the Weyl groups are $W_G=W_\frg=D_n$ and $W_K=W_\frk=A_{n - 1}$, i.e., $W_K$ consists of the permutations of the variables $x_1,\dots,x_n$, while $W_G$
consists of permutations and sign changes of an even number of variables $x_1,\dots,x_n$. 
The set $W^1_{G,K}$ has $2^{n-1}$ elements and can be identified with the sign changes of an even number of variables.
The algebra $S(\frt)^{W_K}$ is generated by
\[
r_1=x_1+\dots+x_n,\quad r_2=x_1x_2+x_1x_3+\dots + x_{n-1}x_n,\ \dots,\quad  r_n=x_1x_2\dots x_n, 
\]
and $S(\frt)^{W_G}$ is generated by
\[
t_1=x_1^2+\dots+x_n^2,\ \dots,\quad t_{n-1}=x_1^2\dots x_{n-1}^2+\dots+x_2^2\dots x_n^2,\quad \bar t_n=x_1x_2\dots x_n.
\]
We set $t_n=\bar t_n^2$. The relations are the same as \eqref{rel C/A} or \eqref{rel C/A bis} (where as before, $t_k$ is replaced by $t_k(\rho)$ for $C(\frp)^K$ and by $0$ for $(\twedge\frp)^K$), 

except that for $k=n$ we have $r_n=\bar t_n$ instead of $r_n^2=t_n$. This last equation enables us to eliminate $r_n$ from the list of generators. Thus we have

\begin{thm}
\label{thm basis D/A}
For $G/K=\SO(2n)/\UU(n)$, the algebras $C(\frp)^K$ and $(\twedge\frp)^K$ are both generated by $r_1,\dots,r_{n-1}$ (or more precisely, their images in the respective quotients). The relations are generated by \eqref{rel C/A bis}  with $t_k$ replaced by $t_k(\rho)$ for $C(\frp)^K$ and by $0$ for $(\twedge\frp)^K$,
and with the last relation replaced by $r_n=\bar t_n$. 

In other words,
$$ 
\begin{aligned}
C(\frp)^K &=
\frac{\bbC[r_1,\dots,r_n]}{\Bigl(\sum_{i+j=2k} (-1)^ir_ir_j=(-1)^k t_k(\rho)\; (1 \leq k \leq n), r_n = \bar{t}_n(\rho) \Bigr)} \\
(\twedge\frp)^K &
= \frac{\bbC[r_1,\dots,r_n]}{\Bigl(\sum_{i+j=2k} (-1)^ir_ir_j= 0\; (1 \leq k \leq n), r_n = 0 \Bigr)} \\
\end{aligned}
$$
The latter algebra is isomorphic to
$H^*(\OLGr^+(\bbC^{2n}), \bbC)$. 

A basis for each of the algebras is represented by the monomials
\eq
\label{basis D/A}
r_1^{\eps_1}r_2^{\eps_2}\dots r_{n-1}^{\eps_{n-1}},\quad \eps_i\in\{0,1\}.
\eeq
The filtration degree of $C(\frp)^K$ inherited from $C(\frp)$, and the gradation degree of $(\twedge\frp)^K$ inherited from $\twedge\frp$, are obtained by setting $\deg r_i=2i$ for $i=1,\dots,n-1$.
\end{thm}
\pf
The same as the proof of Theorem \ref{thm basis C/A}.
\epf

\subsection{The group cases}
\label{subsec group}

In this subsection we consider the group cases $G\times G/\Delta G\cong G$, where $G$ is $\SO(n)$, or $\UU(n)$, or $\Sp(n)$. 

In each of the three cases, the complexified Lie algebra of $G\times G$ is $\frg\oplus\frg$ where $\frg$ is the complexified Lie algebra of $G$, $\frk$ is the diagonal subalgebra $\Delta\frg\cong\frg$ of $\frg\oplus\frg$, and $\frp$ is the antidiagonal subspace of $\frg\oplus\frg$, which is isomorphic to $\frg$ as a $\frg$-module.

Thus we are looking for the description of $C(\frg)^G=C(\frg)^\frg$ and of $(\twedge\frg)^G=(\twedge\frg)^\frg$. Thus we can use the results of \cite{K97} 
in the Clifford case and the well known Hopf-Koszul-Samelson Theorem in the exterior case \cite{cartan1951,Samelson41, koszul50} see \cite[p.{568}]{CCCvol3} for full bibliographic details.

to conclude
\begin{thm} \cite{K97}, \cite{Samelson41}
    \label{gen rel basis group}
The algebras $C(\frp)^K=C(\frg)^\frg$ 
and  $(\twedge\frp)^K=(\twedge\frg)^\frg$ are isomorphic to $C(\Cal P_\wedge(\frp))$, respectively $\twedge\Cal P_\wedge(\frp)$, where $\Cal P_\wedge(\fp)\cong\frh$ denotes the graded subspace defined in Definition \ref{def samspace}. 
The degrees are given in Table \ref{tab degrees prim}.
\end{thm}

\subsection{The cases $G/K=\UU(2n)/\Sp(n)$}
\label{subsec A/C}
In these cases $K$ is connected, so $(\twedge\frp)^K=(\twedge\frp)^\frk$. Since these are unequal rank cases, we first describe the fundamental Cartan subalgebra $\frh=\frt\oplus\fra$. The situation is similar to the case $G/K=\UU(2n)/\OO(2n)$. 

The noncompact symmetric space corresponding to $\UU(2n)/\Sp(n)$ is $\GL(2n,\bbH)/\Sp(n)=\UU^*(2n)/\Sp(n)$. The following can be read off from the information about the classification of real forms in \cite{beyond}.

The fundamental Cartan subalgebra $\frh$ can be identified with $\bbC^{2n}$, with the Cartan involution acting by
$\theta(h_1,\dots,h_{2n})=(-h_{2n},\dots,-h_1)$.
 Hence
\begin{eqnarray*}
&\frt&=\{(h_1,\dots,h_n,-h_n,\dots,-h_1)\,\big|\,h_1,\dots,h_n\in\bbC\}\cong\bbC^n;\\ 
&\fra&=\{(h_1,\dots,h_n,h_n,\dots,h_1)\,\big|\,h_1,\dots,h_n\in\bbC\}\cong\bbC^n.
\end{eqnarray*}
If we now restrict the roots $\pm(\eps_i-\eps_j)$, $1\leq i<j\leq 2n$, to $\frt$, we get 
\[
\pm(\eps_i\pm\eps_j),\ 1\leq i<j\leq n;\quad\ 2\eps_i,\ 1\leq i\leq n.
\]
In other words, we have obtained the root system $C_n$. Since $\Delta(\frk,\frt)$ is also of type $C_n$ (but with smaller multiplicities), $W^1_{\frg,\frk}$ consists only of the identity. (This means that the spin module $S$ is primary, as already observed in \cite{han}.)  

So the algebras $\Proj(S)$ and $\gr\Proj(S)$ are both equal to $\bbC\cdot 1$, and
using Theorem \ref{thm prim and aprim alg} we get 
\begin{thm}
    \label{gen rel basis A/C}
    The algebra  $(\twedge\frp)^K$ is isomorphic to $\twedge(\Cal P_\wedge(\frp))$, where $\Cal P_\wedge(\frp)\cong\fra$ is defined in  Definition \ref{def samspace} and the degrees are given in Table \ref{tab degrees prim}.    
\end{thm}

\bibliographystyle{alpha}
\bibliography{references}

\end{document}